\documentclass[10pt]{article}

\usepackage[T1]{fontenc}
\usepackage[english]{babel}

\usepackage{amsmath,amssymb,amsthm,mathtools,esint}
\usepackage{tikz,graphicx}
\usetikzlibrary{arrows.meta}
\usetikzlibrary{decorations.markings}
\usepackage{graphicx,subcaption}
\usepackage{enumerate}
\usepackage{dsfont}

\theoremstyle{plain}
\newtheorem*{theorem*}{Theorem}
\newtheorem{theorem}{Theorem}[section] 
\newtheorem{lemma}[theorem]{Lemma}
\newtheorem{proposition}[theorem]{Proposition}
\newtheorem{corollary}[theorem]{Corollary}

\newtheorem{example}[theorem]{Example}
\theoremstyle{definition}
\newtheorem{definition}[theorem]{Definition}
\newtheorem{remark}[theorem]{Remark}

\numberwithin{equation}{section}

\mathtoolsset{showonlyrefs}

\newcommand{\eps}{\varepsilon}
\newcommand{\NN}{\mathbb{N}}
\newcommand{\ZZ}{\mathbb{Z}}
\newcommand{\RR}{\mathbb{R}}
\newcommand{\CC}{\mathbb{C}}
\newcommand{\DD}{\mathbb{D}}

\newcommand{\Ordo}{\mathcal{O}}

\newcommand{\EE}{\mathbb{E}}

\DeclareMathOperator{\im}{Im}
\DeclareMathOperator{\re}{Re}

\renewcommand{\d}{\,\mathrm{d}}
\renewcommand{\i}{\mathrm{i}}
\newcommand{\e}{\mathrm{e}}

\DeclareMathOperator{\Tr}{Tr}

\title{Domino tilings of the Aztec diamond with doubly periodic weightings}

\author{Tomas Berggren\footnote{Department of Mathematics, 
	Royal Institute of Technology (KTH),
	Stockholm, Sweden. Email: tobergg@kth.se. Supported by the Swedish Research Council grant (VR) Grant no. 2016-05450 and the G\"oran Gustafsson Foundation.}}

\date{}

\begin{document}

\maketitle

\begin{abstract}
 In this paper we consider domino tilings of the Aztec diamond with doubly periodic weightings. In particular a family of models which, for any $ k \in \NN $, includes models with $ k $ smooth regions is analyzed as the size of the Aztec diamond tends to infinity. We use a non-intersecting paths formulation and give a double integral formula for the correlation kernel of the Aztec diamond of finite size. By a classical steepest descent analysis of the correlation kernel we obtain the local behavior in the smooth and rough regions as the size of the Aztec diamond tends to infinity. From the mentioned limit the macroscopic picture such as the arctic curves and in particular the number of smooth regions is deduced. Moreover we compute the limit of the height function and as a consequence we confirm, in the setting of this paper, that the limit in the rough region fulfills the complex Burgers' equation, as stated by Kenyon and Okounkov. 
\end{abstract}

\section{Introduction}\label{sec:model}

\subsection{The $ 2\times k $-periodic Aztec diamond}\label{sec:aztec_diamond}

Consider the subset of $ \ZZ^2 $
\begin{equation}
 \{(x,y)\in \ZZ^2:|x-(N-\frac{1}{2})|+|y-(N-\frac{1}{2})|\leq N\},
\end{equation}
where $ N \in \NN $ is a fixed number and let $ G $ be the graph formed by the square lattice connecting these points. Color each vertex $ (x,y) $ in $ G $ in white if $ x-N+y $ is even and in black otherwise. The graph $ G $ is bipartite and called the \emph{Aztec diamond graph} of size $ N $. A \emph{dimer covering} $ \mathcal M $ of $ G $, also called a \emph{perfect matching}, is a pairing of the vertexes in $ G $ where each pair consists of adjacent vertexes with different colors, such pairing, or edge, is called a \emph{dimer}. Set a weight on each edge in the graph and let the weight of a dimer covering $ w(\mathcal M) $ be the product of the edge-weights corresponding to each dimer in $ \mathcal M $. A natural probability measure on the space of all dimer coverings on $ G $ is given by
\begin{equation}\label{eq:sec_model:probability_dimer}
 \mathbb P (\mathcal M) = \frac{w(\mathcal M)}{\sum_{\mathcal M'}w(\mathcal M')}, \quad w(\mathcal M) = \prod_{\text{dimer} \in \mathcal M} w(\text{dimer}),
\end{equation}
where the sum is over all possible dimer coverings of $ G $.

The \emph{Aztec diamond} is the union of the faces of the dual graph of the Aztec diamond graph. Color each face according to the bipartite graph. A \emph{domino} consists of two adjacent faces and a \emph{domino tiling} of the Aztec diamond is a cover of the Aztec diamond with dominoes such that no two dominoes intersect. There is a simple bijection between the domino tilings of the Aztec diamond and dimer coverings of the Aztec diamond graph. Namely, a dimer connecting $ (x,y) $ and $ (x',y') $ becomes the domino consisting of the two faces with midpoints at $ (x,y) $ and $ (x',y') $, and vice versa (Figure \ref{fig:bijection_dimer_tiling}). This induces a probability measure on the space of domino tilings of the Aztec diamond from the probability measure defined above on the dimer model.

\begin{figure}[t]
	\begin{center}
		\begin{tikzpicture}[scale=.6]

		\foreach \x in {0,1,2,3}
		{\draw (-.5-\x,\x-2.5)--(.5+\x,\x-2.5);
			\draw (\x-3.5,\x+1.5)--(3.5-\x,\x+1.5);
			\draw (\x-3.5,.5-\x)--(\x-3.5,1.5+\x);
			\draw (\x+.5,\x-2.5)--(\x+.5,4.5-\x);
		}
		
		\foreach \x/\y in {-4/0,-3/1,-2/2,-3/-1,-1/-1,0/-2}
		{ 
			\draw[line width = 1mm] (\x+.5,\y+.5)--(\x+.5,\y+1.5);
			\draw (\x+.5,\y+.5) node[circle,draw=black,fill=black,inner sep=2pt]{};
			\draw (\x+.5,\y+1.5) node[circle,draw=black,fill=white,inner sep=2pt]{};
		}
		
		\foreach \x/\y in {-2/-1,1/-2,2/-1,3/0,2/1,1/2}
		{
			\draw[line width = 1mm] (\x+.5,\y+.5)--(\x+.5,\y+1.5);
			\draw (\x+.5,\y+.5) node[circle,draw=black,fill=white,inner sep=2pt]{};
			\draw (\x+.5,\y+1.5) node[circle,draw=black,fill=black,inner sep=2pt]{};
		}
		
		\foreach \x/\y in {-1/-3,-2/-2,-1/3,0/0}
		{
			\draw[line width = 1mm] (\x+.5,\y+.5)--(\x+1.5,\y+0.5);
			\draw (\x+.5,\y+.5) node[circle,draw=black,fill=black,inner sep=2pt]{};
			\draw (\x+1.5,\y+.5) node[circle,draw=black,fill=white,inner sep=2pt]{};
		}
		
		\foreach \x/\y in {-2/1,0/1,-1/2,-1/4}
		{
			\draw[line width = 1mm] (\x+.5,\y+.5)--(\x+1.5,\y+0.5);
			\draw (\x+.5,\y+.5) node[circle,draw=black,fill=white,inner sep=2pt]{};
			\draw (\x+1.5,\y+.5) node[circle,draw=black,fill=black,inner sep=2pt]{};
		}
		
		\end{tikzpicture}
		\quad
		\begin{tikzpicture}[scale=0.55]
		
		\foreach \x/\y in {-4/0,-3/1,-2/2,-3/-1,-1/-1,0/-2}
		{ \fill[outer color=lightgray,inner color=gray]
			(\x,\y) rectangle (\x+1,\y+1); 
			\draw [line width = 1mm] (\x,\y) rectangle (\x+1,\y+2);
		}
		
		\foreach \x/\y in {-2/-1,1/-2,2/-1,3/0,2/1,1/2}
		{  \fill[outer color=lightgray,inner color=gray]
			(\x,\y+1) rectangle (\x+1,\y+2);
			\draw [line width = 1mm] (\x,\y) rectangle (\x+1,\y+2);
		}
		
		\foreach \x/\y in {-1/-3,-2/-2,-1/3,0/0}
		{  \fill[outer color=lightgray,inner color=gray]
			(\x,\y) rectangle (\x+1,\y+1);
			\draw [line width = 1mm] (\x,\y) rectangle (\x+2,\y+1);
		}
		
		\foreach \x/\y in {-2/1,0/1,-1/2,-1/4}
		{  \fill[outer color=lightgray,inner color=gray]
			(\x+1,\y) rectangle (\x+2,\y+1);
			\draw [line width=1mm] (\x,\y) rectangle (\x+2,\y+1);
		}
		\end{tikzpicture}
		\caption{A sample of a dimer covering of the Aztec diamond graph and corresponding domino tiling of the Aztec diamond. \label{fig:bijection_dimer_tiling}} 
	\end{center}
\end{figure}
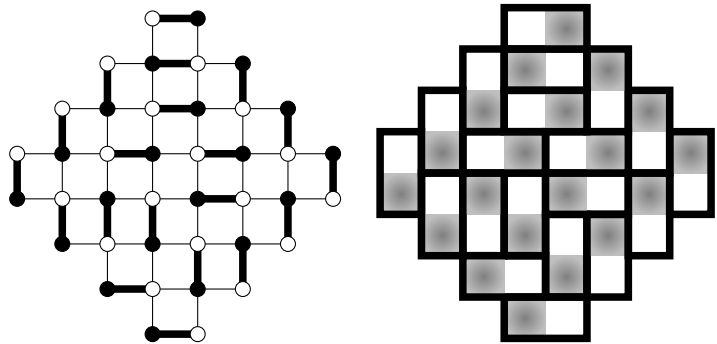

In the present paper we consider a family of \emph{doubly periodic weightings}, the edge weights are periodic in two linearly independent directions, of the Aztec diamond, which was introduced in \cite{DS14}. The model is defined as follows. Set a weight on each face of the Aztec diamond graph, the weight on the edges are defined as the product of its two adjacent faces. Fix a $ k \in \ZZ $. Let the weight on the face with down left corner at $ (i,j) $ (Figure \ref{fig:weights}) be $ a_{i,j} $ with the conditions 
 \begin{equation}
  a_{i,j} = a_{i+2,j-2} \text{ and } a_{i,j} = a_{i+k,j+k}.
 \end{equation}
 We refer to the model defined by the above doubly periodic weighting of the Aztec diamond as the \emph{$ 2 \times k $-periodic Aztec diamond}. In \cite{DS14} the model is referred to as the \emph{$ T $-system with $ k $-toroidal initial data} (we have rotate the diamond by $ \frac{\pi}{2} $ counter clockwise compared with the notation in \cite{DS14}). For simplicity we take the size of the Aztec diamond as $ kN $ with $ N $ even.

The weights given above are not unique in the sense that there are different weightings defining the same probability measure. For instance, let $ b_i $ for $ i=1,\dots,4 $, be the weights on four edges having a common vertex. Taking $ bb_i $ for $ i=1,\dots,4 $ and $ b>0 $ as new weights does not change the probability measure (Figure \ref{fig:gauge_transformation}). This change of weights is called a \emph{gauge transformation}. That a gauge transformation does not change the probability measure is clear since exactly one of the four edges intersecting a common vertex is in each dimer covering, so it gives a factor $ b $ extra to each weight of a dimer covering. 
\begin{figure}[ht]
\begin{center}
 \begin{tikzpicture}[scale=.9]
  {\draw (0,0) node[circle,fill,inner sep=2pt]{};
   \draw (-1,0)--(1,0);
   \draw (0,-1)--(0,1);
   \draw (0,.7) node[left]{$b_1$};
   \draw (0,-.7) node[left]{$b_3$};
   \draw (-.7,0) node[above]{$b_4$};
   \draw (.7,0) node[above]{$b_2$};
   }
   \draw (1.5,0) node{$ \sim$};
 \end{tikzpicture}
 \begin{tikzpicture}[scale=.9]
  {\draw (0,0) node[circle,fill,inner sep=2pt]{};
   \draw (-1,0)--(1,0);
   \draw (0,-1)--(0,1);
   \draw (0,.7) node[left]{$bb_1$};
   \draw (0,-.7) node[left]{$bb_3$};
   \draw (-.7,0) node[above]{$bb_4$};
   \draw (.7,0) node[above]{$bb_2$};
   }
 \end{tikzpicture}
\end{center}
\caption{A gauge transformation. \label{fig:gauge_transformation}}
\end{figure}
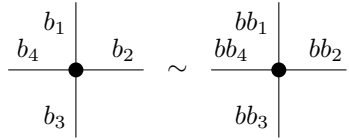

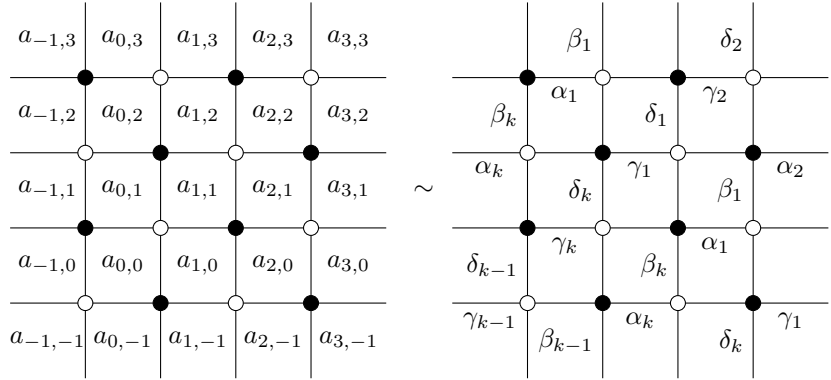
\begin{figure}[t]
\begin{center}
 \begin{tikzpicture}
 \foreach \x in {0,...,3}
 {\draw (-1,\x)--(4,\x);
   \draw (\x,-1)--(\x,4);
 }
 \foreach \x in {-1,0,...,3}
 {
 \foreach \y in {-1,0,...,3}
 {\draw (\x+.5,\y+.5) node {$a_{\x,\y}$};
 }
 }
 \foreach \x in {0,2}
 \foreach \y in {0,2}
 {\draw (\x ,\y) node[circle,draw=black,fill=white,inner sep=2pt]{};
 }
 \foreach \x in {1,3}
 \foreach \y in {1,3}
 {\draw (\x ,\y) node[circle,draw=black,fill=white,inner sep=2pt]{};
 }
 \foreach \x in {0,2}
 \foreach \y in {1,3}
 {\draw (\x ,\y) node[circle,draw=black,fill=black,inner sep=2pt]{};
 }
 \foreach \x in {1,3}
 \foreach \y in {0,2}
 {\draw (\x ,\y) node[circle,draw=black,fill=black,inner sep=2pt]{};
 }
 \draw (4.5,1.5) node {$ \sim $};
 \end{tikzpicture}
 \begin{tikzpicture}
 \foreach \x in {0,...,3}
 {\draw (-1,\x)--(4,\x);
   \draw (\x,-1)--(\x,4);
 }
 \foreach \x in {0,2}
 \foreach \y in {0,2}
 {\draw (\x ,\y) node[circle,draw=black,fill=white,inner sep=2pt]{};
 }
 \foreach \x in {1,3}
 \foreach \y in {1,3}
 {\draw (\x ,\y) node[circle,draw=black,fill=white,inner sep=2pt]{};
 }
 \foreach \x in {0,2}
 \foreach \y in {1,3}
 {\draw (\x ,\y) node[circle,draw=black,fill=black,inner sep=2pt]{};
 }
 \foreach \x in {1,3}
 \foreach \y in {0,2}
 {\draw (\x ,\y) node[circle,draw=black,fill=black,inner sep=2pt]{};
 }
 {\draw (-.5,0) node[below] {$\gamma_{k-1}$};
 \draw (0,.5) node[left] {$\delta_{k-1}$};
 \draw (1,-.5) node[left] {$\beta_{k-1}$};
 
 \draw (-.5,2) node[below] {$\alpha_k$};
 \draw (0,2.5) node[left] {$\beta_k$};
  \draw (.5,1) node[below] {$\gamma_k$};
 \draw (1,1.5) node[left] {$\delta_k$};
  \draw (1.5,0) node[below] {$\alpha_k$};
 \draw (2,.5) node[left] {$\beta_k$};
 \draw (3,-.5) node[left] {$\delta_k$};
 
 \draw (.5,3) node[below] {$\alpha_1$};
 \draw (1,3.5) node[left] {$\beta_1$};
  \draw (1.5,2) node[below] {$\gamma_1$};
 \draw (2,2.5) node[left] {$\delta_1$};
  \draw (2.5,1) node[below] {$\alpha_1$};
 \draw (3,1.5) node[left] {$\beta_1$};
  \draw (3.5,0) node[below] {$\gamma_1$};
  
 \draw (2.5,3) node[below] {$\gamma_2$};
 \draw (3,3.5) node[left] {$\delta_2$};
  \draw (3.5,2) node[below] {$\alpha_2$};
 }
 \end{tikzpicture}
\end{center}
\caption{The two weightings are equivalent in the sense that they define the same probability measure. The weights $ a_{i,j} $ are weights on the faces while $ \alpha_i $ and $ \gamma_i $ are weights on the edges above, and $ \beta_i $ and $ \delta_i $ are weights on the edges to the right. Moreover $ \alpha_i\gamma_i=1=\beta_i\delta_i $.\label{fig:weights}}
\end{figure}

We use a gauge transformation to simplify the weighting. Multiply each edge connected to the vertex $ (i,j) $ with $ a_{i-1,j}^{-1} $ if $ (i,j) $ is a black vertex and with $ a_{i,j-1}^{-1} $ if $ (i,j) $ is a white vertex. We obtain new weights on the edges, as indicated in Figure \ref{fig:weights}, $ \alpha_1,\dots,\alpha_k $, $\beta_1,\dots,\beta_k $, where $ \alpha_1 \cdots\alpha_k=\beta_1\cdots\beta_k $. The exact relation between $ a_{i,j} $ and $ \alpha_\ell, \beta_\ell $ is not important to us, but for book keeping the relation is given in the appendix. See Figure \ref{fig:2x3} for a sample of the $ 2\times k $-periodic Aztec diamond.

\begin{figure}[t]
 \begin{center}
  \includegraphics[scale=.45]{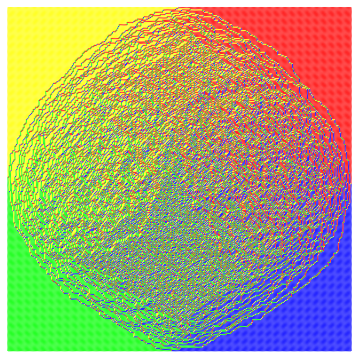}
  \hspace{.4cm}
  \includegraphics[scale=.45]{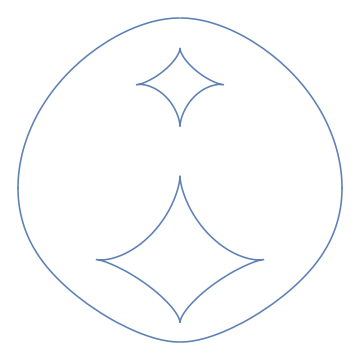}
 \end{center}
 \caption{An example of the $ 2\times 3$-periodic Aztec diamond with $ \alpha_1^{-1}=\alpha_2 = 0.3 $, $ \alpha_3=\beta_1=\beta_2=\beta_3=1 $. \label{fig:2x3}}
\end{figure}

In the present paper we use a non-intersecting paths perspective. The non-intersecting paths form a determinantal point process and thus allow us to draw conclusions about the model by analyzing the correlation kernel. This connection between domino tilings of the Aztec diamond and non-intersecting paths has been used many times before. See \cite{J02} for a careful description of this relation. For completeness we include a description here. 

Draw lines on the dominoes according to
\begin{center}
	\tikz[scale=.6]{ \fill[outer color=lightgray,inner color=gray](0,0) rectangle (1,1);
		\draw [line width = 1mm] (0,0) rectangle (1,2);
		\draw[very thick,red] (0,.5)--(1,1.5);}, \quad
	\tikz[scale=.6]{\fill[outer color=lightgray,inner color=gray](0,0) rectangle (1,1);
		\draw [line width = 1mm] (0,0) rectangle (2,1);
		\draw[very thick,red] (0,.5)--(2,.5);},
	\quad 
	\tikz[scale=.6]{\fill[outer color=lightgray,inner color=gray](0,1) rectangle (1,2);
		\draw [line width = 1mm] (0,0) rectangle (1,2);
		\draw[very thick,red] (0,1.5)--(1,0.5);}
	\quad and \quad
	\tikz[scale=.6]{\fill[outer color=lightgray,inner color=gray](1,0) rectangle (2,1);
		\draw [line width = 1mm] (0,0) rectangle (2,1);}.
\end{center}
These paths are called $ DR $-paths of type I, see e.g. \cite{J02}. Rotate the Aztec diamond clockwise by $ \frac{\pi}{4} $ and add an horizontal line of length one to each path at each integer step, that is, separate each path at the step $ i=1,\dots, kN, $ into the part of the path $ \leq i $ and the part $ > i $ and add an horizontal line of length one in between these parts. This procedure defines a bijective map between domino tilings of the Aztec diamond of size $kN $ and non-intersecting paths on the directed graph in Figure \ref{fig:tiling_paths} with start and end points at $ (0,j-1) $ for $ j=2n-kN+1,\ldots,2n $ respectively $ (2i-1,2n-kN) $ for $ i=1,\dots,kN $, where $ n\geq \frac{kN}{2} $ will be specified later, these paths corresponds to the red part in Figure \ref{fig:tiling_paths}. Under this bijection a probability measure on the domino tilings defines a probability measure on the non-intersecting paths and vice versa.

Continue the above constructed paths and add more paths to obtain $ 2n $ non-intersecting paths on the directed graph in Figure \ref{fig:tiling_paths} with start and end points at $ (0,u_0^j)=(0,j-1) $ respectively $ (2kN,u_{2kN}^j)=(2kN,-kN+j-1) $ for $ j=1,\ldots,2n $. Conversely, for any $ 2n $ non intersecting paths on the graph in Figure \ref{fig:tiling_paths} and with start and end points as above, the paths has to intersect the points $ (2i-1,2n-kN) $ for $ i=1,\dots,kN $. This means that the top part, the part above the points $ (2i-1,2n-kN) $ for $ i=1,\dots,kN $, of the $ 2n $ non-intersecting paths is independent of the rest of the non-intersecting paths. Hence, a probability measure on the domino tilings defines a probability measure on the $ 2n $ non-intersecting paths and vice versa. Finally these non intersecting paths have a one to one correspondence with $ 2kN \times 2 n $ points $ \{(m,u)\}_{m,u} \subset \{0,\dots,2kN\} \times \ZZ $, by considering the intersections of the non-intersecting paths and the vertical lines crossing the coordinate axes at an integer. If a paths intersect such line in multiple integer points take the first intersection from above.

By the above described procedure the $ 2 \times k $-periodic Aztec diamond form a point process defined by the probability measure given below.

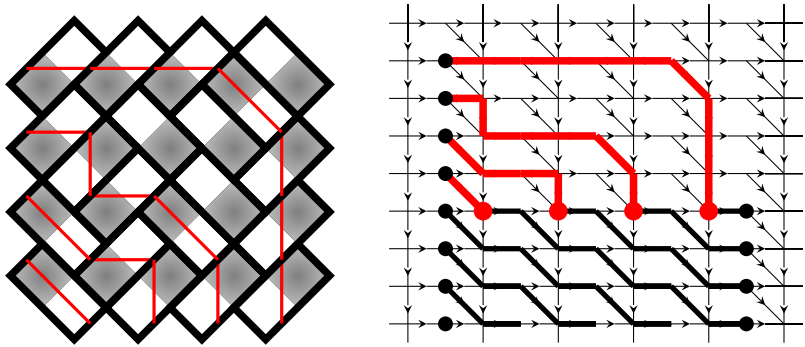
\begin{figure}[t]
	\begin{center}
		\begin{tikzpicture}[rotate=-45, scale=0.6]
		
		\foreach \x/\y in {-4/0,-3/1,-2/2,-3/-1,-1/-1,0/-2}
		{ \fill[outer color=lightgray,inner color=gray]
			(\x,\y) rectangle (\x+1,\y+1); 
			\draw [line width = 1mm] (\x,\y) rectangle (\x+1,\y+2);
			\draw[very thick,red] (\x,\y+.5)--(\x+1,\y+1.5);
		}
		
		\foreach \x/\y in {-2/-1,1/-2,2/-1,3/0,2/1,1/2}
		{  \fill[outer color=lightgray,inner color=gray]
			(\x,\y+1) rectangle (\x+1,\y+2);
			\draw [line width = 1mm] (\x,\y) rectangle (\x+1,\y+2);
			\draw[very thick,red] (\x,\y+1.5)--(\x+1,\y+.5);
		}
		
		\foreach \x/\y in {-1/-3,-2/-2,-1/3,0/0}
		{  \fill[outer color=lightgray,inner color=gray]
			(\x,\y) rectangle (\x+1,\y+1);
			\draw [line width = 1mm] (\x,\y) rectangle (\x+2,\y+1);
			\draw[very thick,red] (\x,\y+.5)--(\x+2,\y+.5);
		}
		
		\foreach \x/\y in {-2/1,0/1,-1/2,-1/4}
		{  \fill[outer color=lightgray,inner color=gray]
			(\x+1,\y) rectangle (\x+2,\y+1);
			\draw [line width=1mm] (\x,\y) rectangle (\x+2,\y+1);
		}
		\end{tikzpicture}
		\hspace{.5cm}
		\begin{tikzpicture}[scale=0.5]
		\tikzset{->-/.style={decoration={
					markings,
					mark=at position .5 with {\arrow{stealth}}},postaction={decorate}}}
		\foreach \y in {-4,-3,-2,-1,0,1,2,3,4}
		{
			\foreach \x in {0,1,2,3,4,5,6,7,8,9}
			{\draw[->,>=stealth] (\x-.5,\y)--(\x+.5,\y);
				\draw (9.5,\y)--(10.5,\y);
		}}
		\foreach \x in {0,2,4,6,8,10}
		{\foreach \y in {-4,-3,-2,-1,0,1,2,3}
			{\draw[-<,>=stealth] (\x,\y-.5)--(\x,\y+.5);
				\draw (\x,3.5)--(\x,4.5);}}
		\foreach \y in {-3,-2,-1,0,1,2,3,4}
		{\foreach \x in {0,2,4,6,8}
			\draw [->-](\x+1,\y)to(\x+2,\y-1);}
		\foreach \x/\y in {1/3,3/3,5/3,1/2,3/1,3/0}
		{
			\draw[line width = 1mm,red] (\x,\y)--(\x+1,\y);
		}
		
		\foreach \x/\y in {8/2,8/1,8/0,2/2,6/0,4/0}
		{
			\draw[line width = 1mm,red] (\x,\y)--(\x,\y-1);
		}
		
		\foreach \x/\y in {7/3,5/1,1/1,1/0}
		{
			\draw[line width = 1mm,red] (\x,\y)--(\x+1,\y-1);
		}
		
		\foreach \x/\y in {2/3,4/3,6/3,2/1,4/1,2/0}
		{
			\draw[line width = 1mm,red] (\x,\y)--(\x+1,\y);
		}
		\foreach \y in {-3,-2,-1}
		{\foreach \x in {2,4,6,8}
		{\draw[line width = .7mm] (\x-1,\y)--(\x,\y-1);
		\draw[line width = .7mm] (\x,\y)--(\x+1,\y);
		\draw[line width = .7mm] (\x,-3)--(\x+1,-3);
		\draw[line width = .7mm] (\x,-4)--(\x+1,-4);}
		}
		\foreach \y in {-4,-3,-2,-1,0,1,2,3}
		{\draw (1,\y) node[circle,fill,inner sep=2pt]{};}
		\foreach \y in {-4,-3,-2,-1}
		{\draw (9,\y) node[circle,fill,inner sep=2pt]{};}
		\foreach \x in {2,4,6,8}
		{\draw (\x,-1) node[circle,fill=red,inner sep=2.5pt]{};}
		\end{tikzpicture}
		\caption{A sample of a domino tiling of the Aztec diamond and the corresponding non-intersecting paths. The sample is the same as in Figure \ref{fig:bijection_dimer_tiling}. \label{fig:tiling_paths}}
	\end{center}
\end{figure}

The point process mentioned above is the probability measure
\begin{equation}\label{eq:sec_model:measure}
 \frac{1}{Z_{n,N}}\prod_{m=1}^{2kN} \det \left( T_{\phi_m}(u_{m-1}^j,u_{m}^k)\right)_{j,k=1}^{2n},
\end{equation} 
defined on $ \{(m,u)\}_{m,u} \subset \{0,\dots,2kN\} \times \ZZ $ where 
\begin{equation}
 u_0^j=j-1 \text{ and } u_{2kN}^j=-kN+j-1 \text{ for } j=1,\ldots,2n,
\end{equation} 
 see \cite[Proposition 5.1]{BD19}. Here the \emph{transition matrix} $ T_{\phi_m} $ is a block Toeplitz matrix given by
\begin{equation}
 T_{\phi_m}(2x+r,2y+s)= \left(\hat \phi_m(y-x)\right)_{r+1,s+1},
 \quad r,s=0,1, \quad x,y \in \ZZ,
\end{equation}
with symbols
\begin{equation}
\phi_{2m-1}(z)  =  
\begin{pmatrix}
1 & \alpha_m z^{-1}\\
\alpha_m^{-1} & 1
\end{pmatrix},
\quad
\phi_{2m}(z) = \frac{1}{1-z^{-1}} 
\begin{pmatrix}
1 & \beta_m z^{-1}\\
\beta_m^{-1} & 1
\end{pmatrix}, 
\end{equation}
for $ m = 1,2,\dots,kN $. To each weighting we associate a matrix-valued function defined as the product over one period of the symbols of the transition matrices,
\begin{equation}\label{eq:sec_model:weight}
\Phi(z) = \frac{1}{(1-z^{-1})^k}\prod_{m=1}^k
\begin{pmatrix}
1 & \alpha_m z^{-1}\\
\alpha_m^{-1} & 1
\end{pmatrix}
\begin{pmatrix}
1 & \beta_mz^{-1}\\
\beta_m^{-1} & 1
\end{pmatrix}.
\end{equation}

The measure \eqref{eq:sec_model:measure} is a determinantal point process, \cite{DK17, J06}, that is,
\begin{multline}\label{eq:sec_model:determinantal}
  \mathbb P \left( \text{ points at } (m_1,2\xi_1+i_1), \ldots, (m_\ell,2\xi_d+i_d)\right) \\
  = \det \left(K(m_j,2\xi_j+i_j;m_\ell,2\xi_\ell+i_\ell)\right)_{j,\ell=1}^d,
\end{multline}
where $ u=2(\xi+n)+i $ for $ i=0,1 $, and the function $ K $ is the \emph{correlation kernel}. 
 
\subsection{The correlation kernel}\label{sec:correlation_kernel}

Our first result is an explicit double integral formula of the correlation kernel in \eqref{eq:sec_model:determinantal}. The explicit double integral formula of the correlation kernel is the starting point of the asymptotic analysis of the $ 2\times k $-periodic Aztec diamond. We prove the double integral formula for a slightly more general weighting than the one discussed above.

 Let $ \alpha_m,\beta_m,\gamma_m>0 $ and consider the probability measure \eqref{eq:sec_model:measure} with 
\begin{equation}
\phi_{2m-1}(z)  =  
\begin{pmatrix}
\gamma_m & \alpha_m z^{-1}\\
\alpha_m^{-1} & \gamma_m^{-1}
\end{pmatrix},
\quad
\phi_{2m}(z) = \frac{1}{1-z^{-1}} 
\begin{pmatrix}
1 & \beta_m z^{-1}\\
\beta_m^{-1} & 1
\end{pmatrix}. 
\end{equation}
The associated matrix-valued function is given by
\begin{equation}\label{eq:sec_model:generalized_weight}
\Phi(z) = \frac{1}{(1-z^{-1})^k}\prod_{m=1}^k
\begin{pmatrix}
\gamma_m & \alpha_m z^{-1}\\
\alpha_m^{-1} & \gamma_m^{-1}
\end{pmatrix}
\begin{pmatrix}
1 & \beta_mz^{-1}\\
\beta_m^{-1} & 1
\end{pmatrix}.
\end{equation}
Let $ \rho_1(z) $ and $ \rho_2(z) $ be the eigenvalues of $ \Phi(z) $ determined for $ z\in \CC \backslash (-\infty,0] $ by the inequality $ |\rho_2(z)|<|\rho_1(z)| $ and by continuity from the upper half plane for $ z\in (-\infty,0] $. For some particular parts of $ (-\infty,0] $ the strict inequality still holds, while on other parts the inequality becomes an equality, see Section \ref{sec:properties_of_eigenvalues} and in particular Lemma \ref{lem:eigenvalues}. Let $ E $ be such that
 \begin{equation}
 \Phi(z) = E(z)
 \begin{pmatrix}
 \rho_1(z) & 0 \\
 0 & \rho_2(z)
 \end{pmatrix}
 E(z)^{-1}.
 \end{equation}
 Such $ E $ exists for all but a finite set of points and such diagonalization of $ \Phi $ is only necessary for our purposes away from this finite set of points.

If $ \gamma_m=1 $ and $ \alpha_1\cdots\alpha_k=\beta_1\cdots\beta_k $ the $ 2\times k $-periodic Aztec diamond is recovered. Note that the above still defines a periodic weighting of the Aztec diamond. These weightings are included here because they give rise to some new interesting pictures that can not be obtained in the $ 2\times k $-periodic Aztec diamond. For instance tilings with cusps in the boundary between the frozen region and rough region, see Figure \ref{fig:not_included_models}. We hope to come back with the asymptotic analysis of this generalized model in future work. Our first main result is the following theorem. The proof is given in Section \ref{sec:proof_main_theorem}.

\begin{figure}[t]
 \begin{center}
  \includegraphics[scale=.45]{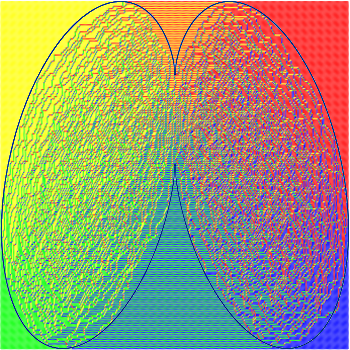}
  \hspace{.5cm}
  \includegraphics[scale=.45]{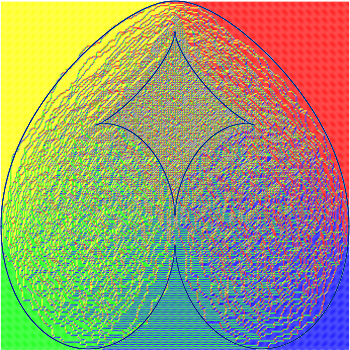}
 \end{center}
 \caption{Two examples of periodic weighting of the Aztec diamond which is covered by Theorem \ref{thm:finite_kernel} but not studied in the rest of the paper. \label{fig:not_included_models}}
\end{figure}

\begin{theorem}\label{thm:finite_kernel}
Consider the Aztec diamond with a doubly periodic weighting with associated matrix-valued function \eqref{eq:sec_model:generalized_weight} of size $ kN\times kN $ defined above with $ N $ even. The correlation kernel of the underlying determinantal point process \eqref{eq:sec_model:determinantal} is given by 
 \begin{multline}
  \left[K(2km,2\xi+i;2km',2\xi'+j)\right]_{i,j=0}^1 \\
  = -\frac{\mathds{1}_{m>m'}}{2\pi\i}\oint_{\gamma_{0,1}} \Phi(z)^{m-m'}z^{\xi'-\xi}\frac{\d z}{z} \\
  + \frac{1}{(2\pi\i)^2}\oint_{\gamma_1}\oint_{\gamma_{0,1}} \frac{w^{\xi'}}{z^{\xi+1}}\frac{(1-z^{-1})^{\frac{kN}{2}}}{(1-w^{-1})^{\frac{kN}{2}}}\rho_1(w)^{\frac{N}{2}-m'} \\
  \times E(w)
  \begin{pmatrix}
   1 & 0 \\
   0 & 0
  \end{pmatrix}
  E(w)^{-1}\Phi(z)^{m-\frac{N}{2}}\frac{\d z\d w}{z-w}, \\
  -\frac{kN}{2}\leq \xi,\xi' \leq -1, \quad 0<m,m'<N, \quad \xi,\xi',m,m'\in \ZZ.
 \end{multline}
 Here $ \gamma_1 $ is a simple closed curve around $ 1 $ and with $ 0 $ in its exterior and $ \gamma_{0,1} $ a simple closed curve around $ 0 $ and $ \gamma_1 $, both intersects the real line exactly at two points.
\end{theorem}
Here and in the rest of the paper $ \mathds{1}_B $ is the indicator function of the set $ B $.

\subsection{Previous results}

The Aztec diamond was first introduced in \cite{EKLP92} where the number of possible tilings of the Aztec diamond of size $ N $ was computed. There are plenty results in the literature dealing with asymptotic results as $ N \to \infty $ of the uniform measure or the measure for which there is a bias of either the horizontal or the vertical dominoes. One of the first such result is the famous arctic circle theorem, which shows that the frozen and the unfrozen region (see below) are separated by a circle when $ N\to \infty $ \cite{JPS95}. Local properties has also been studied, such as placement probabilities, the probability of finding a specific domino on a given place \cite{CEP96}, that the so called zig-zag paths converges to the discrete sine process in the rough region \cite{J02} (compare with Example \ref{ex:sine_kernel}), that the paths separating the frozen and rough region converges to the Airy process minus a parabola \cite{J05}, at the turning points, the points where the rough region touch the boundary, there is a limit to the GUE minor process \cite{JN06}. Also, the fluctuation of the height function in the rough region tends to the Gaussian free field \cite{BG16, CJY15}, an alternating approach would be \cite{D18}.

It is known \cite{KOS06, NHB84} that for a planer bipartite graph with a doubly periodic weighting three different type of macroscopic regions may appear as $ N\to \infty $ . The \emph{frozen} region, also known as the solid region, the unfrozen region is divided into the \emph{rough unfrozen} region and the \emph{smooth unfrozen} region, or simply the rough region and smooth region, these are also known as the liquid respectively gas regions. These regions are characterized by the correlations between dimers as the distance increases \cite{KOS06}. Namely, the correlation does not decay, decay polynomially and decay exponentially for the frozen region, the rough region respectively the smooth region. The smooth region is only expected when the weighting is doubly periodic.

There are very few results on dimer models on a planer bipartite graph with a doubly periodic weighting. One of the reasons is that many of the above results concerning the Aztec diamond with uniform weighting or with a possible bias on one type of diamond can be obtain using that the model falls into the Schur process class \cite{OR03}. This is however not the case for doubly periodic weightings of the Aztec diamond. In \cite{DS14,KO07} the limit shape of planer bipartite graphs with a doubly periodic weighting are studied. In \cite{KO07} a variational problem for the limit shape is reduced to the complex Burgers' equation. As an example the \emph{arctic curves}, the limiting curves separating the different regions, for the two-periodic Aztec diamond is computed. In \cite{DS14} a generating function for the density function is obtained for the $ 2\times k $-periodic Aztec diamond in terms of a $ 4k\times 4k $ system of linear equations. In certain cases the solution is used to analyze the arctic curves. In \cite{KOS06} the authors compute the local correlations of a bipartite graph on the torus, where there are no boundary effects, with a doubly periodic weighting as the size of the mesh tends to zero.

However, before the present paper, the only doubly periodic planar dimer model for which the limit on the microscopic scale has been computed is, to the best of our knowledge, the two-periodic Aztec diamond. This was done in a sequence of papers \cite{BCJ18,CJ16,CY14} using the Kasteleyn approach and later in \cite{DK17} using a non-intersecting paths approach.

More precisely, in \cite{CY14} the authors invert the Kasteleyn matrix in terms of a generating function. In \cite{CJ16} the expression for the generating function is simplified. With this simplified formula the limit of the inverse Kasteleyn matrix is obtained. All three types of regions, frozen, rough and smooth, is present in the limit and the limit of the inverse Kasteleyn matrix is derived in all three. A more detailed analysis is done on the boundary between the frozen region and rough region respectively rough region and smooth region. Due to technical reason the limit is obtained along the main diagonal. In \cite{BCJ18} the investigation of the rough-smooth boundary is continued.

In \cite{DK17} the two-periodic Aztec diamond is expressed as a non-intersecting path model and the correlation kernel is given by a double integral formula. The integrand is given in terms of matrix-valued orthogonal polynomials and by solving a Riemann-Hilbert problem an exact expression of the correlation kernel of the Aztec diamond of finite size is given. A classical steepest descent analysis is then used to obtain the local correlation kernel inside the smooth region and at the cusp of the smooth region.

In fact, a double integral representation of the correlation kernel for a non-intersecting path model with a doubly periodic weighting, or more precisely, for a determinantal point process defined by a product of block Toeplitz determinants with matrix-valued symbols, is given in \cite{DK17}. The integrand is expressed in terms of a matrix-valued orthogonal polynomial. The block structure, and hence the fact that the polynomials are matrix-valued, reflects the doubly periodicity in the model. In a restricted, but still very broad setting, the asymptotic analysis as the degree of the polynomials tends to infinity is reduced to a Wiener-Hopf type factorization of a matrix-valued function \cite{BD19}. The reduction to a Wiener-Hopf type factorization is done using a Riemann-Hilbert problem approach. Moreover a recursive method, inspired by \cite{LP12}, is proposed in order to solve the matrix Wiener-Hopf factorization. For the $ 2 \times k $-periodic Aztec diamond the recursive method leads to a closed formula for the matrix Wiener-Hopf factorization, which is one of the main results of the present paper, see the proof of Theorem \ref{thm:finite_kernel}. There are also resent work \cite{CDKL19} using the formula of the correlation kernel given in \cite{DK17} which does not fall into the class covered in \cite{BD19}.

\subsection{An overview of the paper}

The contribution of the present paper is the analysis of the $ 2\times k $-periodic Aztec diamond. In fact, the $ 2\times k $-periodic Aztec diamond is the only model with multiple smooth regions that has been analyzed. In particular the picture indicated in \cite{DS14} is made exact. We also compute the limit of the expected height function, and the local correlation kernels, both in the smooth and rough regions.

The approach to the analysis of the $ 2\times k $-periodic Aztec diamond is much in the same spirit as the analysis of the Aztec diamond in \cite{J05}, in the sense that we take a non-intersecting paths perspective, use a Wiener-Hopf type factorization of the symbol of a block Toeplitz matrix to obtain a double integral formula of the correlation kernel and then perform a steepest descent analysis to obtain global and local results. However the block structure in the Toeplitz matrix, which reflects the doubly periodicity in the weighting and is therefore not present in \cite{J05}, is essential and makes the problem significantly more complicated. In fact, before \cite{BD19} it was far from obvious how to go through with the approach mentioned above. But also with the kernel in Theorem \ref{thm:finite_kernel} at our disposal the analysis is not straight forward, since the underlying Riemann surface has genus bigger than zero.

The paper is organized as follows. In Section \ref{sec:global_local_properties} the results of the $ 2\times k $-periodic Aztec diamond are presented. First of all the different global regions are defined. This is done using a function defined on a Riemann surface defined basically by the square root of the discriminant of $ \Phi $. The Riemann surface is compared with the spectral curve given in \cite{KOS06}. In Section \ref{sec:global_picture} the geometry of the arctic curves is explained which implies the number of smooth components, Theorem \ref{thm:number_of_smooth_components}. In Section \ref{sec:height_function_dimer} the global shape is obtained by the limit of the expectation of the height function. Moreover the limit in the rough region is compared with the first result in \cite{KO07}. Finally in Section \ref{sec:local_picture} the local correlations in the rough and smooth regions are given in terms of the limit of the correlation kernel.

The rest of the paper is devoted to the proofs of the results in Sections \ref{sec:model} and \ref{sec:global_local_properties}. In particular the proof of the local correlation kernels, Theorems \ref{thm:local_smooth} and \ref{thm:local_rough}, are given in Section \ref{sec:proof_local_correlations}. The core of the proof is that we have an explicit double integral formula for the correlation kernel, which allow us to do a steepest descent analysis and obtain the correlation kernel on the microscopic scale in the smooth and rough regions in the $ N $ limit. The proof of Theorem \ref{thm:finite_kernel} is given in Section \ref{sec:proof_main_theorem}.

\subsection*{Acknowledgments}

I thank Maurice Duits for many inspiring dicussions and for many fruitful comments on the manuscript. I thank Sunil Chhita for providing the code I used to simulate the random samples of the domino tilings of the Aztec diamond and also for inspiring dicussions. I thank Petter Br\"{a}nd\'{e}n for providing me with a proof of Lemma \ref{lem:real_zeros}.

\section{Statement of asymptotic results}\label{sec:global_local_properties}

In this section the asymptotic results on the $ 2\times k $-periodic Aztec diamond are given. For instance the arctic curves are described using a diffeomorphism between the rough region and a part of a Riemann surface. We start this section by defining this Riemann surface and relate it to the spectral curve given in \cite{KOS06}

\subsection{The spectral curve}

Consider the $ 2\times k $-periodic Aztec diamond and let $ \Phi $ be given in \eqref{eq:sec_model:weight}. The curve $ \det(\Phi(z)-\lambda I)=0 $ in $ \CC^2 $ determines the behavior on the macroscopic scale. In particular the branch points of the eigenvalues of $ \Phi $ determines the locations of the smooth regions and the arctic curves are parametrized in terms of the logarithmic derivative of the eigenvalues of $ \Phi $. 

Consider
\begin{equation}\label{eq:sec_model:def_p}
 p(z) = \left((z-1)^k\Tr \Phi(z)\right)^2-4(z-1)^{2k},
\end{equation}
the discriminant of $ (z-1)^k\Phi(z) $. Since each entry in $ (z-1) \phi_{2m-1}(z)\phi_{2m}(z) $ is a polynomial it follows that $ p $ is a polynomial. The zeros of $ p $ are precisely at the branch points of the eigenvalues of $ \Phi $. 
\begin{proposition}\label{prop:zeros_of_p}
 Consider the $ 2\times k $-periodic Aztec diamond with corresponding polynomial $ p $ given by \eqref{eq:sec_model:def_p}. The zeros of $ p $, denoted by $ s_\ell $, $ \ell=0,\dots,2(k-1) $, are real and fulfill the following ordering,
 \begin{equation}
  0=s_0>s_1\geq s_2 > s_3\geq s_4>\dots>s_{2k-3}\geq s_{2(k-1)}. 
 \end{equation}
\end{proposition}
By the above proposition the degree of $ p $ is $ 2k-1 $ and the zeros of $ p $ are of at most order two. Let $ p_0 $ and $ q $ be polynomials with simple zeros such that $ p=q^2p_0 $. We take the leading coefficient of $ q $ as one. Let $ 0=x_0>x_1>\dots> x_{2k'-3}>x_{2(k'-1)} $ be the zeros of $ p_0 $. Here $ k'\leq k $ and the degree of $ p_0 $ is $ 2k'-1 $ and the degree of $ q $ is $ k-k' $.

Consider the compact Riemann surface $ \mathcal R $ defined as the zero set in $ \CC^2 $ of
\begin{equation}
 Q(z,w) = w^2 - p_0(z),
\end{equation}
together with the point $ y_{2k'-1}=(\infty,\infty) $. Let 
\begin{equation}
 C^+_0 = \{(z,w) \in \mathcal R: z\in [0,\infty)\}\cup\{y_{2k'-1}\},
\end{equation}
\begin{equation}
 C^+_\ell = \{(z,w) \in \mathcal R:z \in \overline{I_\ell}\}, 
\end{equation}
for $ \ell = 1,\dots, k'-1 $, where $ I_\ell = (x_{2\ell},x_{2\ell-1}) $, and 
\begin{equation}
 C^+ = \bigcup_{\ell=0}^{k'-1}C^+_\ell.
\end{equation}
For $ \ell = 0,\dots,2(k'-1) $ denote $ y_\ell = (x_\ell,0) \in \mathcal R $. By construction the genus of $ \mathcal R $ is $ k'-1 $.

Below the spectral curve given in \cite{KOS06} is computed. This will not be used in the rest of the paper, but is computed to connect our results to \cite{KOS06} and \cite{KO07}. See in particular Corollary \ref{cor:burger}.
\begin{figure}[t]
\begin{center}
 \begin{tikzpicture}[scale=0.75]
 \foreach \x in {0,2,4,6}
 {\foreach \y in {1,3}
  {\draw (\x,\y)--(\x+1,\y+1);
 }
 }
   \foreach \x in {1,3,5}
 {\foreach \y in {0,2,4}
  {\draw (\x,\y)--(\x+1,\y+1);
 }
 }
 \foreach \y in {3,5}
  {\draw (0,\y)--(1,\y-1);
 }
 \foreach \y in {1,3}
 {\draw (6,\y)--(7,\y-1);
 }
 \foreach \x in {2,4}
 {\foreach \y in {1,3,5}
  {\draw (\x,\y)--(\x+1,\y-1);
 }
 }
   \foreach \x in {1,3,5}
 {\foreach \y in {2,4}
  {\draw (\x,\y)--(\x+1,\y-1);
 }
 }
  \foreach \x in {0,2,4}
 {\foreach \y in {1,3}
  {\draw (\x+1,\y+1) node[circle,draw=black,fill=black,inner sep=2pt]{};
 }
 }
 \foreach \x in {2,4,6}
 {\foreach \y in {1,3}
  {\draw (\x,\y) node[circle,draw=black,fill=white,inner sep=2pt]{};
 }
 }
 {\draw (2.5,3.5) node[above] {$\beta_1$};
 \draw (1.5,3.5) node[above] {$-\alpha_1$};
 \draw (2.5,1.5) node[above] {$\beta_1^{-1}$};
 \draw (1.5,1.5) node[above] {$-\alpha_1^{-1}$};
 \draw (4.5,3.5) node[above] {$\beta_2$};
 \draw (3.5,3.5) node[above] {$-\alpha_2$};
 \draw (4.5,1.5) node[above] {$\beta_2^{-1}$};
 \draw (3.5,1.5) node[above] {$-\alpha_2^{-1}$};
 \draw (6.8,3.5) node[above] {$\beta_3\lambda$};
 \draw (5.5,3.5) node[above] {$-\alpha_3$};
 \draw (6.8,1.5) node[above] {$\beta_3^{-1}\lambda$};
 \draw (5.5,1.5) node[above] {$-\alpha_3^{-1}$};
 }
 \foreach \x in {1,...,5}
 {\draw (\x+.5,.5) node[above] {$ z $};
 }
 \draw (6.5,.5) node[above] {$ z\lambda $};
 \draw (6.5,2.5) node[right] {$ \lambda $};
 
 \draw [<-](-.5,4.5)--(7.5,4.5);
 \draw (-.5,4.5) node[below] {$\gamma_x$};
  \draw [<-](.5,-.5)--(.5,5);
 \draw (.5,-.5) node[left] {$\gamma_y$};
 \end{tikzpicture}
\end{center}
\caption{The complex weight used to define the magnetically altered Kastelyn matrix. Here $ k=3 $. \label{fig:magnetic_weights}}
\end{figure}
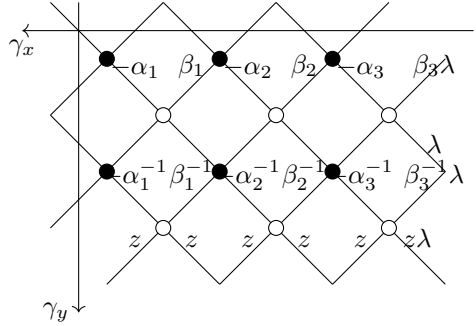

The magnetically altered Kastelyn matrix $ K(z,\lambda) $ is defined such that $  \left(K(z,\lambda)\right)_{i,j} $ is the weight on the edge connecting the black vertex number $ i $ with the white vertex number $ j $, indicated in Figure \ref{fig:magnetic_weights} for $ k=3 $, see \cite{KOS06}. Hence
\begin{equation}
K(z,\lambda)=
 \begin{pmatrix}
  A_1 & \textbf{0} & \textbf{0} & \cdots & \lambda B_k \\
  B_1 & A_2 & \textbf{0} & \cdots & \textbf{0} \\
  \textbf{0} & B_2 & A_3 & \cdots & \textbf{0} \\
   & \cdots &  & \ddots & \\
   \textbf{0} & \textbf{0} & \textbf{0} & \cdots & A_k
 \end{pmatrix}
\end{equation}
where 
\begin{equation}
 A_j =
 \begin{pmatrix}
  -\alpha_j^{-1} & 1\\
  z & -\alpha_j
 \end{pmatrix},
 \quad
 B_j=
 \begin{pmatrix}
  \beta_j^{-1} & 1\\
  z & \beta_j
 \end{pmatrix}
 \text{ and }
 \textbf{0} = 
 \begin{pmatrix}
  0 & 0 \\
  0 & 0
 \end{pmatrix}.
\end{equation}
The characteristic polynomial is given by $ P_C(z,\lambda) = \det K(z,\lambda) $ and the spectral curve is the set satisfying $ P_C(z,\lambda) = 0 $. A simplification, for our purposes, is done using the block structure of the matrix. Namely,
\begin{equation}
 P_C(z,\lambda)=\prod_{j=1}^k\det A_j\det \left(I-(-1)^k\lambda B_kA_k^{-1}\cdots B_1 A_1^{-1}\right),
\end{equation}
and since
\begin{equation}
 \left(B_mA_m^{-1}\right)^T =
 \begin{pmatrix}
  z^\frac{1}{2} & 0 \\
  0 & z^{-\frac{1}{2}}
 \end{pmatrix}
 \phi_{2m-1}(z)\phi_{2m}(z)
 \begin{pmatrix}
  z^{-\frac{1}{2}} & 0 \\
  0 & z^\frac{1}{2}
 \end{pmatrix},
\end{equation}
we obtain
\begin{equation}
 P_C(z,\lambda) = (1-z)^k\det\left(I-(-1)^k\lambda\Phi(z)\right)=(1-z)^k\det\left(\lambda I-(-1)^k\Phi(z)\right).
\end{equation}
The Newton polygon $ N(P_C) $ is the rectangle
\begin{equation}\label{eq:sec_model:newton_polygon}
 N(P_C) = \{(x,y)\in \RR^2:0 \leq x \leq k, 0\leq y\leq 2 \},
\end{equation}
which has $ k-1 $ integer points in the interior. By \cite[Theorem 5.3]{KOS06} we therefore expect to see $ k-1 $ smooth components for a generic choice of weightings. The same theorem also states that the number of smooth regions equals the genus of the spectral curve. Theorem \ref{thm:number_of_smooth_components} below confirms this last statement in our setting, since the genus of the above spectral curve is the same as the genus of $ \mathcal R $.

\subsection{Definition of the frozen, rough and smooth regions}

As the size of the Aztec diamond tends to infinity, we consider global coordinates so the Aztec diamond, or more precisely the top part of the non-intersecting paths corresponding to the Aztec diamond, is contained in the square $ [-1,1]^2 $. Let 
\begin{multline}\label{eq:sec_model:local_coordinates}
 \xi = \frac{kN}{4}(\eta-1)+e_\eta+\zeta, \quad m = \frac{N}{2}(\chi+1)+e_\chi+\kappa \\
 \xi' = \frac{kN}{4}(\eta-1)+e_\eta+\zeta' \text{ and } m'=\frac{N}{2}(\chi+1)+e_\chi+\kappa', 
\end{multline}
where $ (\chi,\eta) \in (-1,1)^2 $ is the global coordinate and $ (\kappa,\zeta),(\kappa',\zeta') \in \ZZ^2 $ are the local coordinates and $ e_\eta,e_\chi \in [0,1) $ are error terms so that the right hand side of each expressions is in $ \ZZ $. Recall that $ 0<m,m'<N $ and $ -\frac{kN}{2}\leq \xi,\xi'\leq -1 $, take $ (\kappa,\zeta) $ ans $ (\kappa',\zeta') $ so these inequalities hold. The local coordinates will not be used until Section \ref{sec:local_picture}.

For $ (\chi,\eta)\in (-1,1)^2 $ consider the function defined for $ (z,w) \in \mathcal R $ given by
\begin{equation}\label{eq:sec_model:F}
 F(z,w;\chi,\eta) = \frac{k}{4}(\eta+1)\log z -\frac{k}{2}\log(z-1)-\frac{\chi}{2}\log \rho(z,w),
\end{equation}
where $ \rho:\mathcal R \to \CC $ is the meromorphic function
\begin{equation}\label{eq:sec_model:eigenvalue}
 \rho(z,w) = \frac{1}{2}\Tr \Phi(z) + \frac{q(z)w}{2(z-1)^k}.
\end{equation}
The function $ F $ is multi-valued, to make it single-valued we take the logarithm as the principle branch. The choice of branch does barely play a role in the analysis. The only situation when it does matter is in the formulation of Theorem \ref{thm:height_function_expectation}. The reason to study this particular function will be clear as we perform the steepest descent analysis of the correlation kernel given in Theorem \ref{thm:finite_kernel}. In fact, $ F $ is the function in the exponent in the integrand which is the reason why we care about the critical points of $ F $ and also the reason why the choice of branch does not play a big role.

\begin{lemma}\label{lem:critical_points}
 Let $ (\chi,\eta) \in (-1,1)^2 $, then $ F(z,w;\chi,\eta) $ has $ 2k' $ critical points in $ \mathcal R $.
\end{lemma}

\begin{lemma}\label{lem:zeros_on_loops}
 For all $ (\chi,\eta) \in (-1,1)^2 $ the function $ F $ has at least two distinct critical points of odd order on every loop $ C^+_\ell $, $ \ell=1,\dots,k'-1 $.
\end{lemma}
The proofs of the above lemmas are given in Section \ref{sec:proof_phases}.

By Lemmas \ref{lem:critical_points} and \ref{lem:zeros_on_loops}, the number of critical points of $ F $ are $ 2k' $ and $ 2k'-2 $ are located in $ C^+ $. What distinguishes the different regions that may appear in the limit is the location of the two remaining critical point.
\begin{definition}\label{def:phases}
 Let $ (\chi,\eta) \in (-1,1)^2 $. We say that $ (\chi,\eta) $ is in 
 \begin{enumerate}[(i)]
  \item $\mathcal G_F$, the frozen region, if $ F $ has two distinct critical points in $ C_0^+ $,
  \item $\mathcal G_R$, the rough region, if $ F $ has non-real critical points, that is if there are critical points in $ \mathcal R \backslash C^+ $,
  \item $\mathcal G_S$, the smooth region, if $ F $ has four distinct critical points in $ C_\ell $ for some $ \ell = 1,\dots,k'-1 $,
  \item $\mathcal E_{FR} $, the boundary between the frozen and the rough regions, if $ F $ has a critical point of order two in $ C_0^+ $,
  \item $\mathcal E_{RS}^{(\ell)} $, for $ \ell=1,\dots,k'-1 $, the boundary between the rough and smooth regions, if $ F $ has a critical point of (at least) order two in $ C_\ell^+ $.
 \end{enumerate}
\end{definition}
By Lemmas \ref{lem:critical_points} and \ref{lem:zeros_on_loops} the above sets are well defined in the sense that they are disjoint. Moreover, it turns out that all $ (\chi,\eta)\in (-1,1)^2 $ are located in one of the sets $ (i)-(v) $, see Proposition \ref{prop:boundary_is_boundary}. That the definition is well defined in the sense that the correlations decay as they should, exponentially in the smooth region and polynomially in the rough region \cite{KOS06}, follows by the formulas in Theorems \ref{thm:local_smooth} and \ref{thm:local_rough}.

\subsection{The arctic curves}\label{sec:global_picture}

In this section we give the global behavior as $ N\to \infty $. The proofs are found in Section \ref{sec:proof_arctic_curves}.

The first proposition tells us about the geometry of the sets $ \mathcal E_{FR}^{(\ell)} $ and $ \mathcal E_{RS}^{(\ell)} $ for $ \ell=1,\dots,k'-1 $.

\begin{proposition}\label{prop:boundary}
 For $ \ell=1,\dots,k'-1 $ the sets $ \mathcal E_{RS}^{(\ell)} $ are simple closed curves containing the points 
 \begin{equation}
  \left(0,\frac{x_{2\ell}+1}{x_{2\ell}-1}\right) \text{ and } \left(0,\frac{x_{2\ell-1}+1}{x_{2\ell-1}-1}\right).
 \end{equation}
 We add the four points $ (0,\pm 1) $ and $ \left(\pm 1,\frac{2}{k}\frac{p'(1)}{p(1)}-1\right) $ to $ \mathcal E_{FR} $ and call also this extended set $ \mathcal E_{FR} $. Then $ \mathcal E_{FR} $ is a simple closed curve. Moreover the above curves are disjoint and symmetric with respect to the line $ \chi=0 $. 
\end{proposition}
In fact, the proof of the above statement gives a parameterization of the curves. Namely, let
\begin{equation}
 \chi(z) = \frac{k}{(z-1)^2\frac{\d}{\d z}\left(\frac{z\rho_1'(z)}{\rho_1(z)}\right)}
\end{equation}
and
\begin{equation}
 \eta(z) = \frac{1}{z-1}\left(\frac{2\frac{z\rho_1'(z)}{\rho_1(z)}}{(z-1)\frac{\d}{\d z}\left(\frac{z\rho_1'(z)}{\rho_1(z)}\right)}+z+1\right),
\end{equation}
where $ \rho_1 $ is an eigenvalue of $ \Phi $, see the beginning of Section \ref{sec:correlation_kernel} or Section \ref{sec:properties_of_eigenvalues} for a precise definition. Then $ I \ni z \mapsto (\chi(z),\eta(z)) $ together with $ I \ni z \mapsto (-\chi(z),\eta(z)) $ parametrizes $ \mathcal E_{RS}^{(\ell)} $ if $ I = \overline{I_\ell} $ and $ \mathcal E_{FR} $ if $ I = [0,\infty] $.

 \begin{figure}[t]
 \begin{center}
 \begin{tikzpicture}[scale=.68]
    \draw (-11,-1.1) node {\includegraphics[scale=.34]{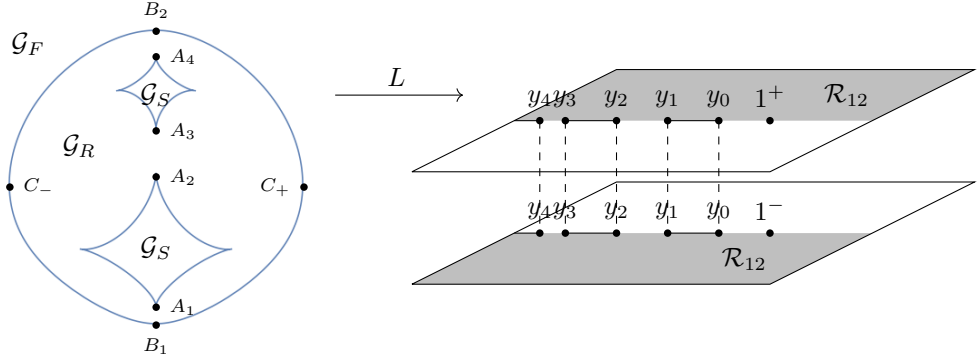}};
    \draw (-12.5,-.5) node {$ \mathcal G_R $};
    \draw (-11,.5) node {$ \mathcal G_S $};
    \draw (-11,-2.5) node {$ \mathcal G_S $};
    \draw (-13.5,1.5) node {$ \mathcal G_F $};
    \draw (-11,1.25) node[circle,fill,inner sep=1pt,label=right:\scriptsize $A_4$]{};
    \draw (-11,-.2) node[circle,fill,inner sep=1pt,label=right:\scriptsize $A_3$]{};
    \draw (-11,-1.1) node[circle,fill,inner sep=1pt,label=right:\scriptsize $A_2$]{};
    \draw (-11,-3.65) node[circle,fill,inner sep=1pt,label=right:\scriptsize $A_1$]{};
    \draw (-11,-4) node[circle,fill,inner sep=1pt,label=below:\scriptsize $B_1$]{};
    \draw (-11,1.75) node[circle,fill,inner sep=1pt,label=above:\scriptsize $B_2$]{};
    \draw (-8.11,-1.3) node[circle,fill,inner sep=1pt,label=left:\scriptsize $C_+$]{};
    \draw (-13.86,-1.3) node[circle,fill,inner sep=1pt,label=right:\scriptsize $C_-$]{};
   \fill[color=lightgray]
   (-4,0)--(3,0)--(5,1)--(-2,1);
   \draw (-6,-1)--(1,-1)--(5,1)--(-2,1)--(-6,-1);
   \draw (1,0) node[circle,fill,inner sep=1pt,label=above:$1^+$]{};
   \draw (0,0) node[circle,fill,inner sep=1pt,label=above:$y_0$]{};
   \draw (-1,0) node[circle,fill,inner sep=1pt,label=above:$y_1$]{};
   \draw (-2,0) node[circle,fill,inner sep=1pt,label=above:$y_2$]{};
   \draw (-3,0) node[circle,fill,inner sep=1pt,label=above:$y_3$]{};
   \draw (-3.5,0) node[circle,fill,inner sep=1pt,label=above:$y_4$]{};
   \draw (2.5,.5) node {$ \mathcal R_{12} $};
   \draw (-1,0)--(0,0);
   \draw (-3,0)--(-2,0);
   \draw (-4,0)--(-3.5,0);
   \fill[color=lightgray]
   (-6,-3.2)--(1,-3.2)--(3,-2.2)--(-4,-2.2);
   \draw (-6,-3.2)--(1,-3.2)--(5,-1.2)--(-2,-1.2)--(-6,-3.2);
   \draw (1,-2.2) node[circle,fill,inner sep=1pt,label=above:$1^-$]{};
   \draw (0,-2.2) node[circle,fill,inner sep=1pt,label=above:$y_0$]{};
   \draw (-1,-2.2) node[circle,fill,inner sep=1pt,label=above:$y_1$]{};
   \draw (-2,-2.2) node[circle,fill,inner sep=1pt,label=above:$y_2$]{};
   \draw (-3,-2.2) node[circle,fill,inner sep=1pt,label=above:$y_3$]{};
   \draw (-3.5,-2.2) node[circle,fill,inner sep=1pt,label=above:$y_4$]{};
   \draw (.5,-2.7) node {$ \mathcal R_{12} $};
   \draw (-1,-2.2)--(0,-2.2);
   \draw (-3,-2.2)--(-2,-2.2);
   \draw (-4,-2.2)--(-3.5,-2.2);
   \draw[dashed] (0,0)--(0,-2.2);
   \draw[dashed] (-1,0)--(-1,-2.2);
   \draw[dashed] (-2,0)--(-2,-2.2);
   \draw[dashed] (-3,0)--(-3,-2.2);
   \draw[dashed] (-3.5,0)--(-3.5,-2.2);
   \draw[->] (-7.5,.5)--(-5,.5);
   \draw (-6.3,.5) node[above]{$L$};
  \end{tikzpicture}
 \end{center}
  \caption{The diffeomorphism $ L $. See \eqref{eq:sec_model:special_points} for the relation between the points indicated in the figure and recall that $ y_\ell = (x_\ell,0) \in \mathcal R $. \label{fig:L}}
\end{figure}

To describe the macroscopic picture we introduce the following map. Let $ \mathcal R_{12} \subset \mathcal R $ be the interior of the closure of the set $ \{(z,p_0(z)^\frac{1}{2}):\im z>0\} \cup \{(z,-p_0(z)^\frac{1}{2}):\im z<0\} $. Here $ p_0(z)^\frac{1}{2} $ is the analytic continuation of the map $ t\in \RR_+ \mapsto p_0(t)^\frac{1}{2} $ to the slitted plane $ \CC \backslash \left((-\infty,x_{2(k'-1)}]\cup_{\ell=0}^{k'-2} [x_{2\ell+1},x_{2\ell}]\right) $, and for later purposes we define $ p_0^\frac{1}{2} $ on the cuts as the limit from the upper half plane. In words $ \mathcal R_{12} $ is the upper half plane of the first sheet and the lower half plane of the second sheet of $ \mathcal R $, such that it is connected and open. 
\begin{proposition}\label{prop:L_map}
For each $ (\chi,\eta) $ there is a unique critical point of $ F $ in $ R_{12} $, say $ (z,w) $. Define $ L:\mathcal G_R \to \mathcal R_{12} $ as
\begin{equation}\label{eq:sec_model:L}
 L(\chi,\eta) = (z,w).
\end{equation}
 The map $ L $ is a diffeomorphism.
\end{proposition}
Similar maps appear also in other models, most common as a map from the rough region to the upper half plane. However in \cite{BG19} and \cite{CDKL19} the rough region and hence the image of the map is a multiply connected domain respectively two disjoint domains. Using that $ L $ is a diffeomorhism we obtain the following proposition, Figure \ref{fig:L}. 
\begin{proposition}\label{prop:boundary_is_boundary}
 The boundary of $ \mathcal G_R $ is the union of $ \mathcal E_{RS}^{(\ell)} $ for $ \ell=1,\dots,k'-1 $ and $ \mathcal E_{FR} $. The curve $ \mathcal E_{FR} $ separates $ \mathcal G_R $ and $ \mathcal G_F $ and the curves $ \mathcal E_{RS}^{(\ell)} $ for $ \ell=1,\dots,k'-1 $ separates $ \mathcal G_R $ and $ \mathcal G_S $. Moreover, all $ (\chi,\eta) \in (-1,1)^2 $ belongs to one of the sets in Definition \ref{def:phases}. 
\end{proposition}

From the above propositions we conclude that the number of smooth regions equal the number of loops $ C_\ell $ which also is the genus of $ \mathcal R $.
\begin{theorem}\label{thm:number_of_smooth_components}
 The number of connected components of $ \mathcal G_S $ is $ k'-1 $.
\end{theorem}

To get a feeling about the map $ L $ we list a few special points on the arctic curves and corresponding point on the boundary of $ \mathcal R_{12} $. Denote $ \mathcal R \ni 1^\pm = (1,\pm p_0(1)^\frac{1}{2}) $. The following limits hold
\begin{multline}\label{eq:sec_model:special_points}
 \lim_{\mathcal R_{12} \ni (z,w) \to y_\ell} L^{-1}(z,w) = \left(0,\frac{x_{\ell}+1}{x_{\ell}-1}\right) = A_\ell, \text{ for } \ell=1,\dots,2(k'-1), \\
 \lim_{\mathcal R_{12} \ni (z,w) \to y_0} L^{-1}(z,w) = (0,-1)=B_1, \quad
 \lim_{\mathcal R_{12} \ni (z,w) \to y_{2k'-1}} L^{-1}(z,w) = (0,1)=B_2, \\
 \lim_{\mathcal R_{12} \ni (z,w) \to 1^\pm} L^{-1}(z,w) = \left(\pm 1,\frac{2}{k}\frac{p'(1)}{p(1)}-1\right)=C_\pm,
\end{multline}
which follows from the parametrization above, see also the proof of Proposition \ref{prop:boundary}.

 Proposition \ref{prop:zeros_of_p} now shows what macroscopic pictures that may and may not appear. That a smooth region vanishes means that $ s_{2\ell} \to s_{2\ell-1} $ for some $ \ell=1,\dots,k-1 $ which does not violate Proposition \ref{prop:zeros_of_p} and we know already from the two-periodic Aztec diamond that a smooth region may vanish. That two connected components of the smooth regions touch is the same as $ s_{2\ell}=s_{2\ell+1} $ for some $ \ell=1,\dots, k-2 $, which can not happen. That the smooth regions would touch the frozen region is the same as $ s_0=s_1 $ or $ s_{2(k-1)}=\infty $, which are also impossible. However two smooth components could touch in a degenerated weighting, that is when some parameters are taken to zero. This is mentioned in \cite{DS14}, and is something we hope to come back to in future work.

\subsection{The height function}\label{sec:height_function_dimer}

 A special object of interest is the height function \cite{T90}. The height function transform the tiling problem into a problem about random surfaces, see \cite{CKP00}.
 
 Since we have a non-intersecting paths perspective we consider the height function of a non-intersecting paths model, see for instance \cite{D18}. Let $ D_{h} = \{(u,v)\in \ZZ^2: 0\leq u \leq 2kN, -kN\leq v \leq -1\} $ and define the height function $ h:D_{h}\to \NN $ as
 \begin{equation}
  h(u,v) = \#\{\text{ point at } (u,v_0), v_0\geq v\} = \sum_{v_0 \geq v} \mathds{1}_{(u,v_0)}.
 \end{equation}
 The height function of the non-intersecting path model is closely related to the height function of the dimer model which is a function on the faces of the graph. It is basically a linear transformation to go from one to the other and in particular this transformation is deterministic.
  
 We are in the present paper interested of the height function on the global scale. It is therefore sufficient to consider the points $ (u,v)=(2km,2\xi) $ for some $ m,\xi \in \ZZ $. Recall the global coordinates given in \eqref{eq:sec_model:generalized_weight} with the local coordinates set to zero. Define 
 \begin{equation}
  h^{(N)}(\chi,\eta) = h\left(kN(\chi+1)+2ke_\chi,\frac{kN}{2}(\eta-1)+2e_\eta \right).
 \end{equation}
The limit shape of the $ 2\times k $-periodic Aztec diamond is given in the following theorem. The proof is given in Section \ref{sec:proof_height_function}.

\begin{theorem}\label{thm:height_function_expectation}
 Consider the $ 2\times k $-periodic Aztec diamond of size $ kN\times kN $ with $ N $ even.
 
 If $ (\chi,\eta)\in \mathcal G_S $ and more precisely is located in the interior of $ \mathcal E_{RS}^{(\ell)} $, then
 \begin{equation}
  \EE\left[\frac{2}{kN}h^{(N)}(\chi,\eta)\right] \to -\frac{1}{2}(\eta-1)+\frac{\chi}{k}\left(n_\ell-k\right),
 \end{equation}
 as $ N\to \infty $, where $ n_\ell = \frac{1}{2} \#\{s_i, \text{ zeros of } p: s_i\geq x_{2\ell-1} \} $.
   
 If $ (\chi,\eta)\in \mathcal G_R $ with $ \chi>0 $, then
  \begin{equation}
  \EE\left[\frac{2}{kN}h^{(N)}(\chi,\eta)\right]\to -\frac{1}{\pi\i k}\int_{\gamma_{z_1}}\frac{\d}{\d z}\left(F(z,p_0(z)^\frac{1}{2};\chi,\eta)\right)\d z
 \end{equation}
  as $ N\to \infty $. Here $ p_0^\frac{1}{2} $ is defined just before Proposition \ref{prop:L_map}, $ z_1 $ is the first component of $ L(\chi,\eta) $ and $ \gamma_{z_1} $ is a curve going from $ \bar z_1 $ to $ z_1 $ intersecting the real line only at a point bigger than one.
\end{theorem}
\begin{remark}
 In the rough region a similar formula can be derived also if $ \chi\leq 0 $.
\end{remark}

Denote $ \EE\left[h^{(\infty)}(\chi,\eta)\right] = \lim_{N\to\infty}\EE\left[\frac{2}{kN}h^{(N)}(\chi,\eta)\right] $. A few remarks are in place.

It is interesting that if $ (\chi,\eta) \in \mathcal G_S $ then $ \EE\left[h^{(\infty)}(\chi,\eta)\right] $ does not depend on the details of the model, but only on $ k $ and in which smooth region $ (\chi,\eta) $ is located. A statement like this is however something to expect. Namely, according to \cite{KOS06} the slope of the height function in a smooth region, expressed in the appropriate coordinate system, can only take the values of the integer points in the interior of the Newton polygon \eqref{eq:sec_model:newton_polygon}. Keep in mind however that \cite{KOS06} consider models on the torus. 
 
A second remark is that with an appropriate choice of branch in the definition of $ F $ we have 
\begin{equation}
 \EE\left[h^{(\infty)}(\chi,\eta)\right] = -\frac{2}{\pi k}\im F(L(\chi,\eta);\chi,\eta).
\end{equation}

The last remark is the following corollary which should be compared with \cite[Theorem 1]{KO07}.

\begin{corollary}\label{cor:burger}
 Let $ (\chi,\eta) \in \mathcal G_R $ and consider the change of variables $ (\chi,\eta) = (kx,2y) $. Let $ x>0 $ and set $ f(\chi,\eta) = \rho(L(\chi,\eta)) $ and $ g(\chi,\eta) = L_1(\chi,\eta) $, the first component of $ L(\chi,\eta) $. Then
 \begin{equation}
  \nabla \EE\left[h^{(\infty)}(kx,2y)\right] = \frac{1}{\pi}\left(\arg(f(\chi,\eta)),-\arg(g(\chi,\eta)) \right),
 \end{equation}
 and solves the equations
 \begin{equation}
  f\frac{\partial g}{\partial x}+g\frac{\partial f}{\partial y} = 0,
 \end{equation}
 and
 \begin{equation}
  \det \left(\Phi\left(g\right)-f I\right) = 0.
 \end{equation}
\end{corollary}

\subsection{Local correlations}\label{sec:local_picture}

We will now present the local correlation kernel in the smooth and rough regions. The proofs are given in Section \ref{sec:proof_local_correlations}. Recall that $ \rho_1 $ and $ \rho_2 $ are the eigenvalues of $ \Phi $ and $ E $ is a matrix consisting of eigenvectors of $ \Phi $, see the section just after \eqref{eq:sec_model:generalized_weight}.

\begin{theorem}\label{thm:local_smooth}
 Consider the $ 2\times k $-periodic Aztec diamond of size $ kN\times kN $ with $ N $ even. Let $ m,m',\xi, \xi' $ be defined in \eqref{eq:sec_model:local_coordinates} and let $ K $ be the correlation kernel in Theorem \ref{thm:finite_kernel}. For $ (\chi,\eta) \in \mathcal G_S $ and more precisely in the interior of $ \mathcal E_{RS}^{(\ell)} $, 
 \begin{multline}
 \lim_{N\to \infty}\left[K\left(2km,2\xi+i;2km',2\xi'+j\right)\right]_{i,j=0}^1 \\
  = \left[K_{smooth}^{(\ell)}(2k\kappa,2\zeta+i;2k\kappa',2\zeta'+j)\right]_{i,j=0}^1,
\end{multline}
where 
 \begin{multline}
  \left[K_{smooth}^{(\ell)}(2k\kappa,2\zeta+i;2k\kappa',2\zeta'+j)\right]_{i,j=0}^1 \\
  = \frac{\mathds{1}_{\kappa\leq \kappa'}}{2\pi\i}\oint_{\gamma_{\ell}} \rho_1(z)^{\kappa-\kappa'}z^{\zeta'-\zeta}E(z)
  \begin{pmatrix}
   1 & 0 \\
   0 & 0
  \end{pmatrix}
  E(z)^{-1}\frac{\d z}{z} \\
  -\frac{\mathds{1}_{\kappa>\kappa'}}{2\pi\i}\oint_{\gamma_\ell} \rho_2(z)^{\kappa-\kappa'}z^{\zeta'-\zeta}E(z)
  \begin{pmatrix}
   0 & 0 \\
   0 & 1
  \end{pmatrix}
  E(z)^{-1}\frac{\d z}{z}, \quad
  \kappa,\kappa',\zeta,\zeta' \in \ZZ.
 \end{multline}
 Here $ \gamma_{\ell} $ is a simple closed curve with $ 0 $ and $ 1 $ in its interior and which intersects the real line exactly two times, ones at $ I_\ell $ and ones at $ (1,\infty) $.
\end{theorem}

\begin{theorem}\label{thm:local_rough}
 Consider the $ 2\times k $-periodic Aztec diamond of size $ kN\times kN $ with $ N $ even. Let $ m,m',\xi, \xi' $ be defined in \eqref{eq:sec_model:local_coordinates} and let $ K $ be the correlation kernel in Theorem \ref{thm:finite_kernel}. For $ (\chi,\eta) \in \mathcal G_R $,
 \begin{multline}
 \lim_{N\to \infty}\left[K\left(2km,2\xi+i;2km',2\xi'+j\right)\right]_{i,j=0}^1 \\
  = \left[K_{rough}^{(\chi,\eta)}(2k\kappa,2\zeta+i;2k\kappa',2\zeta'+j)\right]_{i,j=0}^1,
\end{multline}
where 
 \begin{multline}
  \left[K_{rough}^{(\chi,\eta)}(2k\kappa,2\zeta+i;2k\kappa',2\zeta'+j)\right]_{i,j=0}^1 \\
  = \frac{\mathds{1}_{\chi\leq 0}-\mathds{1}_{\kappa>\kappa'}}{2\pi\i}\oint_{\gamma_{1,0}} \Phi(z)^{\kappa-\kappa'}z^{\zeta'-\zeta}\frac{\d z}{z} \\
  +\frac{1}{2\pi\i}\int_{\gamma_{z_1}} z^{\zeta'-\zeta}\rho_{1+\mathds{1}_{\chi\leq 0}}(z)^{\kappa-\kappa'}E(z)
  \begin{pmatrix}
   \mathds{1}_{\chi>0} & 0 \\
   0 & \mathds{1}_{\chi\leq 0}
  \end{pmatrix}
  E(z)^{-1}\frac{\d z}{z}, \\
  \kappa,\kappa',\zeta,\zeta' \in \ZZ.
 \end{multline}
 Here $ \gamma_{z_1} $ is a simple curve going from $ \overline {z_1} $ to $ z_1 $ and intersect the real line at $ (1,\infty) $, where $ z_1 = z_1(\chi,\eta) $ is the first component of $ L(\chi,\eta) $ and $ \gamma_{0,1} $ is a simple closed curve going around $ 0 $ and $ 1 $.
\end{theorem}

The indicator function depending of $ \chi $ appear in the formula in Theorem \ref{thm:local_rough} since the integral is written as an integral on the complex plane while it more naturally is an integral on the Riemann surface $ \mathcal R $.

As a consistency check we compute the correlation kernel for the uniform weighting.
\begin{example}\label{ex:sine_kernel}
 Consider the uniform weighting, that is $ \alpha_i=\beta_i = 1 $ for $ i=1,\dots,k $. Let $ (\chi,\eta) \in \mathcal G_R $ and let $ \theta = \theta(\chi,\eta) $ and $ u=u(\chi,\eta) $ be the argument respectively the absolute value of the projection of $ L(\chi,\eta) $ to its first argument. Consider the kernel on the line $ \kappa=\kappa' $ with $ \zeta,\zeta'\in \ZZ $. By Theorem \ref{thm:local_rough},
 \begin{multline}
  \left[K_{rough}^{(\chi,\eta)}(2k\kappa,2\zeta+i;2k\kappa,2\zeta'+j)\right]_{i,j=0}^1 \\
  = \frac{1}{4\pi\i}\int_{\gamma_{z_1}} z^{\zeta'-\zeta}
  \begin{pmatrix}
   1 & -z^{-\frac{1}{2}} \\
   -z^\frac{1}{2} & 1
  \end{pmatrix}
  \frac{\d z}{z} \\
  = \left[\frac{g(2\zeta+i)}{g(2\zeta'+j)}\mathbb S_\theta(2\zeta+i,2\zeta'+j)\right]_{i,j=0}^1,
 \end{multline}
 where $ \mathbb S_\theta(x,y) = \frac{\sin\left(\frac{\theta}{2}(x-y)\right)}{\pi(x-y)} $ and $ g(x) = (-u^\frac{1}{2})^x $.
\end{example}
  
\begin{remark}
 The limiting kernel for the frozen region is left out in order to make this paper no longer than necessary.
\end{remark}

\begin{remark}
 Let $ \kappa-\kappa' $ be fixed. Deform $ \gamma_\ell $ in the the formula in Theorem \ref{thm:local_smooth} to a circle of radius $ r\in (x_{2\ell},x_{2\ell-1}) $. By change of variables
 \begin{equation}
  \left[K_{smooth}^{(\ell)}(2k\kappa,2\zeta+i;2k\kappa',2\zeta'+j)\right]_{i,j=0}^1 = r^{\zeta'-\zeta}\hat f(\zeta-\zeta'),
 \end{equation}  
 where
 \begin{multline}
  f(z)=\mathds{1}_{\kappa\leq \kappa'}\rho_1(rz)^{\kappa-\kappa'}E(rz)
  \begin{pmatrix}
   1 & 0 \\
   0 & 0
  \end{pmatrix}
  E(rz)^{-1} \\
  -\mathds{1}_{\kappa>\kappa'}\rho_2(rz)^{\kappa-\kappa'}E(rz)
  \begin{pmatrix}
   0 & 0 \\
   0 & 1
  \end{pmatrix}
  E(rz)^{-1},
  \end{multline}
  and $ \hat f $ is the Fourier coefficient of $ f $. Since $ f $ is analytic in a neighbourhood of the unit circle $ \hat f(\zeta'-\zeta) $ decay exponentially as $ \zeta'-\zeta $ increases. So
 \begin{equation}
  K_{smooth}^{(\ell)}(2k\kappa,2\zeta+i;2k\kappa',2\zeta'+j)K_{smooth}^{(\ell)}(2k\kappa',2\zeta'+j;2k\kappa,2\zeta+i)
 \end{equation}
 decay exponentially as $ \zeta'-\zeta $ increases.
 
 The correlation kernel in the rough region is basically the Fourier coefficient of a smooth function times the indicator function of an interval, in the same way as the correlation kernel in the smooth region is the Fourier coefficient of an analytic function. By integration by parts it follows that 
  \begin{equation}
  K_{rough}^{(\chi,\eta)}(2k\kappa,2\zeta+i;2k\kappa',2\zeta'+j)K_{rough}^{(\chi,\eta)}(2k\kappa',2\zeta'+j;2k\kappa,2\zeta+i)
 \end{equation}
  decay polynomially.
\end{remark}

\section{Proof of Proposition \ref{prop:zeros_of_p}}
In this section we prove Proposition \ref{prop:zeros_of_p}. We first show that the zeros of $ p $ has to be real. The following lemma was proved by Petter Br\"{a}nd\'{e}n.

\begin{lemma}\label{lem:real_zeros}
 Consider the $ 2\times k $-periodic Aztec diamond with corresponding polynomial $ p $ given in \eqref{eq:sec_model:def_p}. The polynomial $ p $ has only real roots.
\end{lemma}

\begin{proof}
Recall that $ p $ is the discriminant of $ (z-1)^k\Phi(z) $. We prove that if $ z \not\in \RR $ then $ (z-1)^k \Phi(z) $ has two linearly independent eigenvectors and is not equal to a multiple of the identity matrix. That the eigenvalues are linearly independent implies that the eigenvalues are distinct or that the matrix is a multiple of the identity matrix. Since the matrix is not a multiple of the identity matrix the discriminant of $ (z-1)^k\Phi(z) $, that is $ p $, is non-zero. Since $ p $ is a real polynomial it is enough to prove the above for $ z $ in the upper half plane.
 
 Fix $ z \in \CC $ with $ \im z>0 $ and define for $ i=1,\dots, k $ the M\"{o}bius transformations
 \begin{equation}
  \varphi_{2i-1}(\xi) = \alpha_i^{-1}\frac{z+\alpha_i^{-1}\xi}{1+\alpha_i^{-1} \xi}, \quad \varphi_{2i}(\xi) = \beta_i\frac{z+\beta_i\xi}{1+\beta_i \xi},
 \end{equation}  
 and $ \Psi = \varphi_1\circ\varphi_2\circ \dots \circ \varphi_{2k} $. Consider the cone
 \begin{equation}
  C_z=\{\xi\in \CC:0<\arg \xi <\arg z\}.
 \end{equation}
 Let $ \xi \in C_z $, since $ C_z $ is a cone with boundary generated by $ z $ and $ 1 $ it follows that
 \begin{equation}
  0<\arg(z+\alpha_i^{-1}\xi)-\arg(1+\alpha_i^{-1}\xi)<\arg z.
 \end{equation}
 That is, $ \varphi_{2i-1}(C_z) \subseteq C_z $. The same is true for $ \varphi_{2i} $. Iterating this over all $ i $ we get $ \Psi(C_z) \subseteq C_z $.
 
 Let $ \varphi $ be a M\"{o}bius transformation taking $ C_z $ to a bounded set $ B_z $. Then $ \varphi \circ \Psi \circ \varphi^{-1}(\overline{B_z}) \subseteq \overline{B_z} $. By Brouwer fixed point theorem there is a point $ \xi_0 \in \overline{B_z} $ such that $ \varphi \circ \Psi \circ \varphi^{-1}(\xi_0)=\xi_0 $, so $ z_0=\varphi^{-1}(\xi_0)\in \overline{C_z} $ is a fixed point of $ \Psi $. Observe that  $ \varphi_{2k-1}\circ \varphi_{2k}(\xi) \in C_z $ if $ \xi \in \{0,\infty\} $. So neither $ 0 $ nor $ \infty $ are fixed points of $ \Psi $.
 
 Represent a M\"{o}bius transform
 \begin{equation}
  \psi(\xi) = \frac{c+d\xi}{a+b\xi} \text{ as } 
  A_\psi = 
  \begin{pmatrix}
   a & b \\
   c & d
  \end{pmatrix}.
 \end{equation}
 This representation is slightly different from the standard representation, but fits our purposes. Namely that
 \begin{equation}
  A_\psi
  \begin{pmatrix}
   1 \\  
   \xi
  \end{pmatrix}
  = (a+b\xi)
  \begin{pmatrix}
   1 \\  
   \psi(\xi)
  \end{pmatrix},
 \end{equation} 
 and thus any fixed point of $ \psi $ gives an eigenvector of $ A_\psi $. For two M\"{o}bius transformations $ \psi $ and $ \phi $ it holds that $ A_{\psi\circ \phi} = A_\psi A_\phi $, so
 \begin{equation}
  A_\Psi = (z-1)^k\Phi(z).
 \end{equation}
 Using the fixed point for $ \Psi $ we write down an eigenvector for $ (z-1)^k\Phi(z) $,
 \begin{equation}
  (z-1)^k\Phi(z)
  \begin{pmatrix}
   1 \\
   z_0
  \end{pmatrix}
  = A_\Psi
  \begin{pmatrix}
   1 \\
   z_0
  \end{pmatrix}
  = \lambda 
  \begin{pmatrix}
   1 \\
   \Psi(z_0)
  \end{pmatrix}
  = \lambda
  \begin{pmatrix}
   1 \\
   z_0
  \end{pmatrix},
 \end{equation}
 where $ \lambda = ((z-1)^k\Phi(z))_{11}+((z-1)^k\Phi(z))_{12}z_0 $. The above equality also implies that $ (z-1)^k\Psi(z) $ is not a multiple of the identity matrix. Indeed, in case $ (1,0) $ is an eigenvector of $ (z-1)^k\Phi(z) $, then $ \Psi(0) = 0 $, which we know is not the case.
 
 We use the same argument on $ (z-1)^k\Phi(z)^{-1} $, with  
 \begin{equation}
  \tilde\varphi_{2i-1}(\xi) = -\alpha_i\frac{z-\alpha_i\xi}{1-\alpha_i \xi} \text{ and } \varphi_{2i}(\xi) = -\beta_i^{-1}\frac{z-\beta_i^{-1}\xi}{1-\beta_i^{-1} \xi},
 \end{equation}
 instead of $ \varphi_{2i-1} $ and $ \varphi_{2i} $, $ \tilde \Psi = \tilde \varphi_{2k}\circ \tilde \varphi_{2k-1}\circ\dots \circ \tilde \varphi_1 $ instead of $ \Psi $ and  
 \begin{equation}
  \tilde C_z=\{\xi\in \CC:-\pi<\arg \xi<\arg(- z)\},
 \end{equation}
 instead of $ C_z $. We obtain a $ \tilde z_0 \in \tilde C_z $ such that
 \begin{equation}
  (z-1)^k\Phi(z)^{-1}
  \begin{pmatrix}
   1 \\
   \tilde z_0
  \end{pmatrix}
  = \tilde \lambda
  \begin{pmatrix}
   1 \\
   \tilde z_0
  \end{pmatrix}.
 \end{equation}
 In particular $ (1,\tilde z_0) $ is an eigenvector of $ (z-1)^k\Phi(z) $.
 
 To summarize $ (1,z_0) $ and $ (1,\tilde z_0) $ are both eigenvectors of $ (z-1)^k\Phi(z) $. Since the intersection of the closure of $ C_z $ and the closure of $ \tilde C_z $ is $ \{0,\infty\} $ and $ z_0\not\in \{0,\infty\} $, we conclude that $ (1,z_0) $ and $ (1,\tilde z_0) $ are linearly independent. Moreover $ (z-1)^k\Phi(z) $ is not a multiple of the identity matrix.
\end{proof}

In order to prove Proposition \ref{prop:zeros_of_p} we first prove the following lemma. The lemma is custom made in order to fit our purposes so the aim is not to give a general statement.

\begin{lemma}\label{lem:interlacing_zeros}
 Let $ u $ and $ v $ be real polynomials of degree $ n $ respectively $ n-1 $ for some $ n>0 $ and with zeros
 \begin{equation}
  1>s_1\geq s_2\geq \dots\geq s_n \text{ respectively } 1>t_1\geq t_2\geq \dots\geq t_{n-1}.
 \end{equation}
  Assume that $ v^{(n-\ell)}(z) < u^{(n-\ell)}(z) $ if $ \ell $ is even and $ u^{(n-\ell)}(z) < v^{(n-\ell)}(z) $ if $ \ell $ is odd for $ \ell = 1,\dots,n $ and $ z<1 $. Assume further that $ u^{(\ell)}(1),v^{(\ell)}(1)>0 $ for $ \ell = 0,\dots,n-1 $. Then
 \begin{equation}
  s_1 > t_1\geq t_2>s_2\geq s_3 > t_3\geq t_4>\dots>s_{n-1}\geq s_n,
 \end{equation}
 if $ n $ is odd and 
 \begin{equation}
  t_1 > s_1\geq s_2>t_2\geq t_3 > s_3\geq s_4>\dots>s_{n-1}\geq s_n,
 \end{equation}
 if $ n $ is even.
\end{lemma}

\begin{proof}
 The proof goes by induction over $ n $.
 
 If $ n=2 $ the assumptions says that $ 1>s_1,s_2,t_1 $ where $ u(s_1)=u(s_2)=v(t_1)=0 $, $ u(1),v(1)>0 $ and $ v(z)<u(z) $ for $ z<0 $. Hence, since both $ u $ and $ v $ are positive at $ 1 $ and the graph of $ v $ lies under the graph of $ u $ for $ z<1 $ we see that $ t_1>s_1\geq s_2 $.
 
 Assume $ n>2 $ and for simplicity that $ n $ is even. The derivatives $ u' $ and $ v' $ fulfill the conditions with $ n $ replaced by $ n-1 $. So
 \begin{equation}\label{eq:proof_interlacing_zeros:zeros_derivative}
  s_1' > t_1'\geq t_2'>s_2'\geq s_3' > t_3'\geq t_4'>\dots>s_{n-2}'\geq s_{n-1}',
 \end{equation}
 where $ s_i' $ and $ t_i' $ are the zeros of $ u' $ respectively $ v' $. Recall that for a real polynomial with only real zeros there is an interlacing structure between the zeros of the polynomial and the zeros of its derivative. Since $ u(1),v(1)>0 $ and $ v(z)<u(z) $ we get that $ t_1>s_1 $. By the interlacing property mentioned above $ t_1' \geq t_2\geq t_2' $ and $ s_1 \geq s_1' $ and hence by \eqref{eq:proof_interlacing_zeros:zeros_derivative} $ s_1 > t_2 $. It follows since $ v(z)<u(z) $ that $ v $ is negative for $ z \in [s_2,s_1] $, hence $ s_2 > t_2 $. We thus have $ t_1>s_1\geq s_2>t_2 $. We use the same argument again. Namely, by the interlacing structure and \eqref{eq:proof_interlacing_zeros:zeros_derivative}, $ t_2>s_3 $. Since $ v(z)<u(z) $ we get that $ u $ is strictly positive on $ [t_2,t_3] $ which implies that $ t_3>s_3 $. Hence $ t_2\geq t_3>s_3 $. Iterate this argument until there are no more zeros to conclude that
  \begin{equation}
  t_1 > s_1\geq s_2>t_2\geq t_3 > s_3\geq s_4>\dots>s_{n-1}\geq s_n.
 \end{equation}
 If $ n $ is odd the same argument holds, it is basically a matter of changing the name of the zeros and the details are left to the reader. This proves the statement.
\end{proof}

Let
\begin{equation}\label{eq:sec_zeros:ppm}
 p_\pm(z) = (z-1)^k\Tr \Phi(z) \pm 2(z-1)^k.
\end{equation}
These polynomials will play the role of $ u $ respectively $ v $ in previous lemma and since $ p=p_+p_- $ we obtain the structure of the zeros of $ p $.

\begin{proof}[Proof of Proposition \ref{prop:zeros_of_p}]
 By definition the zeros of $ p $ are precisely the zeros of $ p_+ $ together with the zeros of $ p_- $. By Lemmma \ref{lem:real_zeros} the zeros of $ p $ and hence $ p_+ $ and $ p_- $ are real. To prove the proposition we show that the zeros are non-positive, and that $ p_+ $ and $ p_- $ fulfill the condition of Lemma \ref{lem:interlacing_zeros} which concludes the proof.

Rewrite $ (z-1)^k\Phi(z) $ to
 \begin{equation}\label{eq:proof_zeros_of_p:capital_phi}
  (z-1)^k\Phi(z) = \prod_{m=1}^k
  \begin{pmatrix}
   z+\alpha_m\beta_m^{-1} &\alpha_m+\beta_m\\
   z(\alpha_m^{-1}+\beta_m^{-1}) & z+\beta_m\alpha_m^{-1}
  \end{pmatrix}.
 \end{equation} 
 It follows from \eqref{eq:proof_zeros_of_p:capital_phi} that $ (z-1)^k\Tr\Phi(z) $ is a real polynomial of degree $ k $ with positive coefficients and leading coefficient $ 2 $. So $ p_+ $ is a polynomial of degree $ k $ and $ p_- $ is a polynomial of degree $ k-1 $.
 
 Ignoring the part on the anti-diagonal in \eqref{eq:proof_zeros_of_p:capital_phi} using positivity yield, for $ z>0 $, the inequality
 \begin{equation}
  (z-1)^k\Tr \Phi(z)\geq \prod_{m=1}^k(z+\alpha_m\beta_m^{-1})+\prod_{m=1}^k(z+\alpha_m^{-1}\beta_m) = \sum_{\ell=0}^ka_\ell z^\ell,
 \end{equation}
 where 
 \begin{equation}
  a_\ell=\sum_{i=1}^{\binom{k}{\ell}}(\gamma_i+\gamma_i^{-1})\geq 2 \binom{k}{\ell},
 \end{equation}
 and $ \gamma_i =\prod \alpha_m\beta_m^{-1} $ where the product is over a subset of $ \{1,\dots,k\} $. So both $ p_+ $ and $ p_- $ and hence $ p $ has non-negative coefficients. Hence, all zeros of $ p $ are non-positive.
 
 Note that $ p_+(z) = p_-(z)+4(z-1)^k $ and therefore 
 \begin{equation}
  p_+^{(k-\ell)}(z) = p_-^{(k-\ell)}(z)+4k(k-1)\cdots (\ell+1)(z-1)^\ell,
 \end{equation}
 for $ \ell = 1,\dots,k $. So $ p_-^{(k-\ell)}(z) < p_+^{(k-\ell)}(z) $ if $ \ell $ is even and $ p_+^{(k-\ell)}(z) < p_-^{(k-\ell)}(z) $ if $ \ell $ is odd. Since $ p_+ $ and $ p_- $ has non-negative coefficients it is clear that $ p_\pm^{\ell}(1)>0 $ for $ \ell=0,\dots,k-1 $.
 
 By Lemma \ref{lem:interlacing_zeros} applied to $ u=p_+ $ and $ v=p_- $ we conclude, since $ p=p_+p_- $, the ordering of the zeros.

 Finally, using that $ \alpha_1\cdots\alpha_k=\beta_1\cdots\beta_k $ in \eqref{eq:proof_zeros_of_p:capital_phi}, it follows that $(-1)^k\Phi(0) = 2 $, so $ s_0=0 $.

\end{proof}

\section{Properties of the eigenvalues}\label{sec:properties_of_eigenvalues}

For the proofs of the statements in Section \ref{sec:global_local_properties} it is convenient to express $ \rho $ and $ F $ in local coordinates. We prove a few properties of $ \rho $ and $ F $, defined in \eqref{eq:sec_model:F} respectively \eqref{eq:sec_model:eigenvalue}, that will be used multiple times throughout the paper.

In local coordinates, away from the special points $ y_0,\dots,y_{2k'-1} \in \mathcal R $, $ \rho $ is given by 
\begin{equation}\label{eq:sec_eigenvalues:local_eigenvalue}
 \rho_1(z) = \frac{1}{2}\Tr \Phi(z) + \frac{q(z)p_0(z)^\frac{1}{2}}{2(z-1)^k},
\end{equation}
and
\begin{equation}
 \rho_2(z) = \frac{1}{2}\Tr \Phi(z) - \frac{q(z)p_0(z)^\frac{1}{2}}{2(z-1)^k},
\end{equation}
recall the definition of $ q $ and $ p_0^\frac{1}{2} $ given just after Proposition \ref{prop:zeros_of_p} respectively just before Proposition \ref{prop:L_map}. This definition of $ \rho_1 $ and $ \rho_2 $ coincide with the definition given just after \eqref{eq:sec_model:generalized_weight}, see Lemma \ref{lem:eigenvalues} \eqref{eq:lem_eigenvalues:inequality} below. Close to $ x_\ell $, $ \ell=0,\dots,2(k'-1) $ there are neighborhoods of $ 0 $ respectively $ x_\ell $ and a biholomorphic map $ \phi_{x_\ell} $ between these neighborhoods such that $ \phi_{x_\ell}(0) = x_\ell $ and $ Q(\phi_{x_\ell}(w),w) = 0 $ where $ \phi_{x_\ell} $ is defined. Since $ p_0 $ has a simple zero at $ x_\ell $ we obtain, by differentiate $ Q(\phi_{x_\ell}(w),w) = 0 $, that
\begin{equation}\label{eq:sec_properties:derivative_local_map}
 \phi_{x_\ell}'(0) = 0,
\end{equation}
and the zero is simple. We use this local coordinate to express $ \rho $ close to $ y_\ell $,
\begin{equation}
 \rho_{y_\ell}(w) = \frac{1}{2}\Tr \Phi(\phi_{x_\ell}(w)) + \frac{q(\phi_{x_\ell}(w))w}{2(\phi_{x_\ell}(w)-1)^k}.
\end{equation}
At $ y_{2k'-1} $ we set $ \tau = z^{-\frac{1}{2}} $. Then in a neighborhood of $ \tau = 0$ the map $ (z,w) \mapsto \tau $ is a  chart and the inverse is given by $ \tau \mapsto (\tau^{-2},\tau^{-(2k'-1)}\sigma(\tau)) $, where $ \sigma $ is analytic in a neighborhood of zero and $ \sigma(\tau)^2 = p_0(z)z^{-(2k'-1)} $. In this local coordinate $ \rho $ becomes
\begin{equation}
 \rho_{y_{2k-1}}(\tau) = \frac{1}{2}\Tr \Phi(\tau^{-2}) + \tau\frac{\tau^{2(k-k')} q(\tau^{-2})\sigma(\tau)}{2(1-\tau^2)^k}.
\end{equation}

Similarly we express $ F $ in local coordinates. Away from $ y_0,\dots,y_{2(k'-1)} \in \mathcal R $, $ F $ in local coordinates is given by
\begin{equation}\label{eq:sec_eigenvalues:F1}
 F_1(z;\chi,\eta)=\frac{k}{4}(\eta+1)\log z -\frac{k}{2}\log(z-1)-\frac{\chi}{2}\log \rho_1(z),
\end{equation}
and
\begin{equation}\label{eq:sec_eigenvalues:F2}
 F_2(z;\chi,\eta)=\frac{k}{4}(\eta+1)\log z -\frac{k}{2}\log(z-1)-\frac{\chi}{2}\log \rho_2(z).
\end{equation}
At $ y_\ell $ for $ \ell=0,\dots,2(k'-1) $,
\begin{equation}
 F_{y_\ell}(w;\chi,\eta)=\frac{k}{4}(\eta+1)\log \phi_{x_\ell}(w) -\frac{k}{2}\log(\phi_{x_\ell}(w)-1)-\frac{\chi}{2}\log \rho_{x_\ell}(w),
\end{equation}
 and
\begin{equation}
 F_{y_{2k'-1}}(\tau;\chi,\eta)=\frac{k}{4}(\eta+1)\log \tau^{-2} -\frac{k}{2}\log(\tau^{-2}-1)-\frac{\chi}{2}\log \rho_{x_{2k'-1}}(\tau),
\end{equation}
in a neighborhood of $ y_{2k'-1} $.

 In the following two lemmas important properties of the eigenvalues of $ \Phi $ are stated and proved. Recall that $ I_\ell = (x_{2\ell},x_{2\ell-1}) $.
  
 \begin{lemma}\label{lem:eigenvalues}
 The eigenvalues of $ \Phi $ have the following properties.
  \begin{enumerate}[(i)]
   \item For $ i=1,2 $, $ \rho_i(z) \to 1 $ as $ |z| \to \infty $ or $ z \to 0 $.\label{eq:lem_eigenvalues:at_infinity}
   \item For all $ z \in (-\infty,0)\backslash \cup_{\ell=1}^{k'-1} \overline{I_\ell} $, $ |\rho_i(z)| = 1 $ for $ i=1,2 $. \label{eq:lem_eigenvalues:at_cuts}
   \item For $ z \in \CC \backslash \left(\{1\}\cup(-\infty,0)\backslash \cup_{\ell=1}^{k'-1} \overline{I_\ell}\right) $, $ |\rho_2(z)| <1< |\rho_1(z)| $. \label{eq:lem_eigenvalues:inequality}
   \item For $ z \in \CC \backslash \left(\{1\}\cup(-\infty,0)\backslash \cup_{\ell=1}^{k'-1} \overline{I_\ell}\right) $, $ \rho_i(\overline z) = \overline{\rho_i(z)} $ for $ i = 1,2 $. \label{eq:lem_eigenvalues:conjugate}
  \end{enumerate}
 \end{lemma}

 \begin{proof}
  \begin{enumerate}[(i)]
   \item The statement follows by \eqref{eq:sec_model:weight} and \eqref{eq:proof_zeros_of_p:capital_phi}, since $ \Phi(z) $ tends to a lower respectively upper triangular matrix with ones on the diagonal as $ |z| \to \infty $ respectively $ z\to 0 $.
   \item Let $ z \in (-\infty,0)\backslash \cup_{\ell=1}^{k'-1} \overline{I_\ell} $. It follows from the definition of $ p_0^\frac{1}{2} $ that $ p_0(z)^\frac{1}{2} \in \i \RR $ which implies that $ \overline{\rho_1(z)} = \rho_2(z) $. Hence
   \begin{equation}
    |\rho_1(z)|^2=\rho_1(z)\overline{\rho_1(z)}=\rho_1(z)\rho_2(z) = \det \Phi(z) = 1.
   \end{equation}
   The statement for $ \rho_2 $ now follows since $ \rho_1\rho_2=1 $.
   \item We know that $ \rho_1 $ and $ \rho_2 $ are analytic and non-zero in $ \CC \backslash \left(\{1\}\cup(-\infty,0)\backslash \cup_{\ell=1}^{k'-1} \overline{I_\ell}\right) $ which implies that $ \rho_2/\rho_1 $ is analytic in $ \CC \backslash \left(\{1\}\cup(-\infty,0)\backslash \cup_{\ell=1}^{k'-1} \overline{I_\ell}\right) $. By \eqref{eq:lem_eigenvalues:at_infinity} and \eqref{eq:lem_eigenvalues:at_cuts} $ |\rho_2(z)|/|\rho_1(z)| \to 1 $ as $ z $ tends to the boundary away from one, while $ |\rho_2(z)|/|\rho_1(z)| \to 0 $ as $ z \to 1 $. It follows by the maximum principle for analytic functions that 
   \begin{equation}
    \left|\frac{\rho_2(z)}{\rho_1(z)}\right| < 1.
   \end{equation}
   The result now follows since $ \rho_1\rho_2 = 1 $. 
   \item Note first that for $ i=1,2 $,
   \begin{equation}
    \rho_i(\overline z)^2-\overline{\Tr \Phi(z)} \rho_i(\overline z)+1=0,
   \end{equation}
   that is $ \rho_i(\overline z) $ is a root to the characteristic polynomial of $ \overline{\Phi(z)} $. So $ \rho_i(\overline z) $ is equal to $ \overline{\rho_1(z)} $ or $ \overline{\rho_2(z)} $. If $ z>0 $ then $ \rho_i(\overline{z}) = \rho_i(z) = \overline{\rho_i(z)} $. By continuity and \eqref{eq:lem_eigenvalues:inequality} the statement follows.
  \end{enumerate}
 \end{proof}

\begin{lemma}\label{lem:logarithmic_derivative}
 For all $ z \in \CC \backslash \left((-\infty,0)\backslash \cup_{\ell=1}^{k'-1} \overline{I_\ell}\right) $,
 \begin{equation}\label{eq:lem_logarithmic_derivative:quotient}
  \frac{\rho_1'(z)}{\rho_1(z)}=-\frac{\rho_2'(z)}{\rho_2(z)},
 \end{equation}
 and
 \begin{equation}\label{eq:lem_logarithmic_derivative:log_derivative}
  \frac{\rho_1'(z)}{\rho_1(z)} = \frac{q(z)^{-1}(z-1)^{k+1}\Tr \Phi'(z)}{(z-1)p_0(z)^\frac{1}{2}},
 \end{equation}
 where
 \begin{equation}\label{eq:lem_logarithmic_derivative:derivative_trace}
  (z-1)^{k+1}\Tr \Phi'(z) = (z-1)\frac{\d}{\d z}\left((z-1)^k\Tr \Phi(z)\right)-k(z-1)^k\Tr \Phi(z),
 \end{equation}
 and $ q(z)^{-1}(z-1)^{k+1}\Tr \Phi'(z) $ is a polynomial of degree $ k'-1 $ which is non-zero at $ x_\ell $ for $ \ell=0,\dots,k'-1 $. Moreover $ \frac{\rho_1'(z)}{\rho_1(z)} \in \RR $ for $ z \in I_\ell $, $ \ell=0,\dots,k'-1 $. 
\end{lemma}
\begin{proof}
The first equality follows by differentiating the equality $ \rho_1(z)\rho_2(z) = 1 $. 
 
Recall the definition of $ q $ and $ p_0 $, namely that $ (z-1)^{2k}\left((\Tr\Phi(z))^2-4\right) = q(z)^2p_0(z) $. The second equality is obtained by differentiating
\begin{equation}
 \rho_1(z) = \frac{1}{2}\Tr \Phi(z) + \frac{1}{2}\left((\Tr\Phi(z))^2-4\right)^\frac{1}{2},
\end{equation}
to get 
\begin{equation}
 \rho_1'(z) = \frac{\Tr \Phi'(z)}{\left((\Tr\Phi(z))^2-4\right)^\frac{1}{2}}\rho_1(z),
\end{equation}
and then rewrite the denominator in terms of $ q $ and $ p_0 $.
 
By the product rule
\begin{equation}
 (z-1)^{k+1}\Tr \Phi'(z) = (z-1)\frac{\d}{\d z}\left((z-1)^k\Tr \Phi(z)\right)-k(z-1)^k\Tr \Phi(z),
\end{equation}
which is a polynomial of degree $ k-1 $, since $ (z-1)^k\Tr \Phi(z) $ is of degree $ k $ and the leading coefficient in the two terms on the right hand side is equal while the coefficient in front of $ z^{k-1} $ does not cancel out each other. Note also that the right hand side is equal to $ (z-1)p_\pm'(z)-kp_\pm(z) $, recall \eqref{eq:sec_zeros:ppm}. A root of $ q $ is a double root of either $ p_+ $ or $ p_- $, since $ p_+ $ and $ p_- $ has no common zeros. Therefore $ q $ divides $ (z-1)p_\pm'(z)-kp_\pm(z) $ and in particular $ q(z)^{-1}(z-1)^{k+1}\Tr \Phi'(z) $ is a polynomial of degree $ k-1-(k-k')=k'-1 $. Now, a zero of $ p_0 $ is a simple zero of $ p_+ $ or $ p_- $, so
\begin{equation}
 (x_\ell-1)^{k+1}\Tr \Phi'(x_\ell) = (x_\ell-1)p_\pm'(x_\ell) \neq 0.
\end{equation}
The last statement follows directly from \eqref{eq:lem_logarithmic_derivative:log_derivative} since $ p_0(z)^\frac{1}{2} \in \RR $ for $ z\in I_\ell $. 
\end{proof}

\section{Proof of Lemmas \ref{lem:critical_points} and \ref{lem:zeros_on_loops}}\label{sec:proof_phases}

In this section we prove Lemmas \ref{lem:critical_points} and \ref{lem:zeros_on_loops}. Lemma \ref{lem:critical_points} is a direct corollary of the following lemma.

\begin{lemma}\label{lem:projected_critical_points}
 Consider the function,
 \begin{equation}\label{eq:polynomial_of_critical_points}
  P(z;\chi,\eta) = k^2\frac{p_0(z)}{z}(z(\eta-1)-(\eta+1))^2-4\chi^2z\left(\frac{(z-1)^{k+1}\Tr \Phi'(z)}{q(z)}\right)^2.
 \end{equation}
 If the point $ (z,w) \in \mathcal R $ is a critical point of $ F(z,w;\chi,\eta) $, defined in \eqref{eq:sec_model:F}, then $ z $ is a root of the polynomial above. Conversely, if $ z $ is a solution of the above polynomial then there is a $ w $ such that $ (z,w) \in \mathcal R $ is a critical point of $ F(z,w;\chi,\eta) $. Moreover the order of the zero and the order of the critical point is the same.
\end{lemma}

\begin{remark}
 That the order is the same in case both $ (z,w) $ and $ (z,-w) $ are critical points should be interpreted that the sum of the orders is the same as the order of the zero of $ P $. 
\end{remark}

\begin{remark}
 The function $ P $ is a polynomial of degree $ 2k' $, by Lemma \ref{lem:logarithmic_derivative} and since the degree of $ \frac{p_0(z)}{z} $ is $ 2k'-2 $.
\end{remark}

\begin{proof}
 Let $ (z,w) \in \mathcal R $. The proofs goes by showing that if $ (z,w) $ is a critical point of $ F $ then $ z $ is a zero of $ P $, and these are the only zeros of $ P $. 
 
 The point $ (z,w)\neq y_\ell $, for $ \ell=0,\dots,2k'-1 $, is a critical point of $ F $ if and only if $ z $ is a critical point of $ F_1 $ or $ F_2 $. The derivative of $ F_i $ is given by
 \begin{equation}
  F_i'(z;\chi,\eta)=\frac{k}{4}(\eta+1)\frac{1}{z} -\frac{k}{2}\frac{1}{z-1}-\frac{\chi}{2}\frac{\rho_i'(z)}{ \rho_i(z)},
 \end{equation}
 for $ i=1,2 $. From \eqref{eq:lem_logarithmic_derivative:quotient} and \eqref{eq:lem_logarithmic_derivative:log_derivative} we derive
 \begin{multline}
  F_1'(z;\chi,\eta)F_2'(z;\chi,\eta) \\= \frac{k^2}{16z^2(z-1)^2}(z(\eta-1)-(\eta+1))^2 
  -\frac{\chi^2}{4(z-1)^2}\frac{\left(\frac{(z-1)^{k+1}}{q(z)}\Tr \Phi'(z)\right)^2}{p_0(z)}.
 \end{multline}
By multiplying the equation with $ 16z(z-1)^2p_0(z) $ we obtain on the right hand side the polynomial
\begin{equation}
 P(z;\chi,\eta) = k^2\frac{p_0(z)}{z}(z(\eta-1)-(\eta+1))^2-4\chi^2z\left(\frac{(z-1)^{k+1}\Tr \Phi'(z)}{q(z)}\right)^2.
\end{equation}
By Lemma \ref{lem:logarithmic_derivative} it is clear that $ P $ is a polynomial. Hence the critical points of $ F_1 $ and $ F_2 $ are a subset of the zeros of $ P $. 

For $ \ell=0,\dots,2k'-1 $ the point $ y_\ell $ is a critical points of $ F $ if and only if $ 0 $ is a critical point of $ F_{y_\ell} $. Close to $ y_\ell $, $ \ell \neq 2k'-1 $,
 \begin{equation}
  F_{y_\ell}'(w;\chi,\eta)=\frac{k}{4}(\eta+1)\frac{\phi_{x_\ell}'(w)}{\phi_{x_\ell}(w)} -\frac{k}{2}\frac{\phi_{x_\ell}'(w)}{\phi_{x_\ell}(w)-1}-\frac{\chi}{2}\frac{\rho_{y_\ell}'(w)}{\rho_{y_\ell}(w)}.
 \end{equation}
 It holds that $ \rho_{y_\ell}(0) = \pm 1 $ by \eqref{eq:sec_zeros:ppm} where the sign depends on if $ x_\ell $ is a zero of $ p_+ $ or $ p_- $ and $ \rho_{y_\ell}'(0) = \frac{1}{2}q(x_\ell)(x_\ell-1)^{-k}\neq 0 $. If $ \ell = 0 $, then $ F_{y_0}'(0;\chi,\eta)=\infty $ and in particular non-zero. If $ \ell \neq  0 $ then, since $ \phi_{x_\ell}'(0)=0 $ and the zero is simple by \eqref{eq:sec_properties:derivative_local_map}, there is a zero of $ F_{y_\ell}' $ of order one or three at $ x_\ell $ if and only if $ \chi=0 $. The order is one if $ \frac{\eta+1}{\eta-1}\neq x_\ell $ and three if $ \frac{\eta+1}{\eta-1}=x_\ell $, which is seen by differentiating $ F'_{y_\ell} $. Such critical point is also a zero of $ P $ with the same order. If $ \ell=2k'-1 $ then a computation shows that $ \rho_{y_{2k'-1}}(0) = \pm 1 $ and $ \rho_{y_{2k'-1}}'(0) = \frac{\sigma(0)}{2} \neq 0 $. It follows that with $ \eta \in (-1,1) $, $ F_{y_{2k'-1}}'(0;\chi,\eta) = \infty $ and in particular non-zero. 
 
 Hence the critical points of $ F $ when projected down to $ \CC $ are zeros of $ P $. To see that they are precisely the zeros of $ P $, we need to check that we did not create any extra zeros in the derivation of $ P $ at $ x_\ell $, $ \ell=0,\dots,2(k'-1) $ or at $ z=1 $. There is no zero at $ z=1 $, namely
\begin{equation}
 P(1;\chi,\eta) = 4k^2p_0(1)(1-\chi^2) \neq 0.
\end{equation}
Recall that $ p_0 $ has a simple zero at $ x_0=0 $, so
\begin{equation}
 P(x_0;\chi,\eta) = k^2\frac{p_0(x_0)}{x_0}(\eta+1)^2 \neq 0.
\end{equation}
When $ \chi\neq 0 $ and $ \ell=1,\dots,2(k'-1) $, Lemma \ref{lem:logarithmic_derivative} implies that
\begin{equation}
 P(x_\ell;\chi,\eta) = -4\chi^2x_\ell\left(\frac{(x_\ell-1)^{k+1}\Tr \Phi'(x_\ell)}{q(x_\ell)}\right)^2 \neq 0.
\end{equation}
Hence, we did not create any fake zeros, that is, all zeros of $ P $ corresponds to a critical point of $ F $.
\end{proof}
 
\begin{proof}[Proof of Lemma \ref{lem:zeros_on_loops}]
 Cover $ C^+_\ell $ by four different charts, $ \psi_i $, $ i=1,\dots,4 $. Let $ C_i $ be a curve in the domain of $ \psi_i $ such that $ C^+_\ell = C_1\cup\dots\cup C_4 $. Let also $ \psi_i $ be such that $ \psi_i(C_i) \subset \RR $ and such that increasing along $ \psi_i(C_i) $ corresponds to going in the same direction in $ C^+_\ell $ for all $ i $. Consider the integral
 \begin{equation}
  \sum_{i=1}^4 \int_{\psi_i(C_i)} \frac{\d}{\d z}\left(F\circ \psi_i^{-1}\right)(z)\d z.
 \end{equation}
 Since $ C^+_\ell $ is a loop the sum is zero. By the choice of $ \psi_i $
 \begin{equation}
  \int_{\psi_i(C_i)} \frac{\d}{\d z}\left(F\circ \psi_i^{-1}\right)(z)\d z = \int_{a_i}^{b_i} \frac{\d}{\d t}\left(F\circ \psi_i^{-1}\right)(t)\d t, 
 \end{equation}
 for some $ a_i<b_i $. Since $ \rho $ is real on $ C_\ell^+ $ and therefore $ \im F $ is constant on $ C_\ell^+ $ we get that $ \frac{\d}{\d t}\left(F\circ \psi_i^{-1}\right)(t) \in \RR $. Hence $ \frac{\d}{\d t}\left(F\circ \psi_i^{-1}\right) $ has to change sign for at least two $ t \in \cup[a_i,b_i] $. By continuity and since $ \frac{\d}{\d t}\left(\psi_j \circ \psi_i^{-1}\right)(t)>0 $ when $ t $ is real and in the domain of definition of $ \psi_j \circ \psi_i^{-1} $, we conclude that
 \begin{equation}
  \frac{\d}{\d z}\left(F\circ \psi_i^{-1}\right)(z)=0,
 \end{equation}
 at two distinct points. That is, $ F $ has at least two critical points in $ C^+_\ell $.
\end{proof}

\section{Proof of Propositions \ref{prop:boundary}, \ref{prop:L_map} and \ref{prop:boundary_is_boundary}}\label{sec:proof_arctic_curves}
To prove Propositions \ref{prop:boundary}, \ref{prop:L_map} and \ref{prop:boundary_is_boundary} we first prove two technical lemmas to make the rest of the proofs more readable.

 Let $ I_0 = (0,\infty)\backslash\{1\} $ and recall that $ I_\ell = (x_{2\ell},x_{2\ell-1}) $ for $ \ell=1,\dots,k'-1 $.
\begin{lemma}\label{lem:surjective}
 For $ \ell=1,\dots,k'-1, $ the map $ I_\ell \ni z \mapsto \frac{z\rho_1'(z)}{\rho_1(z)} \in \RR $ is surjective and the map $ I_0 \ni z \mapsto \frac{z\rho_1'(z)}{\rho_1(z)} \in \RR\backslash\{0\} $ is surjective. Moreover
 \begin{equation}\label{eq:lem_surjective:limit_points}
  \lim_{z\to 0}\frac{z\rho_1'(z)}{\rho_1(z)} = \lim_{z\to \infty}\frac{z\rho_1'(z)}{\rho_1(z)} = 0,
 \end{equation}
 and
 \begin{equation}\label{eq:lem_surjective:at_one}
  (z-1)\frac{z\rho_1'(z)}{\rho_1(z)}|_{z=1} = -k,
 \end{equation}
 which implies that
 \begin{equation}\label{eq:lem_surjective:pole_at_one}
  \lim_{I_0 \ni z\to 1^{\pm}}\frac{z\rho_1'(z)}{\rho_1(z)} = \mp \infty.
 \end{equation}
\end{lemma}
\begin{proof}
 We show that for each $ \ell=1,\dots,k'-1 $ the function $ \rho_1'/\rho_1 $ is surjective when viewed as a function from $ I_\ell $ to $ \RR $. It then follows, since $ 0 \not\in \overline I_\ell $ and since $ \rho_1'/\rho_1 $ is continuous, that the function $ I_\ell\ni z\mapsto \frac{z\rho_1'(z)}{\rho_1(z)} \in \RR $ is surjective. 
 
 A proof as the proof of Lemma \ref{lem:zeros_on_loops} shows that $ \rho $ has at least two critical points in each $ C^+_\ell $ for $ \ell=1,\dots,k'-1 $. The critical points are not at the points $ y_j $ since $ \rho_{x_j}'(0) = \frac{q(x_\ell)}{2(x_j-1)^k} \neq 0 $. By \eqref{eq:lem_logarithmic_derivative:quotient} there is at least one zero of $ \frac{\rho_1'(z)}{\rho_1(z)} $ in $ I_\ell $. However, since the numerator of the right hand side of \eqref{eq:lem_logarithmic_derivative:log_derivative} is a polynomial of degree $ k'-1 $, the function $ \frac{\rho_1'(z)}{\rho_1(z)} $ has at most one zero and hence exactly one zero, counting with multiplicity, in each $ I_\ell $. 
 
 By Lemma \ref{lem:logarithmic_derivative} it follows that
 \begin{equation}
  \left|\frac{\rho_1'(z)}{\rho_1(z)}\right| \to \infty,
 \end{equation}
 when $ z \to x_{2\ell} $ or $ z \to x_{2\ell-1} $. Since $ \frac{\rho_1'(z)}{\rho_1(z)} \in \RR $, is continuous and has exactly one root in $ I_\ell $, we conclude that the function $ I_\ell\ni z\mapsto \frac{\rho_1'(z)}{\rho_1(z)} \in \RR $ is surjective.

 We go on with the behavior on $ I_0 $. We have seen that all zeros of $ \frac{\rho_1'(z)}{\rho_1(z)} $ are located in $ I_\ell $ for $ \ell=1,\dots,k'-1 $, so the function $ z\mapsto \frac{z\rho_1'(z)}{\rho_1(z)} $ from $ I_0 $ to $ \RR\backslash \{0\} $ is well defined. It follows from \eqref{eq:sec_model:def_p}, \eqref{eq:lem_logarithmic_derivative:log_derivative} and \eqref{eq:lem_logarithmic_derivative:derivative_trace} that
 \begin{equation}
  (z-1)\frac{z\rho_1'(z)}{\rho_1(z)}|_{z=1} = -k\frac{(z-1)^k\Tr \Phi(z)|_{z=1}}{q(1)p_0(1)^\frac{1}{2}} = -k,
 \end{equation}
 so \eqref{eq:lem_surjective:pole_at_one} holds. Since $ p_0 $ has a simple zero at $ z=0 $ and is of degree $ 2k'-1 $, \eqref{eq:lem_surjective:limit_points} follows from Lemma \ref{lem:logarithmic_derivative}. By continuity the map $ I_0 \ni z \mapsto \frac{z\rho_1'(z)}{\rho_1(z)} \in \RR\backslash\{0\} $ is surjective.
\end{proof}

\begin{lemma}\label{lem::technical}
 For all 
 \begin{enumerate}[(i)]
  \item $ z \in \CC \backslash \left(\bigcup_{\ell=0}^{k'-1}\overline{I_\ell}\right) $ it holds that
 \begin{equation}
  \im \left(\frac{z\rho_1'(z)}{\rho_1(z)}\right) \neq 0,
 \end{equation} \label{eq:lem_technical:0}
  \item $ z \in I_\ell $, $ \ell=0,\dots, k'-1 $, it holds that
  \begin{equation}
   \frac{\d}{\d z}\left(\frac{z\rho_1'(z)}{\rho_1(z)}\right) \neq 0,
  \end{equation} \label{eq:lem_technical:1}
  \item $ z \in I_0 $ it holds that
  \begin{equation}
   \frac{\d}{\d z}\left(\frac{z\rho_1'(z)}{\rho_1(z)}+\frac{k}{z-1}\right) > 0,
  \end{equation}\label{eq:lem_technical:2}
  \item $ z \in I_0 $ it holds that
  \begin{equation}
   \frac{\d}{\d z}\left((z-1)\frac{\rho_1'(z)}{\rho_1(z)}\right) \neq 0,
  \end{equation}\label{eq:lem_technical:3}
  \item $ z \in I_0 $ it holds that
  \begin{equation}
   \frac{\d }{\d z}\left((z-1)\frac{z\rho_1'(z)}{\rho_1(z)}\right) \neq 0.
  \end{equation}\label{eq:lem_technical:4}
 \end{enumerate}
\end{lemma}
\begin{proof}
 The proofs below follow the same strategy, namely locating all possible roots, by using Lemma \ref{lem:surjective}, of an appropriate polynomial equation. 
 \begin{enumerate}[(i)]
  \item 
  Assume that
 \begin{equation}\label{eq:lem_technical:imaginary_part}
  \im \left(\frac{z\rho_1'(z)}{\rho_1(z)}\right) = 0,
 \end{equation}
 for some $ z \in \CC \backslash \left(\bigcup_{\ell=0}^{k-1}\overline{I_\ell}\right) $, that is $ \frac{z\rho_1'(z)}{\rho_1(z)} = t $ for some $ t \in \RR $. By \eqref{eq:lem_logarithmic_derivative:log_derivative} the above assumption implies that
 \begin{equation}\label{eq:lem_technical_1:polynomial}
  z\left(\frac{(z-1)^{k+1}\Tr \Phi'(z)}{q(z)}\right)^2-t^2(z-1)^2\frac{p_0(z)}{z} = 0,
 \end{equation}
 where the left hand side is a polynomial of degree $ 2k' $ if $ t\neq 0 $ and of degree $ 2k'-1 $ if $ t=0 $.  We show that all roots of the polynomial lie in $ I_\ell $, $ \ell=0,\dots,k'-1 $, which tells us that no zero can come from \eqref{eq:lem_technical:imaginary_part}. If $ t\neq 0 $ then the above polynomial has two roots in each $ I_\ell $, $ \ell=0,\dots,k'-1 $ coming from $ \frac{z\rho_1'(z)}{\rho_1(z)} = \pm t $, by Lemma \ref{lem:surjective}. If $ t=0 $ then there is a double root in each $ I_\ell $, $ \ell=1,\dots,k'-1 $ of \eqref{eq:lem_technical_1:polynomial} coming from $ \frac{z\rho_1'(z)}{\rho_1(z)} = 0 $ and one root $ z=0 $.
 
  \item Assume there is a $ z \in I_\ell $ such that
  \begin{equation}
   \frac{\d}{\d z}\left(\frac{z\rho_1'(z)}{\rho_1(z)}\right) = 0.
  \end{equation}
  Then there is a level curve satisfying
  \begin{equation}
   \im \left(\frac{z\rho_1'(z)}{\rho_1(z)}\right)=0,
  \end{equation}
  leaving the real line, which contradicts \eqref{eq:lem_technical:0}.

  \item Let $ t \in \RR $ and consider the equation
  \begin{equation}\label{eq:lem_technical_2:equation}
   \frac{z\rho_1'(z)}{\rho_1(z)}+\frac{k}{z-1} = t.
  \end{equation}
  The above equality implies that $
   \left(\frac{z\rho_1'(z)}{\rho_1(z)}\right)^2 = \left(t-\frac{k}{z-1}\right)^2,
  $ which by \eqref{eq:lem_logarithmic_derivative:log_derivative} is the same as solving the polynomial equation
  \begin{equation}\label{eq:lem_technical_2:polynomial}
   z\left(\frac{(z-1)^{k+1}\Tr \Phi'(z)}{q(z)}\right)^2 - \frac{p_0(z)}{z}(tz-(k+t))^2=0.
  \end{equation}
  The polynomial on the left hand side is of degree $ 2k' $ if $ t \neq 0 $ and $ 2k'-1 $ if $ t=0 $. By Lemma \ref{lem:surjective} there are two solutions on each interval $ I_\ell $ for $ \ell=1,\dots,k'-1 $, one coming from \eqref{eq:lem_technical_2:equation} and the other coming from \eqref{eq:lem_technical_2:equation} with $ \frac{z\rho_1'(z)}{\rho_1(z)} $ interchanged with $ -\frac{z\rho_1(z)}{\rho_1(z)} $. If these two solutions happen to coincide, the left most term in \eqref{eq:lem_technical_2:polynomial} is zero and the order of the zero is two, and since $ p_0(z)\neq 0 $ it implies that $ tz-(k+t)=0 $, and hence it is a double root of \eqref{eq:lem_technical_2:polynomial}. Hence there can at most be two solutions of \eqref{eq:lem_technical_2:equation} on the positive part of the real line. In fact, there can at most be one solution if $ t \not\in (-k,0] $. Namely, by \eqref{eq:lem_surjective:limit_points} and \eqref{eq:lem_surjective:pole_at_one}, 
  \begin{equation}\label{eq:lem_technical_2:equation_2}
   -\frac{z\rho_1'(z)}{\rho_1(z)}+\frac{k}{z-1} = t
  \end{equation}
  has a solution in $ [0,1) $ if $ t \in (-\infty,-k] $ and a solution in $ (1,\infty) $ if $ t \in (0,\infty) $. By Lemma \ref{lem:surjective} the equation $ \frac{z\rho_1'(z)}{\rho_1(z)} = 0 $ has no solution in $ I_0 $. So a solution to \eqref{eq:lem_technical_2:equation_2} can not be a solution for \eqref{eq:lem_technical_2:equation}. So there is at most one solution left for \eqref{eq:lem_technical_2:equation}.

  To summarize, the equation \eqref{eq:lem_technical_2:equation} has at most one respectively two solutions in $ I_0 $ if $ t \in \RR\backslash (-k,0] $ respectively $ t \in (-k,0] $.
    
  Now, by \eqref{eq:lem_surjective:at_one}, the function
  \begin{equation}\label{eq:lem_technical:injective_function}
   \frac{z\rho_1'(z)}{\rho_1(z)}+\frac{k}{z-1},
  \end{equation}
  is analytic in a neighborhood of $ (0,\infty) $, and by \eqref{eq:lem_surjective:limit_points}
  \begin{equation}
   \lim_{z\to0}\frac{z\rho_1'(z)}{\rho_1(z)}+\frac{k}{z-1} = -k \text{ and } \lim_{z\to \infty} \frac{z\rho_1'(z)}{\rho_1(z)}+\frac{k}{z-1} = 0.
  \end{equation}
  Moreover \eqref{eq:lem_technical:injective_function} is real in $ I_0 $. Hence, by the number of solutions of \eqref{eq:lem_technical_2:equation} in $ I_0 $ which is counted with multiplicity, we conclude that \eqref{eq:lem_technical:injective_function} is an increasing function with no critical point in $ I_0 $. That is \eqref{eq:lem_technical:2} holds.
  \item Assume for a contradiction that \eqref{eq:lem_technical:3} does not hold. Then there is a $ t \in \RR $ such that the equation
  \begin{equation}\label{eq:lem_technical_3:equation}
   (z-1)\frac{\rho_1'(z)}{\rho_1(z)} = t,
  \end{equation}
  has a solution of order two in $ I_0 $. By Lemma \ref{lem:surjective}, $ t\neq 0 $. It implies that the polynomial equation
  \begin{equation}
   \left(\frac{(z-1)^{k+1}\Tr \Phi'(z)}{q(z)}\right)^2-t^2p_0(z) = 0,
  \end{equation}
  has a solution of order two in $ I_0 $. The polynomial equation has $ 2k'-1 $ solutions, counting with multiplicity. There are two solutions in each $ I_\ell $, $ \ell=1,\dots,2(k'-1) $, since \eqref{eq:lem_technical_3:equation} has one solution in $ I_\ell $ and the same equation with $ t $ interchanged with $ -t $ has one different solution in $ I_\ell $, by Lemma \ref{lem:surjective}. Hence there are not enough possible solutions of the polynomial equation to have a solution of order two of \eqref{eq:lem_technical_3:equation} in $ I_0 $, which proves the statement. 
  \item This proof is almost identical to the previous one. Assume there is a $ t \in \RR $ such that
  \begin{equation}
   (z-1)\frac{z\rho_1'(z)}{\rho_1(z)} = t,
  \end{equation}
   has a solution of order two in $ I_0 $, and therefore
  \begin{equation}
   z\left(\frac{(z-1)^{k+1}\Tr \Phi'(z)}{q(z)}\right)^2-t^2\frac{p_0(z)}{z} = 0,
  \end{equation}
  has a solution of order two in $ I_0 $. The polynomial equation has $ 2k'-1 $ solutions. As before we know there are $ 2(k'-1) $ solutions in $ \cup_{\ell=1}^{k'-1}I_\ell $. Hence there can not be a solution of order two in $ I_0 $. 
 \end{enumerate}
\end{proof}

With the above lemmas we give the proofs of Proposition \ref{prop:boundary}, \ref{prop:L_map} and \ref{prop:boundary_is_boundary}.

\begin{proof}[Proof of Proposition \ref{prop:boundary}]
First, let $ (\chi,\eta) \in \mathcal E_{RS} $, $ \ell=1,\dots,k'-1 $ then, by \eqref{eq:polynomial_of_critical_points}, $ F $ has a critical point of order (at least) two at $ x_i $, $i=2\ell,2\ell-1 $, if and only if $ (\chi,\eta)=\left(0,\frac{x_i+1}{x_i-1}\right) $.

If $ (\chi,\eta) \in \mathcal E_{FR}\cap(-1,1)^2 $ ($ (\chi,\eta) \in \mathcal E_{RS}^{(\ell)} $, not one of the points mentioned above, for some $ \ell=1,\dots,k'-1 $) then there is a $ z \in I_0 $ ($ z \in I_\ell$) such that
\begin{equation}\label{eq:prop_SL_boundary:positive}
 A(z) 
 \begin{pmatrix}
  \eta \\
  \chi
 \end{pmatrix}
 = B(z),
\end{equation}
or
\begin{equation}\label{eq:prop_SL_boundary:negative}
 A(z) 
 \begin{pmatrix}
  \eta \\
  -\chi
 \end{pmatrix}
 = B(z),
\end{equation}
where
\begin{equation}
 A(z) = 
 \begin{pmatrix}
  \frac{k}{4} & -\frac{1}{2}\frac{z\rho_1'(z)}{\rho_1(z)} \\
  0 & -\frac{\d}{\d z}\left(\frac{z\rho_1'(z)}{\rho_1(z)}\right)
 \end{pmatrix}
\end{equation}
and 
\begin{equation}
 B(z) = 
 \begin{pmatrix}
  \frac{k}{4}\frac{z+1}{(z-1)} \\
  -\frac{k}{(z-1)^2}
 \end{pmatrix}.
\end{equation}
Here we have used \eqref{eq:lem_logarithmic_derivative:quotient}. The above is a linear equation that we solve for $ (\chi,\eta) $ in terms of $ z \in I_0 $ ($z\in I_\ell$). Now, 
\begin{equation}
 \det A(z) = \frac{k}{4} \frac{\d}{\d z} \left(\frac{z\rho_1'(z)}{\rho_1(z)}\right),
\end{equation}
which is non-zero on $ I_0 $ ($ I_\ell $) by Lemma \ref{lem::technical} \eqref{eq:lem_technical:1}. Hence there is a unique solution, not necessary in $ (-1,1)^2 $, to the equations
\begin{equation}
 A(z) 
 \begin{pmatrix}
  \eta \\
  \chi 
 \end{pmatrix}
 = B(z) \text{ and } A(z) 
 \begin{pmatrix}
  \eta \\
  -\chi 
 \end{pmatrix}
 = B(z).
\end{equation}

These solutions are given by the parametrization
\begin{equation}\label{eq:proof_boundary:parametrization_s}
 \chi(z) = \pm\frac{k}{(z-1)^2\frac{\d}{\d z}\left(\frac{z\rho_1'(z)}{\rho_1(z)}\right)}
\end{equation}
and
\begin{multline}\label{eq:proof_boundary:parametrization_eta}
 \eta(z) = \frac{1}{z-1}\left(\pm\frac{2(z-1)}{k}\frac{z\rho_1'(z)}{\rho_1(z)}\chi(z)+z+1\right) \\
 = \frac{1}{z-1}\left(\frac{2\frac{z\rho_1'(z)}{\rho_1(z)}}{(z-1)\frac{\d}{\d z}\left(\frac{z\rho_1'(z)}{\rho_1(z)}\right)}+z+1\right),
\end{multline}
for $ z \in I_0 $ ($ I_\ell $), where the sign depends on if it is a solution to \eqref{eq:prop_SL_boundary:positive} or \eqref{eq:prop_SL_boundary:negative}. This parametrization directly shows the symmetry in the line $ \chi=0 $.

To show that $ \mathcal E_{FR} $ ($\mathcal E_{RS}^{(\ell)}$) is a closed curve we show that the parametrization is continuous also at the endpoints of the intervals, with the limit being the appropriate point in the statement of the proposition, and that $ (\chi(z),\eta(z)) \in (-1,1)^2 $ when $ z\in I_0 $ ($ I_\ell$). We start to show that this is true for $ \mathcal E_{FR} $.

From the above parameterization we get by differentiating \eqref{eq:lem_logarithmic_derivative:log_derivative} that
\begin{equation}
 \lim_{z\to 0}(\chi(z),\eta(z)) = (0,-1) \text{ and } \lim_{z\to\infty}(\chi(z),\eta(z)) = (0,1).
\end{equation}
When $ z \to 1 $ the parametrization is actually still valid, however this requires some computations. Consider the parametrization with the $ + $-sign in front. Using \eqref{eq:sec_model:def_p}  and \eqref{eq:lem_logarithmic_derivative:derivative_trace} we get the following identities
\begin{multline}
 \left.(z-1)^k\Tr \Phi(z)\right|_{z=1}=p(1)^\frac{1}{2},\quad \left.\frac{\d}{\d z}\left((z-1)^k\Tr \Phi(z)\right)\right|_{z=1}=\frac{p'(1)}{2p(1)^\frac{1}{2}}, \\
 \left.(z-1)^{k+1}\Tr \Phi'(z)\right|_{z=1} = -kp(1)^\frac{1}{2}, \\ 
 \left.\frac{\d}{\d z}\left((z-1)^{k+1}\Tr \Phi'(z)\right)\right|_{z=1}=(1-k)\frac{p'(1)}{2p(1)^\frac{1}{2}}.
\end{multline}
These identities together with \eqref{eq:lem_logarithmic_derivative:log_derivative} implies
\begin{equation}
 \left.(z-1)^2\frac{\d}{\d z}\left(\frac{z\rho_1'(z)}{\rho_1(z)}\right)\right|_{z=1} = k,
\end{equation}
\begin{equation}
 \left.\frac{\d }{\d z}\left((z-1)^2\frac{\d}{\d z}\left(\frac{z\rho_1'(z)}{\rho_1(z)}\right)\right)\right|_{z=1} = 0,
\end{equation}
\begin{equation}
 \left.(z-1)\frac{z\rho_1'(z)}{\rho_1(z)}\right|_{z=1}=-k
\end{equation}
and
\begin{equation}
 \left.\frac{\d}{\d z}\left((z-1)\frac{z\rho_1'(z)}{\rho_1(z)}\right)\right|_{z=1} = -k+\frac{p'(1)}{p(1)}.
\end{equation}
Expand $ \chi $ and $ \eta $ around $ z=1 $ into its Taylor series, using these identities. We obtain, for the branch with the $ + $-sign,
\begin{equation}
 \chi(z) = 1 +\Ordo\left((z-1)^2\right),
\end{equation}
and
\begin{equation}
 \eta(z) = \frac{1}{z-1}\left(\frac{2(z-1)}{k}\frac{z\rho_1'(z)}{\rho_1(z)}\chi(z)+z+1\right) = \frac{2}{k}\frac{p'(1)}{p(1)}-1+\Ordo(z-1).
\end{equation}
That is, the parametrization $ z\to (\chi(z),\eta(z)) $ is extended over $ z=1 $ and $ \chi(1) = 1 $ and $ \eta(1) = \frac{2}{k}\frac{p'(1)}{p(1)}-1 $. For the other parametrization we get by symmetry that $ \chi(1) = -1 $ and $ \eta(1) = \frac{2}{k}\frac{p'(1)}{p(1)}-1 $. Hence the curve is a closed curve.

What is left to show is $ z \in I_0 $ implies $ (\chi,\eta) \in (-1,1)^2 $.

That $ \chi \in (-1,1) $, follows if
\begin{equation}
 (z-1)^2\frac{\d}{\d z}\left(\frac{z\rho_1'(z)}{\rho_1(z)}\right) > k 
\end{equation}
for $ z \in I_0 $. Rewrite this inequality to
\begin{equation}
 \frac{\d}{\d z}\left(\frac{z\rho_1'(z)}{\rho_1(z)}+\frac{k}{z-1}\right) > 0,
\end{equation}
which is Lemma \ref{lem::technical} \eqref{eq:lem_technical:2}.

To see that $ \eta \in (-1,1) $ assume for a contradiction that $ \eta(z) = -1 $ for some $ z \in I_0 $. We will later do the same with $ \eta(z) = 1 $. Solving for $ \chi $ in \eqref{eq:proof_boundary:parametrization_eta} implies
\begin{equation}
 \chi(z) = \mp\frac{kz}{(z-1)\frac{z\rho_1'(z)}{\rho_1(z)}}.
\end{equation}
But we already have \eqref{eq:proof_boundary:parametrization_s}, so
\begin{equation}
 \frac{k}{(z-1)^2\frac{\d}{\d z}\left(\frac{z\rho_1'(z)}{\rho_1(z)}\right)} = -\frac{kz}{(z-1)\frac{z\rho_1'(z)}{\rho_1(z)}}.
\end{equation}
This equation is equivalent to
\begin{equation}
 \frac{\d}{\d z}\left((z-1)\frac{\rho_1'(z)}{\rho_1(z)}\right) = 0,
\end{equation}
which is false by Lemma \ref{lem::technical} \eqref{eq:lem_technical:3}.

Similarly, assume that $ \eta(z) = 1 $. Then
\begin{equation}
 \chi(z) = -\frac{k}{(z-1)\frac{z\rho_1'(z)}{\rho_1(z)}},
\end{equation}
which by \eqref{eq:proof_boundary:parametrization_s} is equivalent to
\begin{equation}
 \frac{\d }{\d z}\left((z-1)\frac{z\rho_1'(z)}{\rho_1(z)}\right) = 0,
\end{equation}
which is false by Lemma \ref{lem::technical} \eqref{eq:lem_technical:4}.

Since $ (0,1) $ and $ (0,-1) $ is contained in $ \mathcal E_{FR} $ which is a closed curve we conclude that $ \mathcal E_{FR} $ is contained in $ [-1,1]^2 $.

Note also for further references that $ \mathcal E_{FR} $ is symmetric with respect to the line $ \chi=0 $ and it intersect that line only at $ (0,1) $ and $ (0,-1) $.

We will now continue with $ \mathcal E_{RS} $ for $ \ell=1,\dots,k'-1 $. Let $ z\mapsto (\chi(z),\eta(z)) $ be the parametrization given above. Then
\begin{equation}
 (\chi(z),\eta(z)) \to \left(0,\frac{x_j+1}{x_j-1}\right)
\end{equation}
as $ z \to x_j $, for $ j=2\ell $ and $ 2\ell-1 $. The limit follows by differentiating \eqref{eq:lem_logarithmic_derivative:log_derivative}. This shows that $ (\chi,\eta)(\overline{I_\ell}) $ for $ \ell=1,\dots,k'-1 $ are closed curves. What is left to show is that the curves are contained in $ (-1,1)^2 $. For any $ \ell =1,\dots,k'-1 $ the sets $ \mathcal E_{FR} $ and $ \mathcal E_{RS}^{(\ell)} $ are disjoint. Since $ \mathcal E_{FR} $ is a closed curve and does only intersect the line $ \chi=0 $ at $ (0,\pm 1) $ while $ \mathcal E_{RS}^{(\ell)} $ intersect the point $ (0,\frac{x_j+1}{x_j-1}) $ we conclude that $ \mathcal E_{RS} $ is contained in the interior of $ \mathcal E_{FR} $ and hence in $ (-1,1)^2 $.

That $ \mathcal E_{FR} $ and $ \mathcal E_{RS} $ are simple follows directly by the definition of the curves and Lemma \ref{lem:zeros_on_loops}.
\end{proof}

\begin{proof}[Proof of Proposition \ref{prop:L_map}]
 It follows from Lemma \ref{lem:eigenvalues} \eqref{eq:lem_eigenvalues:conjugate} that if $ (z,w) \in \mathcal R \backslash C^+ $ is a critical point of $ F $, then so is $ (\bar z,\bar w) $. Moreover, for any $ (z,w) \in \mathcal R \backslash C^+ $ the point $(z,w) $ is in $ \mathcal R_{12} $ if and only if $ (\bar z,\bar w) \not\in \mathcal R_{12} $. Hence $ L $ is well defined.

 For $ z \in \CC \backslash \left(\cup_{\ell=0}^{k'-1}\overline{I_\ell}\right) $ let
 \begin{equation}
  A(z) = 
  \begin{pmatrix}
   -\frac{1}{2}\re \left(\frac{\rho_i'(z)z}{\rho_i(z)}\right) & \frac{k}{4} \\
   -\frac{1}{2}\im \left(\frac{\rho_i'(z)z}{\rho_i(z)}\right) & 0
  \end{pmatrix}
 \end{equation}
 with $ i=1 $ if $ \im z\geq 0 $, $ i=2 $ if $ \im z<0 $, note that $ A $ is continuous over $ \im z=0 $, and let
 \begin{equation}
  B(z) = 
  \begin{pmatrix}
   \frac{k}{4}\re \left(\frac{z+1}{z-1}\right)\\
   \frac{k}{4}\im \left(\frac{z+1}{z-1}\right)
  \end{pmatrix}.
 \end{equation}
 
 If $ (z,w) \in \mathcal R_{12} $, then $ L(\chi,\eta)=(z,w) $ for some $ (\chi,\eta) \in \mathcal G_R $ if and only if
 \begin{equation}\label{eq:prop_L_map:eq1}
  A(z)
  \begin{pmatrix}
   \chi \\
   \eta
  \end{pmatrix}
  = B(z)
 \end{equation} 
 has a solution $ (\chi,\eta) \in (-1,1)^2 $. This follows by the definition of $ L $ and by taking the real and imaginary part of the equation $ zF_i'(z;\chi,\eta)=0 $. Lemma \ref{lem::technical} \eqref{eq:lem_technical:0} together with \eqref{eq:lem_logarithmic_derivative:quotient} implies
 \begin{equation}
  \det A(z) = \frac{k}{8}\im \left(\frac{\rho_i'(z)z}{\rho_i(z)}\right) \neq 0,
 \end{equation}
 Which proves that $ L $ is injective. Moreover it proves that the map $ z \mapsto A(z)^{-1}B(z) $ is smooth on  $ \CC \backslash \left(\cup_{\ell=0}^{k'-1}\overline{I_\ell}\right) $.
 
 To prove surjectivity we prove that $ A(z)^{-1}B(z) \in (-1,1)^2 $ when $ (z,w) \in \mathcal R_{12} $. If $ z \in \RR \backslash \left(\cup_{\ell=0}^{k'-1}\overline{I_\ell}\right) $ the solution to \eqref{eq:prop_L_map:eq1} becomes
 \begin{equation}
  \begin{pmatrix}
   \chi \\
   \eta
  \end{pmatrix}
  =
  \begin{pmatrix}
   0 \\
   \frac{z+1}{z-1}
  \end{pmatrix}
  \in (-1,1)^2,
 \end{equation}
 and in particular the solution is in the interior of the closed curve $ \mathcal E_{FR} $. Due to smoothness of $ z\mapsto A(z)^{-1}B(z) $ on $ \CC \backslash \left(\cup_{\ell=0}^{k'-1}\overline{I_\ell}\right) $ and the definition of $ \mathcal E_{FR} $, the solution of \eqref{eq:prop_L_map:eq1} can not leave the interior of $ \mathcal E_{FR} \subset [-1,1]^2 $.
 
 In local coordinates $ L^{-1} $ is given by the map $ z\mapsto A(z)^{-1}B(z) $ which we have seen is smooth and in particular differentiable. Now, for $ z \in \CC \backslash \left(\cup_{\ell=0}^{k'-1}\overline{I_\ell}\right) $ let $ f(z;\chi,\eta) = F_1'(z;\chi,\eta) $ if $ \im z\geq 0 $ and let $ f(z;\chi,\eta) = F_2'(z;\chi,\eta) $ if $ \im z<0 $. If $ L(\chi,\eta) = (z,w) $, then $ f(z;\chi,\eta)=0 $ and the zero is simple, that is $ f'(z;\chi,\eta) \neq 0 $. By the implicit function theorem the map $ (\chi,\eta) \mapsto z $ is differentiable, that is, $ L $ is differentiable.
\end{proof}

\begin{proof}[Proof of Proposition \ref{prop:boundary_is_boundary}]
 By Proposition \ref{prop:L_map} the map $ L $ is a diffeomorphism, so $ (\chi',\eta') \to \partial \mathcal G_R $ if and only if $ L(\chi',\eta') \to C^+ $.
 
 If $ (\chi,\eta) \in \partial \mathcal G_R\cap (-1,1)^2 $ then by Lemma \ref{lem:projected_critical_points} and Lemma \ref{lem:eigenvalues} \eqref{eq:lem_eigenvalues:conjugate} there is a critical point of order two on $ C^+ $, that is $ (\chi,\eta) \in \mathcal E_{FR} \cup_{\ell=1}^{k'-1} \mathcal E_{RS}^{(\ell)} $.
 
 Conversely, if $ (\chi,\eta) \in \left(\mathcal E_{FR} \cup_{\ell=1}^{k'-1} \mathcal E_{RS}^{(\ell)}\right) \cap (-1,1)^2 $ then there is a critical point $ (z,w) \in C^+ $ of (at least) order two corresponding to $ (\chi,\eta) $. By taking a limit in $ \mathcal R_{12} $,
 \begin{equation}
  \lim_{\mathcal R_{12} \ni (\tilde z,\tilde w) \to(z,w)} L^{-1}(\tilde z,\tilde w) = (\chi',\eta'),
 \end{equation}  
  we get, since $ L $ is a diffeomorphism, that $ (\chi',\eta') \in \partial \mathcal G_R $ and $ (z,w) $ is a critical point of (at least) order two of $ F(\cdot;\chi',\eta') $. We saw however in the proof of Proposition \ref{prop:boundary} that this implies that $ (\chi,\eta) = (\chi',\eta') $. 

For the last two statements, let $ (\chi,\eta) \in (-1,1)^2 $ and let $ (z_1,w_1) $ and $ (z_2,w_2) $ be the two last critical points from Lemma \ref{lem:critical_points} and Lemma \ref{lem:zeros_on_loops}. If $ (z_1,w_1) \in \mathcal R \backslash C^+ $ then, as seen before, $ (z_2,w_2) = (\bar z_1,\bar w_1) $. It follows that the only way $ (\chi,\eta) $ is not located in one of the sets in Definition \ref{def:phases} is if $ (z_1,w_1) $ and $ (z_2,w_2) $ belongs to $ C_{\ell_1}^+ $ respectively $ C_{\ell_2}^+ $ where $ \ell_1 \neq \ell_2 $. However that would contradict the continuity in equation \eqref{eq:polynomial_of_critical_points}. It is now clear, again by continuity of equation \eqref{eq:polynomial_of_critical_points}, that $ \mathcal E_{FR} $ and $ \mathcal E_{RS}^{(\ell)} $, $ \ell=1,\dots,k'-1 $, separates $ \mathcal G_R $ and $ \mathcal G_F $ respectively $ \mathcal G_R $ and $ \mathcal G_S $.
 
\end{proof}
  
 \section{Proof of Theorems \ref{thm:local_smooth} and \ref{thm:local_rough}}\label{sec:proof_local_correlations}
 In this section we prove Theorems \ref{thm:local_smooth} and \ref{thm:local_rough}.
 
 For further references, let $ j_+ \in \{1,2\} $ be such that
 \begin{equation}\label{eq:sec_smooth_rough:j}
  -\frac{\chi}{2}\log|\rho_{j_+}(z)| \geq 0
 \end{equation}
 for all $ z \in \CC $ and let $ j_-\in \{1,2\}\backslash\{j_+\} $. That is, let $ j_+ = 1 $ and $ j_- = 2 $ if $ \chi\leq 0 $ and let $ j_+=2 $ and $ j_- =1 $ if $ \chi>0 $, see Lemma \ref{lem:eigenvalues} \eqref{eq:lem_eigenvalues:at_cuts} and \eqref{eq:lem_eigenvalues:inequality}.

 \subsection{Proof of Theorem \ref{thm:local_smooth}}
 The proof is done by a steepest descent analysis. We first construct the curves we will integrate over.
 
 Let $ \mathcal G_\ell $ be a connected component of $ \mathcal G_S $ and more precisely, be the connected component with boundary $ \mathcal E_{RS}^{(\ell)} $, see Proposition \ref{prop:boundary} and \ref{prop:boundary_is_boundary}. If $ (\chi,\eta) \in \mathcal G_\ell $ then, by definition of $ \mathcal G_S $ and Proposition \ref{prop:boundary}, $ C_\ell^+ $ contains four distinct critical points of $ F $, we denote the projection of them to $ \CC $ as $ z_i(\chi,\eta) \in [x_{2\ell},x_{2\ell-1}] $ for $ i=1,2,3,4 $. Note that the critical points are continuous with respect to $ (\chi,\eta) $ and distinguishable as points on $ \mathcal R $, since $ F $ has no double critical point when $ (\chi,\eta) \in \mathcal G_S $. Fix a point $ (0,\eta_0) \in \mathcal G_\ell $. Since
\begin{equation}\label{eq:sec_smooth:derivative_middle_line}
 F_1'(z;0,\eta_0) = F_2'(z;0,\eta_0)=\frac{k(\eta_0-1)}{4z(z-1)}\left(z-\frac{\eta_0+1}{\eta_0-1}\right),
\end{equation}
two of the $ z_i(0,\eta_0) $ are given by $ \frac{\eta_0+1}{\eta_0-1} $, one is $ x_{2\ell} $ and one is $ x_{2\ell-1} $. It is the first two that are of most importance for us. Fix $ z_1(\chi,\eta) $ and $ z_2(\chi,\eta) $ by require that $ F_i'(z_i(\chi,\eta);\chi,\eta)=0 $ in a neighborhood of $ (0,\eta_0) $ for $ i =1,2 $. By continuity of the critical points of $ F $ this determines $ z_1(\chi,\eta) $ and $ z_2(\chi,\eta) $ on all of $ \mathcal G_\ell $.

 \begin{lemma}\label{lem:smooth:curves}
  Let $ (\chi,\eta) \in \mathcal G_\ell $. We can create $ \Gamma_{\chi,\eta} $, a union of curves going around each interval $ (x_{2j+1},x_{2j}) $, $ j=0,\dots,k'-1 $ where $ x_{2k'-1}=-\infty $, ones, and $ \gamma_{\chi,\eta} $, a loop which intersects $ z_{j_-}(\chi,\eta) $ and has $ 0 $, $ 1 $ and a part of $ \Gamma_{\chi,\eta} $ in its interior, such that
   \begin{equation}\label{eq:lem_smooth:inequality_1}
    \re F_{j_-}(w;\chi,\eta)-\re F_{j_-}(z;\chi,\eta) \leq -\eps < 0,
   \end{equation}
   and
   \begin{equation}\label{eq:lem_smooth:inequality_2}
    \re F_{j_-}(w;\chi,\eta)-\re F_{j_+}(z;\chi,\eta) \leq -\eps < 0,
   \end{equation}
   for $ z \in \gamma_{\chi,\eta} $ and $ w \in \Gamma_{\chi,\eta} $.
 \end{lemma}
 See Figure \ref{fig:smooth:first_deformation} for an example of the curves.
 \begin{proof}
 We recall the definition of $ j_- $, \eqref{eq:sec_smooth_rough:j}, and let 
  \begin{equation}
   \Omega_{\chi,\eta} = \{z \in \CC:\re F_{j_-}(z;\chi,\eta) \geq \re F_{j_-}(z_{j_-}(\chi,\eta);\chi,\eta)\}.
  \end{equation}
  Let $ \Omega_{\chi,\eta}^{(0)} $ be the connected component of $ \Omega_{\chi,\eta} $ containing $ z_{j_-}(\chi,\eta) $, see Figure \ref{fig:smooth:domain} for an example.   
  
  To create the curves with the right properties we start by investigating $ \Omega_{\chi,\eta}^{(0)} $ and show that it is bounded and does not contain any cuts, that is, it does not contain any of the intervals $ (x_{2j+1},x_{2j}) $. We later define $ \gamma_{\chi,\eta} $ and $ \Gamma_{\chi,\eta} $ so that $ \gamma_{\chi,\eta}\subset \Omega_{\chi,\eta}^{(0)} $ and $ \Gamma_{\chi,\eta} \subset \CC\backslash \Omega_{\chi,\eta}^{(0)} $.
  
  It follows from Lemma \ref{lem:eigenvalues} \eqref{eq:lem_eigenvalues:at_infinity}, which shows that $ \re F_i(z;\chi,\eta) \to -\infty $ as $ |z| \to \infty $, that $ \Omega_{\chi,\eta} $ is bounded.
  
  Let $ \gamma_{l (r)} $ be a part of the boundary of $ \Omega_{\chi,\eta}^{(0)} $, ``just'' to the left (right) of $ z_{j_-} $. More precisely let $ \gamma_{l(r)} \subset \partial \Omega_{\chi,\eta}^{(0)} $ be a simple, closed curve intersecting $ z_{j_-} $ and for which there is a neighborhood of $ z_{j_-} $ which $ \gamma_{l(r)} $ divide in such a way that there is no part of $ \Omega_{\chi,\eta}^{(0)} $ to the left (right) of $ \gamma_{l(r)} $ inside this neighborhood (see Figure \ref{fig:smooth:domain}). Both $ \gamma_l $ and $ \gamma_r $ exist since $ z_{j_-}(\chi,\eta) $ is a local maximum of the function $ \RR_- \ni z\mapsto \re F_{j_-}(z;\chi,\eta) $. That $ z_{j_-} $ is a local maximum requires an argument which is given below.
  
  First we show that $ z_{j_-}(\chi,\eta) $ is a critical point of $ z\mapsto F_{j_-}(z;\chi,\eta) $ for all $ (\chi,\eta) \in \mathcal G_\ell $. Recall from the discussion just after \eqref{eq:sec_smooth:derivative_middle_line} that this is true in a neighborhood of $ (0,\eta_0) $ by definition of $ z_{j_-} $. If $ \chi \neq 0 $ then \eqref{eq:polynomial_of_critical_points} and Lemma \ref{lem:logarithmic_derivative} implies that there is no critical point of $ F $ at $ y_{2\ell} $ or $ y_{2\ell-1} $. By continuity of the critical points of $ F $ in $ (\chi,\eta)\in \mathcal G_\ell $ it follows that the critical point in $ C_\ell $ corresponding to $ z_{j_-} $ can not change sheet unless $ \chi=0 $, and by \eqref{eq:sec_smooth:derivative_middle_line} it does not change sheet when $ \chi=0 $ either. Hence $ z_{j_-} $ is a critical point of $ F_{j_-} $, and in particular a critical point of $ \re F_{j_-} $. 
  
  Now, since $ \im F_{j_-} $ is constant on $ I_\ell $, a critical point of $ I_\ell \ni z \mapsto \re F_{j_-}(z;\chi,\eta) $ is a critical point of $ F_{j_-} $. Since the critical points of $ F $ does not coalesce when $ (\chi,\eta) $ is varied in $ \mathcal G_\ell $, the critical points of $ I_\ell \ni z \mapsto \re F_{j_-}(z;\chi,\eta) $ does not coalesce, so the second derivative can not change sign. It follows from \eqref{eq:sec_smooth:derivative_middle_line} that $ z_{j_-}(0,\eta) $ is a local maximum of the function $ I_\ell \ni z \mapsto \re F_{j_-}(z;0,\eta) $. Since $ I_\ell \ni z \mapsto \re F_{j_-}(z;\chi,\eta) $ has a non vanishing second derivative it follows by continuity that $ z_{j_-}(\chi,\eta) $ is a local maximum for all $ (\chi,\eta) \in \mathcal G_\ell $, as claimed. 
  
  Let the ``first'' time $ \gamma_{l(r)} $ intersect the real line away from $ z_{j_-} $ be at $ x_{l(r)} $. By Lemma \ref{lem:eigenvalues} \eqref{eq:lem_eigenvalues:conjugate} $ \Omega_{\chi,\eta}^{(0)} $ is symmetric with respect to the real line, so $ x_{l(r)} $ does not depend on if we follow $ \gamma_{l(r)} $ in the lower or upper half plane and by definition of $ \gamma_{l(r)} $ the ``first'' time is the only time it intersect the real line away from $ z_{j_-} $. We show below that $ x_l, x_r>0 $. 
  
   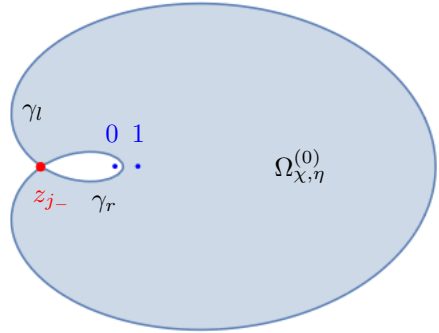
\begin{figure}[t]
  \begin{center}
   \begin{tikzpicture}[scale=1]
   \draw (0,0) node {\includegraphics[scale=.5,trim = .5cm 1.9cm .8cm 1.9cm, clip]{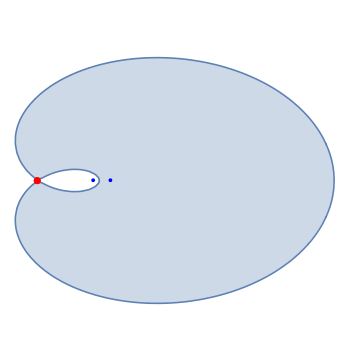}};
   \draw (-2.3,-.2) node [below]{$\color{red}z_{j_-}$};
   \draw (-1.5,.2) node [above]{$\color{blue} 0 $};
   \draw (-1.15,.2) node [above]{$\color{blue} 1 $};
   \draw (1,0) node {$ \Omega_{\chi,\eta}^{(0)}$};
   \draw (-2.55,.7) node {$ \gamma_l $};
   \draw (-1.6,-.5) node {$ \gamma_r $};   
   \end{tikzpicture}
  \end{center}
  \caption{An example of the set $ \Omega_{\chi,\eta}^{(0)} $. The curves $ \gamma_l $ and $ \gamma_r $ are constructed from the boundary of $ \Omega_{\chi,\eta}^{(0)} $. \label{fig:smooth:domain}}
\end{figure}
  
  Assume for a contradiction that $ x_l<z_{j_-} $. By definition of $ j_- $,
  \begin{equation}
   \re F_{j_-}(z;\chi,\eta) \leq \frac{k}{4}(\eta+1)\log|z|-\frac{k}{2}\log|z-1|,
  \end{equation}
  with equality for $ z \in \cup_{j=0}^{k'-1}(x_{2j+1},x_{2j}) $ by Lemma \ref{lem:eigenvalues} \eqref{eq:lem_eigenvalues:at_cuts}. The right hand side is increasing to the left of $ z_0 = (\eta+1)/(\eta-1) $ and decreasing to the right of $ z_0 $. Hence all cuts between $ x_l $ and $ z_0 $ and all cuts between $ z_{j_-} $ and $ z_0 $ are in the interior of $ \Omega_{\chi,\eta} $. In particular all cuts between $ x_l $ and $ z_{j_-} $ are in the interior of $ \Omega_{\chi,\eta} $. This means that $ \re F_{j_-} $ is constant on the boundary of the connected component of the complement of $ \Omega_{\chi,\eta} $ which lies between $ x_l $ and $ z_{j_-} $. Since $ \re F_{j_-} $ is harmonic and non-constant on that connected component, this contradicts the maximum principle. 
  
  With a similar argument we get that $ x_r \not\in (z_{j_-},0) $. Again by the maximum principle we get that $ x_l,x_r\neq z_{j_-} $. Since $ \re F_{j_-} $ is harmonic in the upper and in the lower half plane it follows by the maximum principle that $ \gamma_l $ and $ \gamma_r $ can not intersect ``before'' $ x_l $ respectively $ x_r $ and hence $ \gamma_l $ and $ \gamma_r $ can not intersect at all except at $ z_{j_-} $. Hence $ x_l,x_r>0 $. This means that $ \Omega_{\chi,\eta}^{(0)} $ does not intersect the negative real line away from $ z_{j_-} $ and in particular does not contain any cuts.
  
  We now consider the positive part of the real line. By Lemma \ref{lem:eigenvalues} \eqref{eq:lem_eigenvalues:at_infinity}, $ \re F_{j_-}(z;\chi,\eta) \to -\infty $ as $ |z| \to 0,\infty $. Rewrite $ \re F_1 $ to 
  \begin{equation}
   \re F_1(z;\chi,\eta) = \frac{k}{4}(\eta+1)\log|z|-\frac{k}{2}(1-\chi)\log|z-1|-\frac{\chi}{2}\log|(z-1)^k\rho_1(z)|,
  \end{equation}
  and $ \re F_2 $ to 
  \begin{equation}
   \re F_2(z;\chi,\eta) = \frac{k}{4}(\eta+1)\log|z|-\frac{k}{2}(\chi+1)\log|z-1|-\frac{\chi}{2}\log\left|\frac{\rho_2(z)}{(z-1)^k}\right|.
  \end{equation}
  Since $ \rho_1 $ has a pole of order $ k $ at $ z=1 $ and $ \rho_2 $ has a zero of order $ k $ at $ z=1 $, it follows from the above equalities that $ \re F_{j_-}(z;\chi,\eta) \to \infty $ as $ z \to 1 $. Since $ F_{j_-} $, and therefore also $ \re F_{j_-} $, does not have any critical points in $ (0,\infty) $, it follows that $ \re F_{j_-} $ increases from $ -\infty $ to $ \infty $ on $ (0,1) $ and decreases from $ \infty $ to $ -\infty $ on $ (1,\infty) $. Hence, a neighborhood of $ 1 $ and nothing more of the positive real line is contained in $ \Omega_{\chi,\eta} $.
  
  All the above means that $ \Omega_{\chi,\eta}^{(0)} $ does not contain $ \infty $, zero or any cuts. If $ \Omega_{\chi,\eta}^{(0)} $ does not contain one it would mean that $ \re F_{j_-} $ is harmonic in $ \Omega_{\chi,\eta}^{(0)} $ and constant on the boundary, which by the maximum principle implies that $ \re F_{j_-} $ is constant. So $ 1 \in \Omega_{\chi,\eta}^{(0)} $. 
  
  We are ready to define the curves. Since $ \Omega_{\chi,\eta}^{(0)} $ is connected we can connect $ z_{j_-}(\chi,\eta) $ and a point in $ \Omega_{\chi,\eta}^{(0)}\cap (1,\infty) $, with a path contained in $ \Omega_{\chi,\eta}^{(0)} $ and which does not intersect the real line except at the endpoints. Let $ \gamma_{\chi,\eta} $ be this path together with its reflection in $ \RR $. All of $ \gamma_{\chi,\eta} $ lies in $ \Omega_{\chi,\eta}^{(0)} $ since $ \re F_{j_-}(\bar z;\chi,\eta) = \re F_{j_-}(z;\chi,\eta) $ by Lemma \ref{lem:eigenvalues} \eqref{eq:lem_eigenvalues:conjugate}, which means that $ \Omega_{\chi,\eta}^{(0)} $ is symmetric with respect to the real line.
  
  We now construct $\Gamma_{\chi,\eta} $. Take a union of curves going around each cut ones. Take the curves so close to the cuts that they do not intersect $ \Omega_{\chi,\eta} $. In case there is a component of $ \Omega_{\chi,\eta} $ containing a part of a cut, take the curve such that it goes around that component as well. We do this such that the curves have a positive distance to $ \Omega_{\chi,\eta} $. 
  
  Let $ w \in \Gamma_{\chi,\eta} $ and $ z \in \gamma_{\chi,\eta} $, then
  \begin{equation}
   \re F_{j_-}(w;\chi,\eta)-\re F_{j_-}(z;\chi,\eta) \leq -\eps,
  \end{equation}
  for some $ \eps>0 $, since $ \Gamma_{\chi,\eta} $ has a positive distance to $ \Omega_{\chi,\eta} $. Moreover
  \begin{multline}
   \re F_{j_-}(w;\chi,\eta)-\re F_{j_+}(z;\chi,\eta) = (\re F_{j_-}(w;\chi,\eta)-\re F_{j_-}(z;\chi,\eta)) \\
   +(\re F_{j_-}(z;\chi,\eta)-\re F_{j_+}(z;\chi,\eta)) \leq -\eps,
  \end{multline}
  by definition of $ j_- $ and $ j_+ $.
 \end{proof}
 
 The orientation of $ \gamma_{\chi,\eta} $ and $ \Gamma_{\chi,\eta} $ are taken counter clockwise respectively clockwise. We now prove Theorem \ref{thm:local_smooth}.
 
\begin{proof}[Proof of Theorem \ref{thm:local_smooth}]

 In the expression of the kernel in Theorem \ref{thm:finite_kernel} we have some freedom in the choice of the curve $ \gamma_{0,1} $. We take $ \gamma_{0,1} $ as $ \gamma_{\chi,\eta} $ and deform $ \gamma_1 $ to $ \Gamma_{\chi,\eta} $. Both $ \gamma_{\chi,\eta} $ and $ \Gamma_{\chi,\eta} $ are defined in Lemma \ref{lem:smooth:curves}, see Figure \ref{fig:smooth:first_deformation}. The only contribution from the deformation comes from the simple pole at $ w=z $. Hence
 \begin{multline}
  \left[K\left(2km,2\xi+i;2km',2\xi'+j\right)\right]_{i,j=0}^1 = \\
  = \frac{1}{2\pi\i}\oint_{\gamma_{\chi,\eta}} \rho_1(z)^{\kappa-\kappa'}z^{\zeta'-\zeta}E(z)
  \begin{pmatrix}
   1 & 0 \\
   0 & 0
  \end{pmatrix}
  E(z)^{-1}\frac{\d z}{z} \\
  -\frac{\mathds{1}_{\kappa>\kappa'}}{2\pi\i}\oint_{\gamma_{\chi,\eta}} \Phi(z)^{\kappa-\kappa'}z^{\zeta'-\zeta}\frac{\d z}{z} \\
  + \frac{1}{(2\pi\i)^2}\oint_{\Gamma_{\chi,\eta}}\oint_{\gamma_{\chi,\eta}} \frac{w^{\frac{kN}{4}(\eta+1)+e_\eta+\zeta'}}{z^{\frac{kN}{4}(\eta+1)+e_\eta+\zeta+1}}\frac{(z-1)^{\frac{kN}{2}}}{(w-1)^{\frac{kN}{2}}}\rho_1(w)^{-\frac{N}{2}\chi-e_\chi-\kappa'} \\
  \times E(w)
  \begin{pmatrix}
   1 & 0 \\
   0 & 0
  \end{pmatrix}
  E(w)^{-1}\Phi(z)^{\frac{N}{2}\chi+e_\chi+\kappa}\frac{\d z\d w}{z-w}.
 \end{multline}
 In the first two terms we deform $ \gamma_{\chi,\eta} $ to $ \gamma_\ell $, there is no extra contribution from this deformation. The sum of them is precisely $ K_{smooth}^{(\ell)} $. We show below that the last term is small, which proves the theorem. 
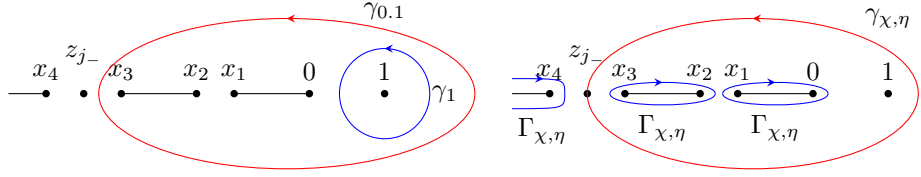
\begin{figure}[t]
 \begin{center}
  \begin{tikzpicture}[scale=1]
  \tikzset{-<-/.style={decoration={markings,mark=at position .25 with {\arrow{stealth}}},postaction={decorate}}}
   \draw (1,0) node[circle,fill,inner sep=1pt,label=above:$1$]{};
   \draw (0,0) node[circle,fill,inner sep=1pt,label=above:$0$]{};
   \draw (-1,0) node[circle,fill,inner sep=1pt,label=above:$x_1$]{};
   \draw (-1.5,0) node[circle,fill,inner sep=1pt,label=above:$x_2$]{};
   \draw (-2.5,0) node[circle,fill,inner sep=1pt,label=above:$x_3$]{};
   \draw (-3.5,0) node[circle,fill,inner sep=1pt,label=above:$x_4$]{};
   \draw (-3,0) node[circle,fill,inner sep=1pt]{};
   \draw (-3,.5) node {$z_{j_-}$};
   \draw (1,1.1) node{$ \gamma_{0.1}$};
   \draw (1.8,0) node{$ \gamma_{1}$};
   \draw (-1,0)--(0,0);
   \draw (-2.5,0)--(-1.5,0);
   \draw (-4,0)--(-3.5,0);
   \draw[-<-,blue] (1.,0) circle (.6);
   \draw[-<-,red] (-.3,0) ellipse (2.5 and 1);
  \end{tikzpicture}
  \hspace{.22cm}
  \begin{tikzpicture}[scale=1]
    \tikzset{->-/.style={decoration={markings,mark=at position .5 with {\arrow{stealth}}},postaction={decorate}}}
    \tikzset{-<-/.style={decoration={markings,mark=at position .25 with {\arrow{stealth}}},postaction={decorate}}}
   \draw (1,0) node[circle,fill,inner sep=1pt,label=above:$1$]{};
   \draw (0,0) node[circle,fill,inner sep=1pt,label=above:$0$]{};
   \draw (-1,0) node[circle,fill,inner sep=1pt,label=above:$x_1$]{};
   \draw (-1.5,0) node[circle,fill,inner sep=1pt,label=above:$x_2$]{};
   \draw (-2.5,0) node[circle,fill,inner sep=1pt,label=above:$x_3$]{};
   \draw (-3.5,0) node[circle,fill,inner sep=1pt,label=above:$x_4$]{};
   \draw (-3,0) node[circle,fill,inner sep=1pt]{};
   \draw (-3,.5) node {$z_{j_-}$};
   \draw (1,1) node{$ \gamma_{\chi,\eta}$};
   \draw (-.5,-.45) node{$ \Gamma_{\chi,\eta}$};
   \draw (-2,-.45) node{$ \Gamma_{\chi,\eta}$};
   \draw (-3.6,-.5) node{$ \Gamma_{\chi,\eta}$};
   \draw (-1,0)--(0,0);
   \draw (-2.5,0)--(-1.5,0);
   \draw (-4,0)--(-3.5,0);
   \draw[->-,blue] (-1.2,0) .. controls (-1.2,.2) and (.2,.2) .. (.2,0);
   \draw[blue] (-1.2,0) .. controls (-1.2,-.2) and (.2,-.2) .. (.2,0);
   \draw[->-,blue] (-2.7,0) .. controls (-2.7,.2) and (-1.3,.2) .. (-1.3,0);
   \draw[blue] (-2.7,0) .. controls (-2.7,-.2) and (-1.3,-.2) .. (-1.3,0);
   \draw[->-,blue] (-4,.2) .. controls (-3.3,.2) .. (-3.3,0);
   \draw[blue] (-4,-.2) .. controls (-3.3,-.2) .. (-3.3,0);
   \draw[-<-,red] (-.8,0) ellipse (2.2 and 1);
  \end{tikzpicture}
  \caption{An example of deformation of the contours in the steepest descent analysis of the correlation kernel at the smooth region.\label{fig:smooth:first_deformation}}
 \end{center}
\end{figure}
 
 In case $ j_-=1 $ we are satisfied with the expression of the last term. However if $ j_-=2 $ we want to change $ \rho_2(w) $ to $ \rho_1(w) $. Why this is the case will be transparent later. By collapsing $ \Gamma_{\chi,\eta} $ to the cuts, that is to $ \cup_{j=0}^{k'-1}(x_{2j+1},x_{2j}) $ where $ x_{2k'-1}=-\infty $, we get the same integral as if we consider the integral where all $ \rho_1(w) $ are interchanged with $ \rho_2(w) $ and collapse $ \Gamma_{\chi,\eta} $ to the cuts, the only cost is a minus sign in front of the integral. To see this note that $ (\rho_1)_\pm(w) = (\rho_2)_\mp(w) $ for $ w \in \cup_{j=0}^{k'-1}(x_{2j+1},x_{2j}) $, where $ (\rho_i)_\pm $ is the limit of $ \rho_i $ approaching from above respectively from below. We use this fact in case $ j_- = 2 $ and recall that $ j_-=1 $ if and only of $ \chi>0 $ to rewrite the last term to
 \begin{multline}
   \frac{(-1)^{\mathds{1}_{\chi\leq 0}}}{(2\pi\i)^2}\oint_{\Gamma_{\chi,\eta}}\oint_{\gamma_{\chi,\eta}} \frac{w^{\frac{kN}{4}(\eta+1)+e_\eta+\zeta'}}{z^{\frac{kN}{4}(\eta+1)+e_\eta+\zeta+1}}\frac{(z-1)^{\frac{kN}{2}}}{(w-1)^{\frac{kN}{2}}}\rho_{j_-}(w)^{-\frac{N}{2}\chi-e_\chi-\kappa'} \\
  \times E(w)
  \begin{pmatrix}
   \mathds{1}_{\chi>0} & 0 \\
   0 & \mathds{1}_{\chi \leq 0}
  \end{pmatrix}
  E(w)^{-1}\Phi(z)^{\frac{N}{2}\chi+e_\chi+\kappa}\frac{\d z\d w}{z-w}.
 \end{multline}
 We write 
 \begin{equation}
  \Phi(z)^n = E(z)
  \begin{pmatrix}
   \rho_1(z)^n & 0 \\
   0 & 0
  \end{pmatrix}
  E(z)^{-1} + E(z)
  \begin{pmatrix}
   0 & 0 \\
   0 & \rho_2(z)^n
  \end{pmatrix}
  E(z)^{-1},
 \end{equation}
 for $ n = \frac{N}{2}\chi+e_\chi+\kappa $ and use this equality to divide the integral into two terms. We end up with the smooth kernel plus the integral
 \begin{multline}
  \frac{(-1)^{\mathds{1}_{\chi \leq 0}}}{(2\pi\i)^2}\oint_{\Gamma_{\chi,\eta}}\oint_{\gamma_{\chi,\eta}}\e^{N\left(F_{j_-}(w;\chi,\eta)-F_1(z;\chi,\eta)\right)} \frac{w^{e_\eta+\zeta'}}{z^{e_\eta+\zeta+1}}\frac{\rho_1(z)^{e_\chi+\kappa}}{\rho_{j_-}(w)^{e_\chi+\kappa'}} \\
  \times E(w)
  \begin{pmatrix}
   \mathds{1}_{\chi>0} & 0 \\
   0 & \mathds{1}_{\chi\leq 0}
  \end{pmatrix}
  E(w)^{-1}E(z)
  \begin{pmatrix}
   1 & 0 \\
   0 & 0
  \end{pmatrix}
  E(z)^{-1}\frac{\d z\d w}{z-w} \\
  + \frac{(-1)^{\mathds{1}_{\chi\leq 0}}}{(2\pi\i)^2}\oint_{\Gamma_{\chi,\eta}}\oint_{\gamma_{\chi,\eta}}\e^{N\left(F_{j_-}(w;\chi,\eta)-F_2(z;\chi,\eta)\right)} \frac{w^{e_\eta+\zeta'}}{z^{e_\eta+\zeta+1}}\frac{\rho_2(z)^{e_\chi+\kappa}}{\rho_{j_-}(w)^{e_\chi+\kappa'}} \\
  \times E(w)
  \begin{pmatrix}
   \mathds{1}_{\chi>0} & 0 \\
   0 & \mathds{1}_{\chi\leq 0}
  \end{pmatrix}
  E(w)^{-1}E(z)
  \begin{pmatrix}
   0 & 0 \\
   0 & 1
  \end{pmatrix}
  E(z)^{-1}\frac{\d z\d w}{z-w},
 \end{multline}
 where $ F_1 $ and $ F_2 $ are given in \eqref{eq:sec_eigenvalues:F1} and \eqref{eq:sec_eigenvalues:F2}. From \eqref{eq:lem_smooth:inequality_1} or \eqref{eq:lem_smooth:inequality_2} we have
 \begin{multline}
  \left|\frac{(-1)^{\mathds{1}_{\chi\leq0}}}{(2\pi\i)^2}\oint_{\Gamma_{\chi,\eta}}\oint_{\gamma_{\chi,\eta}}\e^{N\left(F_{j_-}(w;\chi,\eta)-F_1(z;\chi,\eta)\right)} \frac{w^{e_\eta+\zeta'}}{z^{e_\eta+\zeta+1}}\frac{\rho_1(z)^{e_\chi+\kappa}}{\rho_{j_-}(w)^{e_\chi+\kappa'}} \right. \\
  \left.\times E(w)
  \begin{pmatrix}
   \mathds{1}_{\chi>0} & 0 \\
   0 & \mathds{1}_{\chi\leq0}
  \end{pmatrix}
  E(w)^{-1}E(z)
  \begin{pmatrix}
   1 & 0 \\
   0 & 0
  \end{pmatrix}
  E(z)^{-1}\frac{\d z\d w}{z-w}\right| \\
  \leq \e^{N\sup_{w\in\Gamma_{\chi,\eta}}\sup_{z\in\gamma_{\chi,\eta}}\left(\re F_{j_-}(w;\chi,\eta)-\re F_1(z;\chi,\eta)\right)} \Ordo(1) 
  = \Ordo\left(\e^{-N\eps}\right). 
 \end{multline}
 Similarly we get that the last term is $ \Ordo\left(\e^{-N\eps}\right) $ using the other of \eqref{eq:lem_smooth:inequality_1} and \eqref{eq:lem_smooth:inequality_2}, 
 which proves the theorem.
\end{proof}

\subsection{Proof of Theorem \ref{thm:local_rough}}

The structure of the proof is the same as the proof of Theorem \ref{thm:local_smooth}. 

Let $ (\chi,\eta) \in \mathcal G_R $. Let $ z_1=z_1(\chi,\eta) $ be the first component of $ L(\chi,\eta) $ and $ z_2=z_2(\chi,\eta) = \overline{z_1}(\chi,\eta)$.
\begin{lemma}\label{lem:rough:critical_points}
If $ z_1,z_2 \not\in \RR $ then $ z_1 $ and $ z_2 $ are critical points of $ F_{j_-} $. 
\end{lemma}

\begin{proof}
 By Lemma \ref{lem:eigenvalues} \eqref{eq:lem_eigenvalues:conjugate} it is enough to prove the statement for $ z_1(\chi,\eta) $.
 
 Take an $ x>1 $ such that $ \frac{\rho_1'(x)}{\rho_1(x)}<0 $, such $ x $ exists by \eqref{eq:lem_surjective:pole_at_one}. By definition of $ j_\pm $ and Lemma \ref{lem:logarithmic_derivative}, 
  \begin{equation}\label{eq:lem_rough_critical_points}
  -\frac{\chi}{2}\frac{\rho_{j_+}'(x)}{\rho_{j_+}(x)}<0,
 \end{equation}
 for all $ \chi \in (-1,1)\backslash \{0\} $. Since $ L $ is a diffeomorphism, Proposition \ref{prop:L_map}, there is a sequence in $ \mathcal G_R $ such that $ (\chi',\eta') \to (\chi,\eta) $ for some $ (\chi,\eta) \in (-1,1)^2 $, $ \chi \neq 0 $, and $ \lim_{(\chi',\eta')\to (\chi,\eta)}z_1(\chi',\eta')= x $. Take the sequence close enough to $ (\chi,\eta) $ so there is an $ i\in \{1,2\} $ such that $ F_i'(z_1(\chi',\eta');\chi',\eta') = 0 $. By Lemma \ref{lem:projected_critical_points}, $ F_i'(x;\chi,\eta) = 0 $. However by \eqref{eq:lem_rough_critical_points}
 \begin{equation}
  F_{j_+}'(x;\chi,\eta) = \frac{k}{4}\frac{(\eta-1)x-(\eta+1)}{x(x-1)}-\frac{\chi}{2}\frac{\rho_{j_+}'(x)}{\rho_{j_+}(x)} < 0.
 \end{equation}
 So $ i=j_- $, that is, the lemma is true for $ z_1(\chi',\eta') $. By continuity of $ z_1 $ the statement is true for all $ (\chi,\eta) \in \mathcal G_R $, since $ z_1 $ passes a cut precisely when $ \chi $ changes sign and therefore $ j_- $ changes value.
\end{proof}

\begin{lemma}\label{lem:rough:curves}
  Let $ (\chi,\eta) \in \mathcal G_R $. There are two curves $ \gamma_{\chi,\eta} $ and $ \Gamma_{\chi,\eta} $ going around zero and one respectively going around one and not zero and only intersecting each other at $ z_1 $ and $ z_2 $, such that, for $ z \in \gamma_{\chi,\eta} $ and $ w \in \Gamma_{\chi,\eta} $,
  \begin{equation}
   \re F_{j_-}(w;\chi,\eta)-\re F_{j_-}(z;\chi,\eta) \leq 0,
  \end{equation}
  with equality only if $ w=z=z_i $, $ i=1,2 $, and 
  \begin{equation}
   \re F_{j_-}(w;\chi,\eta)-\re F_{j_+}(z;\chi,\eta) \leq -\eps,
  \end{equation}
  if $ \chi \neq 0 $ for some $ \eps > 0 $. 
 \end{lemma}
 
 \begin{proof}
  Let
  \begin{equation}
   \Omega_{\chi,\eta} = \{z \in \CC:\re F_{j_-}(z;\chi,\eta) \geq \re F_{j_-}(z_1;\chi,\eta)\},
  \end{equation}
  and let $ \Omega_{\chi,\eta}^{(0)} $ be the connected component of $ \Omega_{\chi,\eta} $ which contains $ z_1 $. As in the smooth case we investigate $ \Omega_{\chi,\eta}^{(0)} $ in order to create $ \gamma_{\chi,\eta} $ and $ \Gamma_{\chi,\eta} $ with the desired properties.

  By Lemma \ref{lem:eigenvalues} \eqref{eq:lem_eigenvalues:at_infinity}, $ F_{j_-}(z;\chi,\eta) \to -\infty $ as $ |z| \to \infty $, so $ \Omega_{\chi,\eta} $ is bounded. 
  
  Below we prove that the boundary of $ \Omega_{\chi,\eta}^{(0)} $ intersects the negative part of the real line at, at most, two places and at the positive part of the real line at, at most, two places.

Since $ \re F_{j_-} $ is increasing to infinity in $ (0,1) $ and decreasing from infinity in $ (1,\infty) $ (which we saw in the proof of Lemma \ref{lem:smooth:curves}) it is clear that the claim is true on the positive part of the real line. That this is true also on the negative real line is slightly more subtle. Assume the boundary of $ \Omega_{\chi,\eta}^{(0)} $ intersects the negative part of the real line at three distinct points. Then there are two of these points, say $ x_l $ and $ x_r $, which lie on the boundary of a bounded connected component of the complement of $ \Omega_{\chi,\eta} $, recall that $ \Omega_{\chi,\eta} $ is symmetric with respect to $ \RR $ by Lemma \ref{lem:eigenvalues} \eqref{eq:lem_eigenvalues:conjugate}, call the component $ K $. An argument as when we investigated corresponding set in the smooth case shows that all cuts between $ x_l $ and $ x_r $ are contained in $ \Omega_{\chi,\eta} $. Namely, by definition of $ j_- $,
  \begin{equation}
   \re F_{j_-}(z;\chi,\eta) \leq \frac{k}{4}(\eta+1)\log|z|-\frac{k}{2}\log|z-1|,
  \end{equation}
  with equality for $ z \in \cup_{j=0}^{k'-1}(x_{2j+1},x_{2j}) $, where $ x_{2k'-1}=-\infty $, by Lemma \ref{lem:eigenvalues} \eqref{eq:lem_eigenvalues:at_cuts}. Moreover the right hand side is increasing before $ z_0 = (\eta+1)/(\eta-1) $ and decreasing after $ z_0 $. Hence all cuts between $ x_l $ and $ z_0 $ and all cuts between $ x_r $ and $ z_0 $ are in the interior of $ \Omega_{\chi,\eta} $. In particular all cuts between $ x_l $ and $ x_r $ are in the interior of $ \Omega_{\chi,\eta} $. This means that $\re F_{j_-} $ is harmonic in $ K $ and is constant on the boundary. Since $\re F_{j_-} $ is non-constant on $ K $ we obtain a contradiction by the maximum principle. Hence the boundary of $ \Omega_{\chi,\eta}^{(0)} $ intersect the negative part of the real line at no more than two points. 
  
  To summarize, since $ z_1 $ and $ z_2 $ are critical points of $ F_{j_-} $, Lemma \ref{lem:rough:critical_points}, the boundary of $ \Omega_{\chi,\eta}^{(0)} $ consists of four curves leaving $ z_1 $, two going to the negative part of the real line and two going to the positive part of the real line. By symmetry the same is true for $ z_2 $. The only exception is if $ \chi=0 $, that is $ z_1=z_2 \in \RR $, then $ z_1 $ is the only intersection of $ \partial \Omega_{\chi,\eta}^{(0)} $ and the negative part of the real line.
  
  Take $ \gamma_{\chi,\eta} $ as any closed simple curve inside $ \Omega_{\chi,\eta} $ that intersects $ z_1 $, $ z_2 $, the negative part of the real line and $ (1,\infty) $. Take two simple curves in the exterior of $ \Omega_{\chi,\eta}\backslash\{z_1,z_2\} $ going from $ z_1 $ to $ z_2 $ which intersect the positive part of the real line at some point bigger than one respectively smaller than one. Take $ \Gamma_{\chi,\eta} $ as the union of these two curves.
  
  Let $ w \in \Gamma_{\chi,\eta} $ and $ z \in \gamma_{\chi,\eta} $, then
  \begin{equation}
   \re F_{j_-}(w;\chi,\eta)-\re F_{j_-}(z;\chi,\eta) \leq 0,
  \end{equation}
  with equality only if $ w=z=z_i $, $ i\in \{1,2\} $, by definition of $ \Omega_{\chi,\eta} $, $ \gamma_{\chi,\eta} $ and $ \Gamma_{\chi,\eta} $. Moreover, if $ \chi \neq 0 $ then
  \begin{multline}
   \re F_{j_-}(w;\chi,\eta)-\re F_{j_+}(z;\chi,\eta) = (\re F_{j_-}(w;\chi,\eta)-\re F_{j_-}(z;\chi,\eta)) \\
   +(\re F_{j_-}(z;\chi,\eta)-\re F_{j_+}(z;\chi,\eta)) \leq -\eps,
  \end{multline}
  for some $ \eps>0 $ by definition of $ j_- $ and $ j_+ $ and Lemma \ref{lem:eigenvalues} \eqref{eq:lem_eigenvalues:inequality}.
 \end{proof}
 
 Let the orientation on $ \gamma_{\chi,\eta} $ and $ \Gamma_{\chi,\eta} $ be counter clockwise. See Figure \ref{fig:rough:first_deformation} for an example of the curves.

 \begin{proof}[Proof of Theorem \ref{thm:local_rough}]
 As in the calculations of the smooth kernel the formula in Theorem \ref{thm:finite_kernel} is in the right form if $ j_-=1 $, that is if $ \chi>0 $. If $ j_-=2 $ then we rewrite the formula for the kernel slightly. In that case move $ \gamma_1 $ through the cuts and back again, in a similar way as we did in the proof of the smooth kernel, to change $ \rho_1(w) $ to $ \rho_2(w) $. However this time we get a contribution when $ w=z $. We obtain
 \begin{multline}
  \left[K(2km,2\xi+i;2km',2\xi'+j)\right]_{i,j=0}^1 \\
  = \frac{\mathds{1}_{\chi \leq 0}-\mathds{1}_{m>m'}}{2\pi\i}\oint_{\gamma_{0,1}} \Phi(z)^{m-m'}z^{\xi'-\xi}\frac{\d z}{z} \\
  + \frac{(-1)^{\mathds{1}_{\chi\leq 0}}}{(2\pi\i)^2}\oint_{\gamma_1}\oint_{\gamma_{0,1}} \frac{w^{\xi'}}{z^{\xi+1}}\frac{(1-z^{-1})^{\frac{kN}{2}}}{(1-w^{-1})^{\frac{kN}{2}}}\rho_{j_-}(w)^{\frac{N}{2}-m'} \\
  \times E(w)
  \begin{pmatrix}
   \mathds{1}_{\chi>0} & 0 \\
   0 & \mathds{1}_{\chi\leq 0}
  \end{pmatrix}
  E(w)^{-1}\Phi(z)^{m-\frac{N}{2}}\frac{\d z\d w}{z-w}.
 \end{multline}
 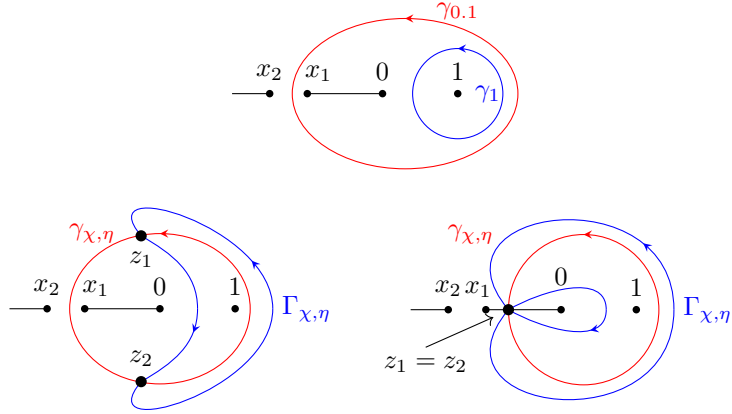
\begin{figure}[t]
 \begin{center}
  \begin{tikzpicture}[scale=1]
  \tikzset{-<-/.style={decoration={markings,mark=at position .25 with {\arrow{stealth}}},postaction={decorate}}}
   \draw (-1,0)--(0,0);
   \draw (-2,0)--(-1.5,0);
   \draw[-<-,blue] (1.,0) circle (.6);
   \draw[-<-,red] (.3,0) ellipse (1.5 and 1);
   \draw (1,0) node[circle,fill,inner sep=1pt,label=above:$1$]{};
   \draw (0,0) node[circle,fill,inner sep=1pt,label=above:$0$]{};
   \draw (-1,0) node[circle,fill,inner sep=1pt]{};
   \draw (-.85,.05) node[above]{$x_1$};
   \draw (-1.5,0) node[circle,fill,inner sep=1pt,label=above:$x_2$]{};
   \draw (1,1.1) node{$\color{red} \gamma_{0.1}$};
   \draw (1.4,0) node{$\color{blue} \gamma_{1}$};
  \end{tikzpicture}
  \\
  \begin{tikzpicture}[scale=1]
  \tikzset{-<-/.style={decoration={markings,mark=at position .25 with {\arrow{stealth}}},postaction={decorate}}}
   \draw (-1,0)--(0,0);
   \draw (-2,0)--(-1.5,0);
   \draw [-<-,blue] (.5,0) to[out=-90,in=30] (-.25,-.97);
   \draw [blue] (-.25,.97) to[out=-30,in=90] (.5,0);
   \draw [blue] (-.25,-.97) to[out=-130,in=-90,distance=1.1cm] (1.5,0);
   \draw [-<-,blue] (1.5,0) to[out=90,in=130,distance=1.1cm] (-.25,.97);
   \draw[-<-,red] (0,0) ellipse (1.2 and 1);
   \draw (1,0) node[circle,fill,inner sep=1pt,label=above:$1$]{};
   \draw (0,0) node[circle,fill,inner sep=1pt,label=above:$0$]{};
   \draw (-1,0) node[circle,fill,inner sep=1pt]{};
   \draw (-.85,.05) node[above]{$x_1$};
   \draw (-1.5,0) node[circle,fill,inner sep=1pt,label=above:$x_2$]{};
   \draw (-.25,.97) node[circle,fill,inner sep=1.5pt,label=below:$z_1$]{};
   \draw (-.25,-.97) node[circle,fill,inner sep=1.5pt,label=above:$z_2$]{};
   \draw (-.9,1) node{$ \color{red}\gamma_{\chi,\eta}$};
   \draw (1.95,0) node{$\color{blue} \Gamma_{\chi,\eta}$};
  \end{tikzpicture}
  \quad
  \begin{tikzpicture}[scale=1]
  \tikzset{-<-/.style={decoration={markings,mark=at position .25 with {\arrow{stealth}}},postaction={decorate}}}
   \draw (-1,0)--(0,0);
   \draw (-2,0)--(-1.5,0);
   \draw [blue] (-.7,0) to[out=30,in=90] (.6,0);
   \draw [-<-,blue] (.6,0) to[out=-90,in=-30] (-.7,0);
   \draw [blue] (-.7,0) to[out=-130,in=-90,distance=1.8cm] (1.5,0);
   \draw [-<-,blue] (1.5,0) to[out=90,in=130,distance=1.8cm] (-.7,0);
   \draw[-<-,red] (.3,0) ellipse (1. and 1);
   \draw (1,0) node[circle,fill,inner sep=1pt,label=above:$1$]{};
   \draw (0,0) node[circle,fill,inner sep=1pt]{};
   \draw (0,.2) node[above]{$0$};
   \draw (-1,0) node[circle,fill,inner sep=1pt]{};
   \draw (-1.1,.05) node[above]{$x_1$};
   \draw (-1.5,0) node[circle,fill,inner sep=1pt,label=above:$x_2$]{};
   \draw (-.7,0) node[circle,fill,inner sep=1.5pt]{};
   \draw [->] (-1.8,-.5) -- (-.9,-.1); 
   \draw (-1.8,-.5) node[below]{$z_1=z_2$};
   \draw (-1.2,1) node{$ \color{red} \gamma_{\chi,\eta}$};
   \draw (1.95,0) node{$ \color{blue} \Gamma_{\chi,\eta}$};
  \end{tikzpicture}
  \caption{Two examples of deformation of the contours in the steepest descent analysis of the correlation kernel at the rough region. \label{fig:rough:first_deformation}}
 \end{center}
\end{figure}

 Take the curve $ \gamma_{0,1} $ as $ \gamma_{\chi,\eta} $ and deform $ \gamma_1 $ to $ \Gamma_{\chi,\eta} $, defined in Lemma \ref{lem:rough:curves}, see Figure \ref{fig:rough:first_deformation}. The only contribution from this deformation comes from the pole at $ w=z $. Let $ \gamma_{z_1} $ be the part of $ \gamma_{\chi,\eta} $ going between $ z_1 $ and $ z_2 $ and intersecting the positive real line, with orientation from $ z_2 $ to $ z_1 $. By definition of $ z_1 $ and \eqref{eq:sec_smooth_rough:j} $ \gamma_{z_1} $ has the same orientation as $ \gamma_{\chi,\eta} $ if $ j_-=1 $ and opposite orientation if $ j_-=2 $. We obtain
 \begin{multline}
  \left[K(2km,2\xi+i;2km',2\xi'+j)\right]_{i,j=0}^1 \\
  = \frac{\mathds{1}_{\chi \leq 0}-\mathds{1}_{\kappa>\kappa'}}{2\pi\i}\oint_{\gamma_{\chi,\eta}} \Phi(z)^{\kappa-\kappa'}z^{\zeta'-\zeta}\frac{\d z}{z} \\
  +\frac{1}{2\pi\i}\int_{\gamma_{z_1}} z^{\zeta'-\zeta}\rho_{j_-}(z)^{\kappa-\kappa'}E(z)
  \begin{pmatrix}
   \mathds{1}_{\chi>0} & 0 \\
   0 & \mathds{1}_{\chi\leq 0}
  \end{pmatrix}
  E(z)^{-1}\frac{\d z}{z} \\
  + \frac{(-1)^{\mathds{1}_{\chi \leq 0}}}{(2\pi\i)^2}\oint_{\Gamma_{\chi,\eta}}\oint_{\gamma_{\chi,\eta}} \frac{w^{\frac{kN}{4}(\eta-1)+e_\eta+\zeta'}}{z^{\frac{kN}{4}(\eta-1)+e_\eta+\zeta+1}}\frac{(1-z^{-1})^{\frac{kN}{2}}}{(1-w^{-1})^{\frac{kN}{2}}}\rho_{j_-}(w)^{-\frac{N}{2}\chi-e_\chi-\kappa'} \\
  \times E(w)
  \begin{pmatrix}
   \mathds{1}_{\chi>0} & 0 \\
   0 & \mathds{1}_{\chi \leq 0}
  \end{pmatrix}
  E(w)^{-1}\Phi(z)^{\frac{N}{2}\chi+e_\chi+\kappa}\frac{\d z\d w}{z-w}.
 \end{multline}
 The sum of the first two term is precisely $ K^{(\chi,\eta)}_{rough} $ and the last term is small. 
 
 To see that the last term is small we proceed as in the smooth case. Divide $ \Phi $ into two terms. The integral of these two terms are
  \begin{multline}
  \frac{(-1)^{\mathds{1}_{\chi\leq 0}}}{(2\pi\i)^2}\oint_{\Gamma_{\chi,\eta}}\oint_{\gamma_{\chi,\eta}} \e^{N(F_{j_-}(w;\chi,\eta)-F_1(z;\chi,\eta))} \frac{w^{e_\eta+\zeta'}}{z^{e_\eta+\zeta+1}}\frac{\rho_1(z)^{e_\chi+\kappa}}{\rho_{j_-}(w)^{e_\chi+\kappa'}} \\
  \times E(w)
  \begin{pmatrix}
   \mathds{1}_{\chi>0} & 0 \\
   0 & \mathds{1}_{\chi \leq 0}
  \end{pmatrix}
  E(w)^{-1}E(z)
  \begin{pmatrix}
   1 & 0 \\
   0 & 0
  \end{pmatrix}
  E(z)^{-1}
  \frac{\d z\d w}{z-w}.
 \end{multline}
 and
  \begin{multline}
  \frac{(-1)^{\mathds{1}_{\chi \leq 0}}}{(2\pi\i)^2}\oint_{\Gamma_{\chi,\eta}}\oint_{\gamma_{\chi,\eta}} \e^{N(F_{j_-}(w;\chi,\eta)-F_2(z;\chi,\eta))} \frac{w^{e_\eta+\zeta'}}{z^{e_\eta+\zeta+1}}\frac{\rho_2(z)^{e_\chi+\kappa}}{\rho_{j_-}(w)^{e_\chi+\kappa'}} \\
  \times E(w)
  \begin{pmatrix}
   \mathds{1}_{\chi>0} & 0 \\
   0 & \mathds{1}_{\chi \leq 0}
  \end{pmatrix}
  E(w)^{-1}E(z)
  \begin{pmatrix}
   0 & 0 \\
   0 & 1
  \end{pmatrix}
  E(z)^{-1}
  \frac{\d z\d w}{z-w}.
 \end{multline} 
 By Lemma \ref{lem:rough:curves} one of the integrals decay exponentially with $ N $, if $ \chi \neq 0 $, and the other, after removing a neighborhood around $ z_1 $ and $ z_2 $, also decay exponentially with $ N $. The parts that are left, the part around $ z_1 $ and $ z_2 $, are of order $ N^{-\frac{1}{2}} $, since the order of the critical points are one. That the decay is $ N^{-\frac{1}{2}} $ is standard and follows by a Taylor expansion of $ F_{j_-} $ in a neighborhood of $ z_i $, $ i=1,2 $, and change of variables. We leave the details to the reader and refer to \cite{OR03}. If $ \chi=0 $ the functions $ F_1 $ and $ F_2 $ coincide, and the decay of both integrals are then $ N^{-\frac{1}{2}} $.
\end{proof}

\section{Proof of Theorem \ref{thm:height_function_expectation} and Corollary \ref{cor:burger}}\label{sec:proof_height_function}
 
 For the proof of Theorem \ref{thm:height_function_expectation} we will reuse many arguments from Section \ref{sec:proof_local_correlations}.

\begin{proof}[Proof of Theorem \ref{thm:height_function_expectation}]
 By definition of $ h $,
 \begin{multline}
  \EE\left[h\left(2km,2\xi\right)\right] = \sum_{y_0\geq2 \xi}\EE\left[\mathds{1}_{(2km,y_0)}\right] \\
  = \sum_{-1\geq\xi_0\geq \xi} \Tr \left[K(2km,2\xi_0+i;2km,2\xi_0+j)\right]_{i,j=0}^1.
 \end{multline}
 From Theorem \ref{thm:finite_kernel} we have the formula for the kernel. We take the sum inside the integrals and use $ \sum_{-1\geq\xi_0\geq \xi}\left(\frac{w}{z}\right)^{\xi_0} = \left(\frac{w}{z}\right)^{\xi}\frac{z}{z-w}-\frac{z}{z-w} $ to obtain
 \begin{multline}
  \EE\left[h\left(2km,2\xi\right)\right]
  =\Tr\left[\frac{1}{(2\pi\i)^2}\oint_{\gamma_1}\oint_{\gamma_{0,1}} \frac{w^{\xi}}{z^{\xi}}\frac{(1-z^{-1})^{\frac{kN}{2}}}{(1-w^{-1})^{\frac{kN}{2}}}\rho_1(w)^{\frac{N}{2}-m}\right. \\
  \left.\times E(w)
  \begin{pmatrix}
   1 & 0 \\
   0 & 0
  \end{pmatrix}
  E(w)^{-1}\Phi(z)^{m-\frac{N}{2}}\frac{\d z\d w}{(z-w)^2}\right] \\
  -\Tr\left[\frac{1}{(2\pi\i)^2}\oint_{\gamma_1}\oint_{\gamma_{0,1}}\frac{(1-z^{-1})^{\frac{kN}{2}}}{(1-w^{-1})^{\frac{kN}{2}}}\rho_1(w)^{\frac{N}{2}-m}\right. \\
  \left.\times E(w)
  \begin{pmatrix}
   1 & 0 \\
   0 & 0
  \end{pmatrix}
  E(w)^{-1}\Phi(z)^{m-\frac{N}{2}}\frac{\d z\d w}{(z-w)^2}\right].
 \end{multline}
 The second term is zero since the integrand with respect to $ z $ is analytic over infinity and hence in the exterior of $ \gamma_{0,1} $. Note that the formula inside the trace in the first term is very similar to the correlation kernel in Theorem \ref{thm:finite_kernel} with $ m'\mapsto m $, $ \xi'\mapsto \xi $ and $ \xi \mapsto \xi-1 $. The only difference is that the exponent of $ (z-w) $ is two instead of one. Taking the limit as $ N $ tends to infinity therefore follows the proof of Theorem \ref{thm:local_smooth} respectively \ref{thm:local_rough} almost word by word. The only major difference is the contribution from the pole at $ z=w $ when deforming the contours. 
 
 Let $ (\chi,\eta) \in \mathcal G_S $. Deform the contours $ \gamma_1 $ to $ \Gamma_{\chi,\eta} $ and $ \gamma_{0,1} $ to $ \gamma_{\chi,\eta} $ given in Lemma \ref{lem:smooth:curves}. We get a contribution from $ z=w $ while the rest of the integral tends to zero as $ N\to\infty $, by the proof of Theorem \ref{thm:local_smooth} with $ (z-w) $ interchanged with $ (z-w)^2 $. The contribution at $ z=w $ is
 \begin{multline}
  \Tr\left[\frac{-1}{2\pi\i}\oint_{\gamma_{\chi,\eta}}\frac{\d}{\d w}\left( \frac{w^{\xi}}{z^{\xi}}\frac{(1-z^{-1})^{\frac{kN}{2}}}{(1-w^{-1})^{\frac{kN}{2}}}\rho_1(w)^{\frac{N}{2}-m}\right.\right. \\
  \left.\left.\left.\times E(w)
  \begin{pmatrix}
   1 & 0 \\
   0 & 0
  \end{pmatrix}
  E(w)^{-1}\Phi(z)^{m-\frac{N}{2}}\right)\right|_{w=z}\d z\right].
 \end{multline}
 By taking the trace inside the integral and using $ \Tr E(z)
 \begin{pmatrix}
  1 & 0 \\
  0 & 0
 \end{pmatrix}
 E(z)^{-1} = 1 $ this becomes
 \begin{multline}
  -\left(\xi+\frac{kN}{2}\right)\frac{1}{2\pi\i}\oint_{\gamma_{\chi,\eta}}\frac{\d z}{z} + \frac{kN}{2}\frac{1}{2\pi\i}\oint_{\gamma_{\chi,\eta}}\frac{\d z}{z-1} \\
  -\left(\frac{N}{2}-m\right)\frac{1}{2\pi\i}\oint_{\gamma_{\chi,\eta}}\frac{\rho_1'(z)}{\rho_1(z)}\d z \\
  - \frac{1}{2\pi\i}\oint_{\gamma_{\chi,\eta}}\rho_1(z)^{\frac{N}{2}-m} \Tr\left[\left(E
  \begin{pmatrix}
   1 & 0 \\
   0 & 0
  \end{pmatrix}
  E^{-1}\right)'(z)\Phi(z)^{m-\frac{N}{2}}\right]\d z.
 \end{multline}
 The sum of the first two term is simply $ -\xi $ since $ \gamma_{\chi,\eta} $ goes around both zero and one. For the last term note first that
 \begin{multline}\label{eq:proof_height_function:product_rule}
  \left(E
  \begin{pmatrix}
   1 & 0 \\
   0 & 0
  \end{pmatrix}
  E^{-1}\right)'(z)
  = \left(\left(E
  \begin{pmatrix}
   1 & 0 \\
   0 & 0
  \end{pmatrix}
  E^{-1}\right)^2\right)'(z) \\
  = \left(E
  \begin{pmatrix}
   1 & 0 \\
   0 & 0
  \end{pmatrix}
  E^{-1}\right)'(z)\left(E
  \begin{pmatrix}
   1 & 0 \\
   0 & 0
  \end{pmatrix}
  E^{-1}\right)(z) \\
  + \left(E
  \begin{pmatrix}
   1 & 0 \\
   0 & 0
  \end{pmatrix}
  E^{-1}\right)(z)\left(E
  \begin{pmatrix}
   1 & 0 \\
   0 & 0
  \end{pmatrix}
  E^{-1}\right)'(z).
 \end{multline}
 By \eqref{eq:proof_height_function:product_rule} and the cyclic property of the trace the last term becomes
 \begin{multline}
  -\frac{2}{2\pi\i}\oint_{\gamma_{\chi,\eta}}\Tr\left[\left(E
  \begin{pmatrix}
   1 & 0 \\
   0 & 0
  \end{pmatrix}
  E^{-1}\right)'(z)\left(E
  \begin{pmatrix}
   1 & 0 \\
   0 & 0
  \end{pmatrix}
  E^{-1}\right)(z)\right]\d z \\
  =-\frac{1}{2\pi\i}\oint_{\gamma_{1,\chi,\eta}}\Tr\left[\left(E
  \begin{pmatrix}
   1 & 0 \\
   0 & 0
  \end{pmatrix}
  E^{-1}\right)'(z)\right]\d z = 0,
 \end{multline}
 where in the first equality \eqref{eq:proof_height_function:product_rule} is used and the last equality follows by changing the order of the derivative and the trace.
 
 To summarize,
 \begin{equation}
  \EE\left[h\left(2km,2\xi\right)\right] = -\xi-\left(\frac{N}{2}-m\right)\frac{1}{2\pi\i}\oint_{\gamma_{\chi,\eta}}\frac{\rho_1'(z)}{\rho_1(z)}\d z +o(1),
 \end{equation}
 as $ N \to \infty $. Hence
 \begin{equation}
  \EE\left[\frac{2}{kN}h^{(N)}(\chi,\eta)\right] = -\frac{1}{2}(\eta-1)+\frac{\chi}{k}\frac{1}{2\pi \i}\oint_{\gamma_{\chi,\eta}}\frac{\rho_1'(z)}{\rho_1(z)}\d z +\Ordo(N^{-1}),
 \end{equation}
 as $ N\to \infty $. 
 
 To finalize the proof of the first statement we prove the equality
 \begin{equation}\label{eq:proof_height_function:argument_principle}
  \frac{1}{2\pi \i}\oint_{\gamma_{\chi,\eta}}\frac{\rho_1'(z)}{\rho_1(z)}\d z = n_\ell-k.
 \end{equation}
 
 Assume first that $ p $, recall \eqref{eq:sec_model:def_p}, has only simple zeros. Deform $ \gamma_{\chi,\eta} $ to $ n_\ell +1 $ simple closed curves $ \gamma_i $ for $ i=0,\dots,n_\ell $ such that $ \gamma_{n_\ell} $ goes around $ 1 $ and no cuts, and $ \gamma_i $, if $ i\neq n_\ell $, goes around the cut $ (x_{2i+1},x_{2i}) $. The contribution from the integral over $ \gamma_{n_\ell} $ is $ -k $ by the argument principle, recall that $ \rho_1 $ has a pole of order $ k $ at $ z=1 $. We show below that the contribution of the integral over $ \gamma_i $ if $ i\neq n_\ell $ is one.
 
 Let $ [0,1]\ni t \mapsto \gamma_i(t) $ be a parametrization of $ \gamma_i $. Then
 \begin{equation}
  \frac{1}{2\pi \i}\oint_{\gamma_i}\frac{\rho_1'(z)}{\rho_1(z)}\d z = \frac{1}{2\pi \i}\int_0^1\frac{\rho_1'(\gamma_i(t))}{\rho_1(\gamma_i(t))}\gamma_i'(t)\d t = \text{Ind}_{\rho_1\circ \gamma_i}(0),
 \end{equation}
 where $ \text{Ind}_{\rho_1\circ \gamma_i}(0) $ is the index of $ 0 $ with respect to $ \rho_1 \circ \gamma_i $, that is, the number of times $ \rho_1 \circ \gamma_i $ wind around zero. By \eqref{eq:sec_eigenvalues:local_eigenvalue} and \eqref{eq:sec_zeros:ppm} it follows that $ \rho_1(x_i) = \pm 1 $ if $ x_i $ is a zero of $ p_\mp $. Along the cut $ (x_{2i+1},x_{2i}) $ the functions $ (p_0^{1/2})_\pm$, the limit of $ p_0^{1/2} $ from the upper respectively lower half plane, take values in $ \i\RR_\pm $ if $ i $ is even and in $ \i\RR_\mp $ if $ i $ odd. We assume for now that $ k $ is odd. By the proof of Proposition \ref{prop:zeros_of_p} both $ x_{2i-1} $ and $ x_{2i} $ are zeros of $ p_+ $ if $ i $ is even and zeros of $ p_- $ if $ i $ is odd. In particular $ \rho_1 $ is $ 1 $ at one end point and $ -1 $ at the other end point of each cut. The image of the curve $ (x_{2i+1},x_{2i}) \ni z \mapsto (\rho_i(z))_+ $ oriented from $ x_{2i} $ to $ x_{2i+1} $ together with the image of the curve $ (x_{2i+1},x_{2i}) \ni z \mapsto (\rho_i(z))_- $ oriented from $ x_{2i+1} $ to $ x_{2i} $ is a simple closed curve going around zero ones counter clockwise. In fact, by Lemma \eqref{lem:eigenvalues} \eqref{eq:lem_eigenvalues:at_cuts} the image is precisely the unit circle. That it goes around zero only ones follows since $ \i \im (\rho_1)_\pm = \frac{(p_0^{1/2})_\pm}{(z-1)^k} $ which has no zeros in $ (x_{2i+1},x_{2i}) $. Hence, by taking $ \gamma_i $ close to the cut $ (x_{2i+1},x_{2i}) $ we see that
 \begin{equation}\label{eq:proof_height_function:integral_cut}
  \frac{1}{2\pi \i}\oint_{\gamma_i}\frac{\rho_1'(z)}{\rho_1(z)}\d z=1.
 \end{equation}
  If instead $ k $ is even, then $ x_{2i-1} $ and $ x_{2i} $ is a zero of $ p_- $ if $ i $ is even and a zero of $ p_+ $ if $ i $ is odd. But going back and forth as above still produces a close curve going around zero ones counter clockwise. So we recover \eqref{eq:proof_height_function:integral_cut}. Summing up all the curves proves \eqref{eq:proof_height_function:argument_principle}.
 
 By a continuity argument \eqref{eq:proof_height_function:argument_principle} still holds if $ p $ has double zeros. Which proves the first statement. 
 
 Let $ (\chi,\eta) \in \mathcal G_R $. Deform $ \gamma_1 $ and $ \gamma_{0,1} $ to $ \Gamma_{\chi,\eta} $ respectively $ \gamma_{\chi,\eta} $ given in Lemma \ref{lem:rough:curves}. As before, but by following the proof of Theorem \ref{thm:local_rough} instead of the proof of Theorem \ref{thm:local_smooth}, the contribution comes from the pole at $ z=w $. Note that the error term now is constant, since the exponent of $ (z-w) $ is two instead of one, this is however enough since we divide the expectation of the height function by a factor $ N $. The contribution at $ z=w $ is
 \begin{multline}
  \Tr\left[\frac{-1}{2\pi\i}\int_{\gamma_{z_1}}\frac{\d}{\d w}\left( \frac{w^{\xi}}{z^{\xi}}\frac{(1-z^{-1})^{\frac{kN}{2}}}{(1-w^{-1})^{\frac{kN}{2}}}\rho_1(w)^{\frac{N}{2}-m}\right.\right. \\
  \left.\left.\left.\times E(w)
  \begin{pmatrix}
   1 & 0 \\
   0 & 0
  \end{pmatrix}
  E(w)^{-1}\Phi(z)^{m-\frac{N}{2}}\right)\right|_{w=z}\d z\right],
 \end{multline}
 where $ \gamma_{z_1} $ is the part of $ \gamma_{\chi,\eta} $ that goes from $ z_2 $ to $ z_1 $ and intersects the positive part of the real line. To simplify this we use the same computations as when $ (\chi,\eta) $ is in a smooth component, which is possible since we barely, in that situation, used the specific curve we integrated over. We get
 \begin{multline}
  \EE\left[h\left(2km,2\xi\right)\right] = -\left(\xi+\frac{kN}{2}\right)\frac{1}{2\pi\i}\int_{\gamma_{z_1}}\frac{\d z}{z} + \frac{kN}{2}\frac{1}{2\pi\i}\int_{\gamma_{z_1}}\frac{\d z}{z-1} \\
  -\left(\frac{N}{2}-m\right)\frac{1}{2\pi\i}\int_{\gamma_{z_1}}\frac{\rho_1'(z)}{\rho_1(z)}\d z +\Ordo(1),
 \end{multline}
 as $ N \to \infty $. Hence
 \begin{multline}
  \EE\left[\frac{2}{kN}h^{(N)}(\chi,\eta)\right]\to-\frac{1}{\pi\i k}\int_{\gamma_{z_1}}F_1'(z;\chi,\eta)\d z \\ 
  = -\frac{1}{\pi\i k}\int_{\gamma_{z_1}}\frac{\d}{\d z}\left(F(z,p_0(z)^\frac{1}{2};\chi,\eta)\right)\d z
 \end{multline}
  as $ N\to \infty $.
\end{proof}

\begin{proof}[Proof of Corollary \ref{cor:burger}]
 Since
 \begin{equation}
  -\frac{1}{\pi\i k}\int_{\gamma_{z_1}}\frac{\d}{\d z}\left(F(z,p_0(z)^\frac{1}{2};\chi,\eta)\right)\d z = -\frac{1}{\pi\i k}\left(F_1(z_1;\chi,\eta)-F_1(\bar z_1;\chi,\eta)\right),
 \end{equation}
 it follows, by differentiating the right hand side and using that $ F_1'(z_i;\chi,\eta)=0 $ for $ i=1,2 $, that
 \begin{equation}
  \nabla \EE\left[h^{(\infty)}(kx,2y)\right] = \frac{1}{\pi}\left(\im \log \rho_1(L_1(\chi,\eta)),-\im \log L_1(\chi,\eta)\right),
 \end{equation}
 which proves the first statement. The other statement follows by a direct verification using
 \begin{equation}
  \frac{\partial L_1}{\partial \chi} = -\frac{2}{k}\frac{L_1\rho_1'(L_1)}{\rho_1(L_1)}\frac{\partial L_1}{\partial \eta},
 \end{equation}
 which follows by differentiating the equation $ F_1'(L_1(\chi,\eta);\chi,\eta)=0 $ ones with respect to $ \chi $ and ones with respect to $ \eta $ and recalling that the zero is simple. The last equality is true since $ \rho_1 $ is an eigenvalue of $ \Phi $.
\end{proof}

\section{Proof of Theorem \ref{thm:finite_kernel}}\label{sec:proof_main_theorem}

The proof of Theorem \ref{thm:finite_kernel} is based on \cite[Theorem 3.1]{BD19} and also the method discussed in the same paper how to explicitly obtain the associated Wiener-Hopf factorization. Loosely speaking \cite[Theorem 3.1]{BD19} states the following. Consider a determinantal point process defined by a probability measure of the form \eqref{eq:sec_model:measure} with $ \phi=\prod \phi_m $. Assume $ \phi $ admits a factorization $ \phi=\phi_+\phi_-=\tilde \phi_-\tilde \phi_+ $ where $ \phi_+,\tilde \phi_+ $ and $ \phi_-\tilde \phi_- $ are analytic and non-singular inside respectively outside the unit circle and $ \phi_-, \tilde \phi_- $ behaves as $ z^{M}I $ as $ z\to \infty $. Then the correlation kernel has a limit as $ n \to \infty $ and the limiting kernel is given by
\begin{multline}	
\left[ K(m, p\xi +j; m', p \xi'+i) \right]_{i,j=0}^{p-1} = - \frac{\mathds{1}_{m > m'}}{2\pi \i} \oint_{|z|=1} \prod_{j=m+1}^{m'}\phi_{j}(z) z^{\xi'-\xi} \frac{\d z}{z} \\
+  \frac{1}{(2\pi \i)^2} \oiint_{|w|<|z|} \prod_{j=m'+1}^{2kN}\phi_{j}(w) \widetilde \phi^{-1}_+(w)\widetilde \phi^{-1}_-(z) \prod_{j=1}^m\phi_{j}(z)  \frac{w^{\xi'-M}\d z \d w}{z^{\xi-M+1} (z-w)},
\end{multline}
where $ p $ is the periodicity of the weighting in the vertical direction and $ pM $ is the height difference of the start and endpoints. In our case $ p=2 $ and $ pM = -kN $.
 
\begin{proof}[Proof of Theorem \ref{thm:finite_kernel}]
As mentioned the proof is based on the technique developed in \cite{BD19}, where the argument is thoroughly explained. For a more detailed description of the argument we therefore refer to \cite{BD19} and in particular the proof of Theorem 5.2 which is very similar to the current proof. 

With a Wiener-Hopf type factorization of $ \phi(z) = \Phi(z)^N $, \cite[Theorem 3.1]{BD19} applies directly. However such factorization is not possible, since $ \phi $ is singular and has singularities on the unit circle. We therefore introduce an extra parameter $ 0<a<1 $ in the model, which we later take to $ 1 $. That is, consider the measure defined by \eqref{eq:sec_model:measure} but with $ \phi_{a,m} $ instead of $ \phi_m $ where
\begin{multline}
\phi_{a,2m-1}(z)  =  
\begin{pmatrix}
\gamma_m & a^{-1}\alpha_m z^{-1}\\
a^{-1}\alpha_m^{-1} & \gamma_m^{-1}
\end{pmatrix}, \\
\phi_{a,2m}(z) = \frac{1}{1-a^2z^{-1}} 
\begin{pmatrix}
1 & a\beta_m z^{-1}\\
a\beta_m^{-1} & 1
\end{pmatrix}, 
\end{multline}
for $ m=1,2,\dots,kN $.

To obtain the factorization of $ \phi_a=\prod_{m=1}^{kN}\phi_{a,2m-1}\phi_{a,2m} $ we use the method discussed in \cite[Section 4]{BD19}. That is, we move all the $ \phi_{a,2m-1} $ to one side and the $ \phi_{a,2m} $ to the other side, for $ m=1,\dots,kN $ using a switching rule. The switching rule we use here is slightly different from the one proposed in \cite{BD19}, but the argument in \cite[Section 4]{BD19} still works in the same way. What we use is the algebraic fact that
\begin{equation}\label{eq:proof_main_theorem:switching}
 \begin{pmatrix}
  a & b z^{-1} \\
  c & a^{-1} 
 \end{pmatrix}
 \begin{pmatrix}
  \alpha & \beta z^{-1} \\
  \gamma & \alpha^{-1} 
 \end{pmatrix}
 =
  \begin{pmatrix}
  a & \gamma x z^{-1} \\
  \beta x^{-1} & a^{-1} 
 \end{pmatrix}
 \begin{pmatrix}
  \alpha & c x z^{-1} \\
  bx^{-1} & \alpha^{-1} 
 \end{pmatrix},
\end{equation}
where $ x = \frac{a\beta+b\alpha^{-1}}{a c+\alpha^{-1}\gamma} $. What is important here is that the determinant of the first matrix on the left hand side is equal to the determinant of the second matrix in the right hand side, and vice versa. Note also that in the case $ c=b^{-1} $ and $ \gamma=\beta^{-1} $ the equality \eqref{eq:proof_main_theorem:switching} becomes the trivial equality
\begin{equation}\label{eq:proof_main_theorem:trivial_switching}
 \begin{pmatrix}
  a & b z^{-1} \\
  b^{-1} & a^{-1} 
 \end{pmatrix}
 \begin{pmatrix}
  \alpha & \beta z^{-1} \\
  \beta^{-1} & \alpha^{-1} 
 \end{pmatrix}
 =
 \begin{pmatrix}
  a & b z^{-1} \\
  b^{-1} & a^{-1} 
 \end{pmatrix}
 \begin{pmatrix}
  \alpha & \beta z^{-1} \\
  \beta^{-1} & \alpha^{-1} 
 \end{pmatrix}.
\end{equation}
Applying \eqref{eq:proof_main_theorem:switching} to $ \phi_{a,2m-1}$ and $ \phi_{a,2m}  $ we get $ \phi_{a,2m-1}\phi_{a,2m} = \tilde \phi_{a,2m}\tilde\phi_{a,2m-1} $, where $ \tilde \phi_{a,j} $ is analytic and non-singular exactly where $ \phi_{a,j} $ is analytic and non-singular. We use this switching rule pairwise on the factors in $ \phi_a $. We do it $ kN $ times, at $ i$:th time we switch $ kN+1-i $ pairs, by leaving the first and last $ i-1 $ factors untouched. For a more thorough explanation see \cite[Section 4]{BD19}. Doing this two times, first taking the $ \phi_{a,m} $ with $ m $ even to the left and the second time to the right, we get
\begin{equation}
 \phi_a = \prod_{m=1,\text{even}}^{2kN}\phi_{a,m}'\prod_{m=1,\text{odd}}^{2kN}\phi_{a,m}'=\prod_{m=1,\text{odd}}^{2kN}\phi_{a,m}''\prod_{m=1,\text{even}}^{2kN}\phi_{a,m}''.
\end{equation}
Here $ \phi_{a,m} $, $ \phi_{a,m}' $ and $ \phi_{a,m}'' $ are analytic and non-singular simultaneously.

Set
\begin{equation}
 \phi_{a,+}(z) = z^\frac{kN}{2}\prod_{m=1,\text{odd}}^{2kN}\phi''_{a,m}(z)C, \quad  \phi_{a,-}(z) = C^{-1}z^{-\frac{kN}{2}}\prod_{m=1,\text{even}}^{2kN}\phi''_{a,m}(z),
\end{equation}
and
\begin{equation}
 \widetilde \phi_{a,+}(z) = \widetilde Cz^\frac{kN}{2}\prod_{m=1,\text{odd}}^{2kN}\phi'_{a,m}(z), \quad  \widetilde \phi_{a,-}(z) = z^{-\frac{kN}{2}}\prod_{m=1,\text{even}}^{2kN}\phi'_{a,m}(z)\widetilde C^{-1}.
\end{equation}
Here $ C $ and $ \widetilde C $ are normalizing factors so $ \phi_{a,-}(z),\widetilde \phi_{a,-}(z) \sim z^{-\frac{kN}{2}}I $ as $ z \to \infty $. Since $ \phi'_{a,m} $ and $ \phi''_{a,m} $ has the same form as $ \phi_{a,m} $ we see that $ \phi_{a,+} $ and $ \widetilde \phi_{a,+} $ are analytic and non-singular in $ \DD $ and continuous up to the boundary, and $ \phi_{a,-} $ and $ \widetilde \phi_{a,-} $ are analytic and non-singular  in $ \overline \DD^c $ and continuous up to the boundary. Moreover 
	\begin{equation}
	 \phi_{a,+}\phi_{a,-} = \widetilde \phi_{a,-}\widetilde \phi_{a,+} = \phi_a = \Phi_a^N,
	\end{equation}
	where
	\begin{equation}
	 \Phi_a = \prod_{i=1}^{2k}\phi_{a,m}.
	\end{equation}
    This is what we need for \cite[Theorem 3.1]{BD19}. However the construction also implies a few more properties we need later on,
	\begin{enumerate}[(i)]
	 \item $ \widetilde \phi_{a,-} $ is analytic in $ \CC\backslash \{a^2\} $, \label{eq:proof_main_theorem:prop_factorization_1}
	 \item $ \widetilde \phi_{a,+}^{-1} $ is analytic in $ \CC\backslash \{a^{-2}\} $,\label{eq:proof_main_theorem:prop_factorization_2}
	 \item \label{eq:proof_main_theorem:prop_factorization_3}
	 and
	 \begin{equation} 
	  \prod_{m=1,\text{even}}^{2kN}\phi'_{a,m}(z) \to (1-z^{-1})^{-\frac{kN}{2}}\Phi(z)^\frac{N}{2},
	 \end{equation}
	 as $ a \to 1 $.
	\end{enumerate}
	Equation \eqref{eq:proof_main_theorem:prop_factorization_1} and \eqref{eq:proof_main_theorem:prop_factorization_2} follows by construction of the factors and since the correspondent statement is true for $ \phi_{a,m} $, except at $ z=0 $, where it still is true for the product due to the factor $ z^\frac{kN}{2} $. For \eqref{eq:proof_main_theorem:prop_factorization_3}, note that the observation \eqref{eq:proof_main_theorem:trivial_switching} implies
\begin{equation} 
 \prod_{m=1,\text{even}}^{2kN}\phi'_{1,m}(z) = (1-z^{-1})^{-\frac{kN}{2}}\Phi(z)^\frac{N}{2}.
\end{equation}
So the limit follows since the switching rule is continuous in $ a $.
	
With the above factorization we use \cite[Theorem 3.1]{BD19}. That is, the model defined above converges to a determinantal point process as $ n \to \infty $. We are interested in the part $ \xi,\xi'\geq -\frac{kN}{2} $. Then the correlation kernel is given by 
	\begin{multline}\label{eq:proof_main_theorem:a_kernel}
	\left[K_{top}^{(a)}(2km,2\xi+i;2km',2\xi'+j)\right]_{i,j=0}^1 \\
	= -\frac{\mathds{1}_{m>m'}}{2\pi\i}\oint_{\gamma_{0,1}} \Phi_a(z)^{m-m'}z^{\xi'-\xi}\frac{\d z}{z} \\
	+ \frac{1}{(2\pi\i)^2}\oint_{\gamma_a}\oint_{\gamma_{0,1,a}} \frac{w^{\xi'}}{z^{\xi+1}} \Phi_a(w)^{N-m'} \left(\prod_{k=1,\text{odd}}^{2kN}\phi'_{a,m}(w)\right)^{-1} \\
	\times\left(\prod_{k=1,\text{even}}^{2kN}\phi'_{a,m}(z)\right)^{-1}\Phi_a(z)^m\frac{\d z\d w}{z-w},
	\end{multline}
	where $ \gamma_a $ is a simple closed curve with $ a^2 $ in the interior and $ a^{-2} $ in the exterior, and $ \gamma_{0,1,a} $ is a simple closed curve with $ 0 $, $ 1 $ and $ \gamma_a $ in the interior. That $ \xi'\geq -\frac{kN}{2} $ is required in order to move the curve of the first integrand over zero to $ \gamma_a $. We would like to take $ a \to 1 $. The problem however is that the integrand with respect to $ w $ is singular both at $ a^2 $ and $ a^{-2} $ which lies on different sides of $ \gamma_a $, see Figure \ref{fig:contours}, which complicates the limit procedure. To solve this problem we go to the eigenvalues.
	
	\begin{figure}[t]
	 \begin{center}
	  \begin{tikzpicture}[scale=1]
	   \tikzset{-<-/.style={decoration={markings,mark=at position .25 with {\arrow{stealth}}},postaction={decorate}}}
	   \draw[-<-,blue] (.9,0) circle (.7);
	   \draw[-<-,red] (0,0) ellipse (3 and 1);
	   \draw (.55,0) node[circle,fill,inner sep=1pt,label=below:$a^2$]{};
	   \draw (1,0) node[circle,fill,inner sep=1pt,label=above:$1$]{};
	   \draw (1.818,0) node[circle,fill,inner sep=1pt,label=below:$a^{-2}$]{};
	   \draw (0,0) node[circle,fill,inner sep=1pt,label=above:$0$]{};
	   \draw (1.5,.6) node{$ \gamma_a$};
	   \draw (3,.6) node{$ \gamma_{0,1,a}$};
	  \end{tikzpicture}
	 \end{center}
	\caption{The contours of integration. The point $ a^2 $ is a pole for $ \rho_{a,1} $ and $ a^{-2} $ is a zero for $ \rho_{a,2} $.
	\label{fig:contours}}
	\end{figure}
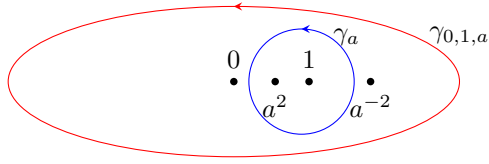
	
	Let $ \Phi_a(w) = E_a(w)\Lambda_a(w)E_a(w)^{-1} $ be an eigenvalue decomposition of $ \Phi_a $ with eigenvalues $ \rho_{a,1} $ and $ \rho_{a,2} $.

	\emph{Claim:}
	There is a neighborhood $ \mathcal B $ of $ z=1 $ containing $ a^2 $ and $ a^{-2} $ such that the eigenvalue $ \rho_{a,1} $ is analytic and non-zero in $ \mathcal B $ except at $ a^2 $. Moreover we can take $ \gamma_a $ in $ \mathcal B $ in such a way that $ \rho_{a,2} $ is analytic on and in the interior of $ \gamma_a $.
	
	We assume the claim for now and prove it later.	Use the eigenvalue decomposition  to write $ \Phi_a^{N-m'} $ as a sum,
	\begin{multline}
	\Phi_a(w)^{N-m'} = \rho_{a,1}(w)^{N-m'}E_a(w)
	\begin{pmatrix}
	1 & 0 \\
	0 & 0
	\end{pmatrix}
	E_a(w)^{-1} \\
	+ \rho_{a,2}(w)^{N-m'}E_a(w)
	\begin{pmatrix}
	0 & 0 \\
	0 & 1
	\end{pmatrix}
	E_a(w)^{-1},
	\end{multline}
	and use this in \eqref{eq:proof_main_theorem:a_kernel}. By \eqref{eq:proof_main_theorem:prop_factorization_2} and the claim, the term with $ \rho_{a,2} $ vanishes, since the integrand with respect to $ w $ is analytic inside $ \gamma_a $. For the second term, the one with $ \rho_{a,1} $, we write the integrand with respect to $ w $, recall that $ \Phi_a^N\tilde \phi_{a,+}^{-1} = \tilde \phi_{a,-} $, as
	\begin{multline}
	\rho_{a,1}(w)^{N-m'}E_a(w)
	\begin{pmatrix}
	1 & 0 \\
	0 & 0
	\end{pmatrix}
	E_a(w)^{-1}
	\left(\prod_{k=1,\text{odd}}^{2kN}\phi'_{a,m}(w)\right)^{-1} = \\
	\rho_{a,1}(w)^{-m'}E_a(w)
	\begin{pmatrix}
	1 & 0 \\
	0 & 0
	\end{pmatrix}
	E_a(w)^{-1}
	\prod_{k=1,\text{even}}^{2kN}\phi'_{a,m}(w),
	\end{multline}
	which, by \eqref{eq:proof_main_theorem:prop_factorization_1} and the claim is analytic over $ a^{-2} $. Move the contour $ \gamma_a $ in \eqref{eq:proof_main_theorem:a_kernel} to a contour $ \gamma_1 $ containing $ a^2 $, $1$ and $ a^{-2}$ in the interior. By analyticity of the kernel we may, and we do, take $ a \to 1 $. That is
	\begin{multline}
	\left[K_{top}^{(a)}(2km,2\xi+i;2km',2\xi'+j)\right]_{i,j=0}^1 \\
	= -\frac{\mathds{1}_{m>m'}}{2\pi\i}\oint_{\gamma_{0,1}} \Phi_a(z)^{m-m'}z^{\xi'-\xi}\frac{\d z}{z} \\
	+ \frac{1}{(2\pi\i)^2}\oint_{\gamma_1}\oint_{\gamma_{0,1,a}} \frac{w^{\xi'}}{z^{\xi+1}}\rho_{a,1}(w)^{-m'}E_a(w)
	\begin{pmatrix}
	1 & 0 \\
	0 & 0
	\end{pmatrix}
	E_a(w)^{-1}
	\prod_{k=1,\text{even}}^{2kN}\phi'_{a,m}(w) \\
	\times\left(\prod_{k=1,\text{even}}^{2kN}\phi'_{a,m}(z)\right)^{-1}\Phi_a(z)^m\frac{\d z\d w}{z-w}, \\
	\to -\frac{\mathds{1}_{m>m'}}{2\pi\i}\oint_{\gamma_{0,1}} \Phi(z)^{m-m'}z^{\xi'-\xi}\frac{\d z}{z} \\
	+ \frac{1}{(2\pi\i)^2}\oint_{\gamma_1}\oint_{\gamma_{0,1}} \frac{w^{\xi'}}{z^{\xi+1}}\frac{(1-z^{-1})^\frac{kN}{2}}{(1-w^{-1})^\frac{kN}{2}}\rho_1(w)^{-m'} E(w)
	\begin{pmatrix}
	1 & 0 \\
	0 & 0
	\end{pmatrix}
	E(w)^{-1}
	\Phi(w)^{\frac{N}{2}} \\
	\times\Phi(z)^{-\frac{N}{2}}\Phi(z)^m\frac{\d z\d w}{z-w},
	\end{multline}
	as $ a \to 1 $, where we used \eqref{eq:proof_main_theorem:prop_factorization_3} and that $ \Phi_a, \rho_{a,1}, E_a $ tends to $ \Phi, \rho_{1}, E $ respectively as $ a  \to 1 $. 

	What is left to show is the claim. Note first that
	\begin{equation}\label{eq:proof_main_theorem:determinant}
	 \rho_{a,1}(z)\rho_{a,2}(z) = \det \Phi_a(z) = \left(\frac{1-a^{-2}z^{-1}}{1-a^2z^{-1}}\right)^k.
	\end{equation}
	We can express the eigenvalues explicitly as
	\begin{equation}
 	 \rho_{a,1}(z) = \frac{1}{2}\frac{\Tr \left((1-a^2z^{-1})^k\Phi_a(z)\right)+p_a(z)^\frac{1}{2}}{(1-a^2z^{-1})^k}
 	\end{equation}
 	and
	\begin{equation}
 	 \rho_{a,2}(z) = \frac{1}{2}\frac{\Tr \left((1-a^2z^{-1})^k\Phi_a(z)\right)-p_a(z)^\frac{1}{2}}{(1-a^2z^{-1})^k},
 	\end{equation}
 	where
 	\begin{equation}
 	 p_a(z) = \Tr \left((1-a^2z^{-1})^k\Phi_a(z)\right)^2-4(1-a^2z^{-1})^k(1-a^{-2}z^{-1})^k.
 	\end{equation}
 	This formula does not look to appealing, but it is sufficient for our purposes. In fact, as long as we are away from the branch coming from the square root and the possible pole at $ z=a^2 $ both $ \rho_{a,1} $ and $ \rho_{a,2} $ are analytic. 
 	
 	For $ z>0 $ we get a lower bound on the trace by ignoring the part of the anti-diagonal,
 	\begin{equation}
 	 \Tr\left( (1-a^2z^{-1})^k\Phi_a(z)\right) \geq \Tr \prod_{m=1}^{k}
 	 \begin{pmatrix}
 	  \gamma_m & 0 \\
 	  0 & \gamma_m^{-1}
 	 \end{pmatrix}
 	 = \gamma+\gamma^{-1}\geq 2,
 	\end{equation}
 	where $ \gamma = \prod \gamma_m $. So for $ a $ close enough to $ 1 $, $ p_a(1)>0 $ which tells us there is a neighborhood $ \mathcal B $ of $ z=1 $ so that for $ a $ close enough to one $ p_a(z)>0 $ for $ z \in \mathcal B \cap \RR $. That is, we can take $ \mathcal B $, by possible making it smaller, such that $ p_a^\frac{1}{2} $ is analytic in $ \mathcal B $ and $ p_a(z)^\frac{1}{2}>0$ for $ \mathcal B \cap \RR $. Take $ a $ so close to one that $ \mathcal B $ contains $ a^2 $ and $ a^{-2} $. 
	 
	  Now, by the above, both terms in the numerator of $ \rho_{a,1} $ is strictly positive for real $ z\in \mathcal B $ and in particular non-zero at $ z=a^{-2} $ and $ z=a^2 $. By \eqref{eq:proof_main_theorem:determinant} $ \rho_{a,1} $ is non-zero in $ \mathcal B $. We conclude that $ \rho_{a,1} $ is analytic and non-zero in $ \mathcal B $ except at $ a^2 $, where it has a pole of order $ k $. By \eqref{eq:proof_main_theorem:determinant} $ \rho_{a,2} $ is analytic in all of $ \mathcal B $, also at $ a^2 $, so it is enough to take $ \gamma_a $ inside $ \mathcal B $. This proves the claim and hence the theorem.
\end{proof}

 \appendix
 
 \section{Gauge transformation of the weighting}
 
 We give the relation between the weights on the edges and the weights on the non-intersecting paths. This relation goes through a gauge transformation.
 
 For $ i=1,\dots,k $ set
\begin{multline}
 \alpha_i=\frac{a_{i-1,kN+i-3}}{a_{i,kN+i-3}}\frac{a_{i-1,kN+i-2}}{a_{i-2,kN+i-2}}, \quad \beta_i=\frac{a_{i-1,kN+i-2}}{a_{i-1,kN+i-1}}\frac{a_{i,kN+i-2}}{a_{i,kN+i-3}}, \\
 \gamma_i=\frac{a_{i,kN+i-4}}{a_{i+1,kN+i-4}}\frac{a_{i,kN+i-3}}{a_{i-1,kN+i-3}} \text{ and } \delta_i=\frac{a_{i,kN+i-3}}{a_{i,kN+i-2}}\frac{a_{i+1,kN+i-3}}{a_{i+1,kN+i-4}}.
\end{multline}
The weighting of the $ 2\times k $-periodic Aztec diamond, defined in Sections \ref{sec:aztec_diamond}, is equivalent to the weighting
\begin{multline}
 \alpha_i \text{ on the edge } \{(i-1,kN+i-2),(i,kN+i-2)\}+2n(1,-1)+kn'(1,1), \\ 
 \beta_i \text{ on the edge } \{(i,kN+i-2),(i,kN+i-1)\}+2n(1,-1)+kn'(1,1), \\
 \gamma_i \text{ on the edge } \{(i,kN+i-3),(i+1,kN+i-3)\}+2n(1,-1)+kn'(1,1), \\ 
 \delta_i \text{ on the edge } \{(i+1,kN+i-3),(i+1,kN+i-2)\}+2n(1,-1)+kn'(1,1),
\end{multline}
for $ i=1,\dots,k $ and all $ n,n'\in \ZZ $ and with the weight one on all other edges. This follows by multiply each edge connected to the vertex $ (i,j) $ with $ a_{i-1,j}^{-1} $ if $ (i,j) $ is a black vertex and with $ a_{i,j-1}^{-1} $ if $ (i,j) $ is a white vertex.   

Note that
\begin{equation}
 \alpha_i\gamma_i = \frac{a_{i-1,kN+i-2}}{a_{i+1,kN+i-4}}\frac{a_{i,kN+i-4}}{a_{i-2,kN+i-2}} = 1,
\end{equation}
and similarly $ \beta_i\delta_i=1 $ for $ i=1,\dots,k $, by the periodicity condition on $ a_{i,j} $. Note also that
\begin{equation}
 \alpha_1\cdots \alpha_k=\frac{a_{0,kN-2}}{a_{-1,kN-1}}\beta_1\cdots\beta_{k-1}\frac{a_{k-1,kN-k-2}}{a_{k,kN+k-3}} = \beta_1\cdots\beta_k.
\end{equation}

\end{document}